\documentclass[reqno,11pt,final]{amsart}

\usepackage[backend=biber,
  style=alphabetic, 
  giveninits=true, 
  doi=false, 
  url=false, 
  maxbibnames=99, 
  maxcitenames=99,
  maxalphanames=99,
  sorting=nyt
]{biblatex}

\AtEveryBibitem{
  \clearfield{issn}
  \clearfield{isbn}
}

\DeclareFieldFormat[article]{title}{\mkbibemph{#1}}

\DeclareFieldFormat[inproceedings]{title}{\mkbibemph{#1}}

\DeclareFieldFormat[incollection]{title}{\mkbibemph{#1}}

\DeclareFieldFormat{booktitle}{#1}

\DeclareFieldFormat{journaltitle}{#1}

 \DeclareFieldFormat{title}{\mkbibemph{#1}\addcomma\space}

\DeclareBibliographyDriver{book}{%
  \printnames{author}%
  \setunit{\addcomma\space}%
  \printfield[title]{title}%
  \setunit{\addcomma\space}%
  \printfield{series}%
  \setunit{\space}%
  \printfield{volume}%
  \setunit{\addperiod\space}%
  \printlist{publisher}%
  \setunit{\addcomma\space}%
  \printfield{year}%
  \setunit{\addcomma\space}%
  \printfield{pages}%
  \finentry
}

\renewbibmacro*{volume+number+eid}{%
  \mkbibbold{\printfield{volume}}%
  \setunit{\addcomma\space}%
  \printtext{no.\space\printfield{number}}%
}

\DeclareFieldFormat{pages}{#1}

\DeclareBibliographyDriver{article}{%
  \printnames{author}%
  \setunit{\addcomma\space}%
  \printfield[title]{title}%
  \setunit{\addcomma\space}%
  \printfield{journaltitle}%
  \setunit{\addspace}%
  \usebibmacro{volume+number+eid}%
  \setunit{\addspace}%
  \printtext{\mkbibparens{\printfield{year}}}%
  \setunit{\addcomma\space}%
  \printfield{pages}%
  \finentry
}

\usepackage{rotating}
\usepackage{xcolor}
\usepackage{transparent}
\usepackage{amsmath,amssymb,amscd}
\usepackage[mathscr]{eucal}
\usepackage{mathtools}
\usepackage{url}
\usepackage[colorlinks = true,linkcolor = blue,urlcolor  = blue,citecolor = blue,anchorcolor = blue,linktocpage=true,
]{hyperref}
\usepackage{amsthm}

\usepackage{stmaryrd}
\usepackage{bm}
\usepackage{comment}

\allowdisplaybreaks[2]

\usepackage[inline]{enumitem}
\setlist[enumerate,1]{label=(\arabic*),ref=\arabic*}
\setlist[itemize,1]{itemsep=0pt}
\usepackage{setspace}
\usepackage[math]{cellspace}
\usepackage{tikz}
\usetikzlibrary{positioning}
\usetikzlibrary{cd}
\usetikzlibrary{calc}
\usetikzlibrary{shapes.geometric}
\usetikzlibrary{decorations.pathreplacing,decorations.markings}
\usetikzlibrary{backgrounds}
\usetikzlibrary{arrows.meta}

\tikzset{
  morphism/.style={
    commutative diagrams/.cd,
    every arrow,
    every label
  }
}

\allowdisplaybreaks

\definecolor{mygreen}{rgb}{0,0.7,0.3}
\definecolor{myblue}{rgb}{0,0.50,1.20}
\definecolor{myorange}{rgb}{1,0.5,0.1}

\definecolor{fillred}{rgb}{1,0.9,0.9}
\definecolor{fillgreen}{rgb}{0.9,1,0.9}

\usepackage[inline]{showlabels}
\definecolor{refkey}{rgb}{0,0.7,0.3}
\definecolor{labelkey}{rgb}{1,0,0}

\makeatletter\showlabels{blx@bibitem}\makeatother

\DeclareMathOperator{\Ad}{Ad}

\DeclareMathOperator{\id}{id}
\DeclareMathOperator{\Hom}{Hom}
\DeclareMathOperator{\Iso}{Iso}
\DeclareMathOperator{\End}{End}
\DeclareMathOperator{\Equiv}{Equiv}

\DeclareMathOperator{\rank}{rank}
\DeclareMathOperator{\Det}{\mathsf{Det}}
\DeclareMathOperator{\InvStr}{\mathsf{Inv}}
\DeclareMathOperator{\hol}{\mathsf{geom}}
\DeclareMathOperator{\mon}{mon}
\DeclareMathOperator{\GL}{GL}
\DeclareMathOperator{\SL}{SL}
\DeclareMathOperator{\PSL}{PSL}
\DeclareMathOperator{\PGL}{PGL}

\DeclareMathOperator{\GSp}{GSp}
\DeclareMathOperator{\PGSp}{PGSp}
\DeclareMathOperator{\Sp}{Sp}
\DeclareMathOperator{\Spec}{Spec}
\DeclareMathOperator{\Rep}{Rep}
\DeclareMathOperator{\Img}{Im}
\DeclareMathOperator{\Ker}{Ker}

\DeclareMathOperator{\Loc}{Loc}
\DeclareMathOperator{\interior}{int}
\DeclareMathOperator{\Lie}{Lie}
\DeclareMathOperator{\Obj}{Obj}
\DeclareMathOperator*{\colim}{colim}
\DeclareMathOperator{\Fin}{Fin}
\DeclareMathOperator{\lin}{lin}
\DeclareMathOperator{\Cone}{Cone}

\DeclareMathOperator{\torsion}{\mathsf{tor}}
\DeclareMathOperator{\sign}{sign}
\DeclareMathOperator{\parity}{\mathsf{parity}}
\DeclareMathOperator{\betti}{betti}

\numberwithin{equation}{section}

\usepackage{thmtools}
\usepackage{cleveref}

\declaretheorem[name=Theorem,numberwithin=section]{thm}
\declaretheorem[sharenumber=thm,name=Theorem]{theorem}

\declaretheorem[sharenumber=thm,name=Corollary]{cor}
\declaretheorem[sharenumber=thm,name=Corollary]{corollary}

\declaretheorem[sharenumber=thm,name=Lemma]{lemma}

\declaretheorem[sharenumber=thm,name=Proposition]{prop}

\declaretheorem[name=Theorem]{introthm}

\theoremstyle{definition}

\declaretheorem[sharenumber=thm,name=Definition]{dfn}
\declaretheorem[sharenumber=thm,name=Definition]{definition}

\declaretheorem[sharenumber=thm,name=Example]{example}

\declaretheorem[sharenumber=thm,name=Remark]{remark}

\declaretheorem[sharenumber=introthm,name=Example]{introex}

\crefname{thm}{Theorem}{Theorems}
\crefname{theorem}{Theorem}{Theorems}

\crefname{cor}{Corollary}{Corollaries}
\crefname{corollary}{Corollary}{Corollaries}

\crefname{lem}{Lemma}{Lemmas}
\crefname{lemma}{Lemma}{Lemmas}

\crefname{prop}{Proposition}{Propositions}
\crefname{proposition}{Proposition}{Propositions}

\crefname{introthm}{Theorem}{Theorems}
\crefname{introconj}{Conjecture}{Conjectures}
\crefname{introcor}{Corollary}{Corollaries}
\crefname{introex}{Example}{Examples}

\crefname{dfn}{Definition}{Definitions}
\crefname{definition}{Definition}{Definitions}

\crefname{ex}{Example}{Examples}
\crefname{example}{Example}{Examples}

\crefname{claim}{Claim}{Claims}

\crefname{conj}{Conjecture}{Conjectures}
\crefname{conjecture}{Conjecture}{Conjectures}

\crefname{rem}{Remark}{Remarks}
\crefname{remark}{Remark}{Remarks}

\crefname{conv}{Notation}{Notations}
\crefname{notation}{Notation}{Notations}

\crefname{assum}{Assumption}{Assumptions}
\crefname{assumption}{Assumption}{Assumptions}

\crefname{lemdef}{Lemma-Definition}{Lemma-Definitions}

\crefname{figure}{Figure}{Figures}
\crefname{section}{Section}{Sections}
\crefname{subsection}{Section}{Sections}
\crefname{appendix}{Appendix}{Appendices}

\newcommand{\fgl}{\mathfrak{gl}}
\newcommand{\fsp}{\mathfrak{sp}}
\newcommand{\fgsp}{\mathfrak{gsp}}
\newcommand{\fpgsp}{\mathfrak{pgsp}}

\newcommand\X{{\mathcal{X} }}

\newcommand{\cO}{\mathcal{O}}

\newcommand{\bC}{\mathbb{C}}

\newcommand{\bQ}{\mathbb{Q}}

\newcommand{\bZ}{\mathbb{Z}}

\newcommand{\reg}[1]{#1\text{-}\mathrm{reg}}
\newcommand{\Alg}{\mathsf{Alg}}
\newcommand{\Set}{\mathsf{Set}}
\newcommand{\Grpd}{\mathsf{Grpd}}

\newcommand{\iHom}{\underline{\Hom}}
\newcommand{\BG}[1]{B #1
}

\newcommand{\mtx}[1]{\begin{pmatrix} #1 \end{pmatrix}}
\newcommand{\bmtx}[1]{\begin{bmatrix} #1 \end{bmatrix}}
\newcommand{\otinv}{\mathop{\otimes}\mathord{-}1}
\newcommand{\detExtend}{\mathsf{extend}}

\tikzset{->-/.style 2 args={
	postaction={decorate},
	decoration={markings, mark=at position #1 with {\arrow[thick, #2]{>}}} 
    },
    ->-/.default={0.5}{}
}
\tikzset{-<-/.style 2 args={
	postaction={decorate},
	decoration={markings, mark=at position #1 with {\arrow[thick, #2]{<}}} 
    },
    -<-/.default={0.5}{}
}

\tikzset{->>-/.style 2 args={
	postaction={decorate},
	decoration={markings, mark=at position #1 with {\arrow[thick, #2]{>>}}} 
    },
    ->>-/.default={0.5}{}
}
\tikzset{-<<-/.style 2 args={
	postaction={decorate},
	decoration={markings, mark=at position #1 with {\arrow[thick, #2]{<<}}} 
    },
    -<<-/.default={0.5}{}
}

\newcommand{\crs}[2]{\fill[white] (#1) circle(#2);} 

\pgfdeclarelayer{bg}    
\pgfdeclarelayer{top}    
\pgfsetlayers{bg,main,top}  

\begin{document}

\title[Adjoint torsion of 3-manifolds for semisimple algebraic groups]{Adjoint Reidemeister torsion of 3-manifolds with torus boundary for semisimple algebraic groups}

\author[Tsukasa Ishibashi]{Tsukasa Ishibashi}
\address{Tsukasa Ishibashi, Mathematical Institute, Tohoku University, 6-3 Aoba, Aramaki, Aoba-ku, Sendai, Miyagi 980-8578, Japan.}
\email{tsukasa.ishibashi.a6@tohoku.ac.jp}

\author[Yuma Mizuno]{Yuma Mizuno}
\address{Yuma Mizuno, School of Mathematical Sciences, University College Cork, Western Road, Cork, Ireland.}
\email{mizuno.y.aj@gmail.com}

\date{\today}

\begin{abstract}
  Let $M$ be a compact oriented $3$-manifold with boundary consisting of tori,
  and let $G$ be a semisimple algebraic group.
  We define the adjoint torsion function on the moduli stack of $G$-local systems on $M$ 
  satisfying a certain
  regularity condition,
  extending the construction by Porti for $G = \mathrm{SL}_2$.
  When $M$ is a cusped hyperbolic manifold,
  we prove that the local system associated with the image of the complete hyperbolic structure via a principal embedding $\mathrm{PGL}_2 \to G$ 
  satisfies the regularity condition. 
  Moreover, we provide a formula expressing its adjoint torsion as a product of $\mathrm{PGL}_2$-torsions associated with the simple $\mathrm{PGL}_2$-modules with multiplicity given by the exponents of the Lie algebra of $G$. 
  
  We compute the adjoint $\mathrm{PGSp}_4$-torsions of the figure-eight knot complement for two boundary-unipotent local systems, one is arising from the complete hyperbolic structure via a principal embedding, 
  and the other is defined over a number field of degree $6$ and not arising from any $\mathrm{PGL}_2$-local system via principal embeddings.  
\end{abstract}

\maketitle

\setcounter{tocdepth}{1}
\tableofcontents

\section{Introduction}\label{sec:intro}

\subsection{Background}
Three-dimensional Chern-Simons theory is a subject of significant interest in both theoretical physics and mathematics. 
From a mathematical perspective, it provides a framework for giving quantum invariants of three-manifolds and knots~\cite{Witten89}.
As one aspect, the perturbation theory of the Chern-Simons theory with a complex gauge group $G_{\mathbb{C}}$ is expected to yield an asymptotic series for each representation $\rho : \pi_1 M \to G_{\mathbb{C}}$~\cite{DGLZ}. It takes the form
\begin{equation}
  \widehat{\Phi}^{(M,\, \rho)}(h) = e^{v_\rho / h} \Phi^{(M,\, \rho)} (h),
  \label{eq:asymptotic expansion}
\end{equation}
where $\Phi^{(M,\, \rho)} (h)$ is a formal power series in $h$:
\begin{equation}
  \Phi^{(M,\, \rho)} (h) = \delta_\rho (1 + c_{1, \rho} h + c_{2, \rho} h^2 + \cdots).
  \label{eq:formal power series}
\end{equation}
For example, for $G_{\mathbb{C}} = \SL_2(\mathbb{C})$ and $M = S^3 \setminus 4_1$, where $4_1$ is the figure-eight knot,
the series associated with the geometric representation $\mathsf{geom}: \pi_1(M) \to \SL_2(\bC)$ corresponding to the complete hyperbolic structure is computed as
\begin{equation*}
  v_\mathsf{geom} = i \cdot 2.02988 \dots, \quad
\end{equation*}
\begin{align*}
  \Phi^{(S^3 \setminus 4_1,\, \mathsf{geom})} (h) = 
  \frac{1}{\sqrt[4]{3}} \biggl( 1 
  + \frac{11}{72\sqrt{-3}} h
  + \frac{697}{2(72\sqrt{-3})^{2}} h^{2}
  + \frac{724351}{30(72\sqrt{-3})^{3}} h^{3}
  + \cdots \biggr).
\end{align*}

The asymptotic series \eqref{eq:asymptotic expansion} has been studied in detail when $G_{\mathbb{C}} = \SL_2(\mathbb{C})$~\cite{Gukov, DimofteGaroufalidis, GZ2024}.
In this case, there are geometric interpretations of the values $v_\rho$ and $\delta_\rho$ in terms of hyperbolic geometry. Assume that $M$ is a cusped hyperbolic $3$-manifold, 
e.g., the complement of a hyperbolic knot. 
Let $\mathsf{geom}: \pi_1(M) \to \SL_2(\bC)$ be the geometric representation (also called the holonomy representation) corresponding to the complete hyperbolic structure. 
Then it is expected that $v_\mathsf{geom} = i (\mathsf{vol} + i \mathsf{cs})$, where $\mathsf{vol}$ is the hyperbolic volume and $\mathsf{cs}$ is the Chern-Simons invariant of $M$, 
and $\delta_\mathsf{geom}$ is proportional to the square root of the Ray-Singer analytic torsion associated with $\mathsf{geom}$.
Moreover, there is a well-known conjecture in knot theory, the \emph{volume conjecture}, which relates the asymptotic series \eqref{eq:asymptotic expansion} for the complement of a hyperbolic knot $K$ to a polynomial invariant of knots, the colored Jones polynomials. Specifically,
the original volume conjecture~\cite{Kashaev} and its higher order refinement~\cite{Gukov} state that the asymptotic series should coincide with the asymptotics of the Kashaev invariant $\langle K \rangle_N$:
\begin{equation}
  \langle K \rangle_N \sim N^{3/2} \widehat{\Phi}^{(S^3 \setminus K,\, \mathsf{geom})} \left(\frac{2\pi i}{N}\right).
\end{equation}
We note that the Kashaev invariant coincides with the \emph{colored Jones polynomial} evaluated at the $N$-th root of unity $q=\exp(2\pi i /N)$~\cite{MurakamiMurakami}.
We also note that the asymptotic series \eqref{eq:asymptotic expansion} associated with boundary-unipotent representations other than $\mathsf{geom}$ appears in the \emph{refined quantum modularity conjecture}~\cite{GZ2024}, which is a far-reaching refinement of the volume conjecture,
suggesting the importance of studying the series associated
with more general representations than $\mathsf{geom}$.

\subsection{Adjoint Reidemeister torsion for $G_{\mathbb{C}} = \SL_2(\mathbb{C})$}
From the perspective of the Cheeger-M\"{u}ller theorem~\cite{Cheeger, Muller}, 
the Ray-singer analytic torsion, which is expected to appear in the asymptotic series \eqref{eq:asymptotic expansion},
is also related to the Reidemeister's combinatorial torsion defined using a CW structure.

In the case of $G_{\mathbb{C}} = \SL_2(\mathbb{C})$, there is a variant of Reidemeister torsion 
due to Porti~\cite{Porti} for $3$-manifolds with torus boundary components. 
Its inverse square root is expected to coincide with the constant term $\delta_\rho$~\cite{DuboisGaroufalidis, GukovMurakami}.

In general, given a finite CW complex $M$, a representation $\rho:\pi_1 (M) \to \SL_n (\mathbb{C})$
of the fundamental group,
and a homology orientation on $M$, we have the (sign-refined) Reidemeister torsion $\torsion_M (\rho)$,
which is a simple homotopy invariant~\cite{Whitehead, Turaev}.
If the twisted chain complex $C_\bullet (M; \rho)$ is acyclic (\emph{i.e.}, the homology groups $H_i (M; \rho)$ are all trivial),
the torsion is defined as a non-zero complex number.
In general, it is interpreted as a volume form on the determinant 
line $\det H_\bullet (M; \rho)$.

Porti showed that if $M$ is a $3$-manifold whose boundary is a disjoint union $\bigsqcup_i T_i$ of tori and 
$\rho : \pi_1 (M) \to \SL_2(\mathbb{C})$ is a representation
satisfying a certain genericity condition,
then by choosing a curve $\gamma_i$ on each $T_i$ one can naturally trivialize the determinant line $\det H_\bullet (M; \Ad \circ \rho)$,
where $\Ad : \SL_2(\mathbb{C}) \to \SL(\mathfrak{sl}_2(\mathbb{C})) \cong \SL_3(\mathbb{C})$ is the adjoint representation.
Consequently, one has the torsion invariant $\torsion_{M, \gamma} (\Ad \circ \rho)$ defined as a complex number.
Porti also showed that the required genericity assumption is satisfied for the geometric representation $\mathsf{geom}: \pi_1(M) \to \SL_2(\bC)$ when $M$ is a cusped hyperbolic $3$-manifold.
When $M$ is a knot complement and $\gamma$ is chosen to be the meridian, this value is conjecturally related to 
$\delta_\rho$ in the asymptotic series~\eqref{eq:formal power series}.

Another important aspect of Porti's definition is the algebraicity of the torsion with respect to representations: it is defined as a rational function on the irreducible component $X_0(M)$ (called the \emph{distinguished component}) of the character variety $X(M) \coloneqq \Hom (\pi_1 (M), \SL_2(\mathbb{C}))\sslash\PSL_2(\mathbb{C})$ containing the geometric representation.

There are several interesting developments and related works 
concerning Porti's torsion.
\begin{description}
  \item[Asymptotic series in Chern-Simons theory]
  As mentioned above, the constant term $\delta_\rho$ in the asymptotic series~\eqref{eq:formal power series} is expected to be expressed in terms of Porti's torsion.
  Also, there is a mathematical definition of 
  the asymptotic series \eqref{eq:asymptotic expansion} given by Dimofte-Garoufalidis~\cite{DimofteGaroufalidis}
  using an ideal triangulation of $M$,
  and it is expected that the constant term of the asymptotic series in their definition
  (``the 1-loop invariant'') coincides with Porti's torsion.
  This has been proved for fibered $3$-manifolds~\cite{DGY},
  for fundamental shadow links~\cite{PandeyWong},
  and for hyperbolic two-bridge knots~\cite{OhtsukiTakata, GaroufalidisYoon}.
  See also~\cite{Siejakowski} for a reduction of the conjecture 
  based on a geometric interpretation of Thurston's gluing equations.
  \item [Relation to the twisted Alexander polynomial]
Under suitable assumptions, the adjoint $\SL_2$-torsion $\torsion_{M, \lambda} (\Ad \circ \rho)$ associated with a multi-longitude $\lambda \subset \partial M$
can be obtained as a derivative of the twisted Alexander polynomial~\cite{Yamaguchi, DuboisYamaguchi}.
In this situation, one can compute the torsion by using Fox calculus.
  \item [Vanishing property]
Gang-Kim-Yoon conjectured that the adjoint $\SL_2(\mathbb{C})$-torsions satisfy a certain vanishing property,
which is derived from the physics of wrapped M5-branes on the manifold~\cite{GangKimYoon}.
Their conjecture says that, for each level set of the trace function, 
the inverse sum of the adjoint torsions over 
a level set of the trace function vanishes.
The conjecture has been verified 
for hyperbolic two-bridge knots~\cite{Yoon22} and
for hyperbolic once-punctured torus bundles with tunnel number one~\cite{TranYamaguchi}.
See also~\cite{PortiYoon, Wakijo, Yoon22_twist}.
Moreover, this vanishing property is expected to be refined to a quadratic relation for the perturbation series $\Phi^{(K),\sigma}(h)$,
as observed in~\cite{GZ2024}.
\end{description}

The purpose of this paper is to extend the Porti's definition to general $G_{\mathbb{C}}$.

\subsection{Main results}
Throughout the paper, we work in an algebraic setting.
Let $k$ be a field of characteristic $0$, and let $G$ be a connected semisimple algebraic group over $k$.

\subsubsection{Definition of the adjoint torsion function}

Let $M$ be a compact oriented $3$-manifold whose boundary is a disjoint union of tori:
$\partial M = \bigsqcup_{i=1}^m T_i$.
We take a family of oriented simple closed curves $\gamma = (\gamma_i)_{i=1}^m$ on $\partial M$
such that each $\gamma_i$ is on $T_i$.
We fix a CW structure on $M$, $\partial M$, and $\gamma$
such that the inclusions $\gamma \subset \partial M \subset M$ are subcomplexes.

We denote by $\Loc_{G, M}$ the moduli stack of $G$-local systems on $M$ (\cref{def:moduli of G-local systems}).
Its underlying set is given by
\begin{equation}
  \lvert \Loc_{G, M} \rvert = 
  \{ (K, P) \mid 
  \text{$P$ is a $G$-local systems on $M$ 
    over a field extension $K / k$
  } \} / \sim
\end{equation}
where the two local systems are equivalent if they are isomorphic after
extending the coefficient fields to a common field extension.
The adjoint representation of $G$ on $\mathfrak{g}$ assigns to each $G$-local system 
$P \in \lvert \Loc_{G, M}\rvert$ a linear local system $P \times^G \mathfrak{g}$ on $M$.
Let $H_\bullet(M; P \times^G \mathfrak{g})$ and $H^\bullet(M; P \times^G \mathfrak{g})$ 
denote the twisted homology and cohomology of $M$ with local coefficients in $P \times^G \mathfrak{g}$.

We say that a $G$-local system $(K, P) \in \lvert \Loc_{G, M}\rvert$ 
is \emph{boundary-adjoint-regular} if the following conditions are satisfied:
\begin{enumerate}
  \item $H^0 (M; P \times^G \mathfrak{g}) = 0$,
  \item the restriction map $H^1 (M; P \times^G \mathfrak{g}) \to H^1 (\partial M; P|_{\partial M} \times^G \mathfrak{g})$ 
  is injective,
  \item $\dim_K H^0 (\partial M; P|_{\partial M} \times^G \mathfrak{g}) = m \cdot \rank G$.
\end{enumerate}
The condition (1) means that $P$ has no global flat section over $M$, (2) is an immersion condition at $P$, and (3) is an expected-dimension condition for the invariant space for monodromy homomorphisms on $\partial M$,
which is a homological analogue of the regular condition in the sense of Steinberg \cite{Steinberg}.

We actually define the moduli stack $\Loc_{G, M}^{\reg{\partial\text{-}\mathrm{Ad}}}$
of boundary-adjoint-regular $G$-local systems on $M$
as a (locally closed) substack of $\Loc_{G, M}$,
whose underlying set $\lvert \Loc_{G, M}^{\reg{\partial\text{-}\mathrm{Ad}}} \rvert$ consists 
of boundary-adjoint-regular $G$-local systems
as defined above. See \eqref{eq:adjoint regular substack},
Definition~\ref{def:boundary-regular}, and 
Lemma~\ref{lem:boundary regular substack}.

Let $\mathfrak{o} : \det H_\bullet (M; \mathbb{Q}) \cong \mathbb{Q}$ be a homological orientation of $M$.

\begin{introthm}[{Definition \ref{def:adjoint torsion function}}]
  We define the \emph{adjoint torsion function} as a morphism of stacks
  \begin{align}
    \torsion_{G, M, \gamma, \mathfrak{o}}^{\mathrm{Ad}} : \Loc_{G, M}^{\reg{\partial\text{-}\mathrm{Ad}}} \to \mathbb{A}_k^1,
  \end{align}
  whose value 
  $\torsion_{G, M, \gamma, \mathfrak{o}}^{\mathrm{Ad}}(P) \in K$
  at a boundary-adjoint-regular local system $(K, P) 
  \in \lvert \Loc_{G, M}^{\reg{\partial\text{-}\mathrm{Ad}}} \rvert$
  is defined as a non-acyclic Reidemeister torsion associated with $P \times^{G} \mathfrak{g}$.
\end{introthm}

In the case $G=\SL_2$, our adjoint torsion function corresponds to the Porti's torsion formulated in terms of the stack $\Loc_{\SL_2,M}$ instead of the character variety (GIT quotient). Kitayama-Terashima \cite{KitayamaTerashima} introduced adjoint $\PGL_n$-torsion as a set-theoretic function, which has been later formulated algebraically also on the character variety by Porti \cite{Porti18}. 

The adjoint torsion function is a simple homotopy invariant (\cref{thm:topological_invariance}).
If $\rank G$ is even, then
the adjoint torsion function does not depend on the choice of 
a homological orientation $\mathfrak{o}$ (\cref{lemma:torsion of even local system}).

\subsubsection{Principal embedding of the geometric representation}
Over an algebraically closed field $k$,
there is a principal embedding $\iota : \PGL_2 \to G$ \cite{Kostant} that sends 
the element $\begin{bmatrix}
  1 & 1 \\ 0 & 1
\end{bmatrix} \in \PGL_2(k)$ to a regular unipotent element in $G(k)$ in the sense 
of Steinberg~\cite{Steinberg}.
The map $\iota$ induces a morphism
$\iota : \Loc_{\PGL_2, M} \to \Loc_{G, M}$  between the moduli stacks of local systems.
We show that the local system arising from the complete hyperbolic structure 
via principal embeddings are boundary-adjoint-regular. 
\begin{introthm}[\cref{thm:geom_regular}]
  Let $k = \mathbb{C}$. Let $\iota : \PGL_2 \to G$ be a principal embedding.
  Suppose that the interior of $M$ admits a complete hyperbolic structure.
  Then the local system $\iota(\mathsf{geom}) \in \Loc_{G,M}$, 
  which is the image of the geometric local system associated with the complete hyperbolic structure under the principal embedding, 
  is boundary-adjoint-regular.
\end{introthm}
The adjoint torsion of $\iota(\mathsf{geom})$ is computed as follows. 
Let $V_d$ be the $d$-dimensional simple module of $\PGL_2$ for each positive odd integer $d$.
The pullback of the adjoint representation by the principal embedding $\iota$ 
decomposes as
\begin{equation*}
  \iota^* \mathfrak{g} \cong \bigoplus_{i} V_{2 m_i + 1},
\end{equation*}
where the sequence of numbers $(m_i)_i$ are called the \emph{exponents} of $\mathfrak{g}$~\cite{Kostant}.
Table \ref{table:exponents} lists the exponents associated with simple Lie algebras. The following gives a generalization of \cite[Proposition 5.6]{Porti18} for $G=\PGL_n$:

\begin{introthm}[\cref{thm:adjoint_torsion_hol}]\label{cor:intro ad tor decomp}
  We have
  \begin{equation}\label{eq:intro ad tor decomp}
      \torsion_{G, M,\gamma,\mathfrak{o}}^{\Ad}(\iota(\mathsf{geom})) = 
        \prod_{i}
          \torsion_{M,\gamma,\mathfrak{o}}(\mathsf{geom} \times^{\PGL_2} V_{2 m_i + 1}),
  \end{equation}
  where, in the right-hand side, 
  $\torsion_{M,\gamma,\mathfrak{o}}(\mathcal{L})$ denotes
  the Reidemeister torsion associated with a linear local system $\mathcal{L}$ 
  (see Definition \ref{def:torsion function}).
\end{introthm}
As a consequence, we see $\torsion_{G, M,\gamma,\mathfrak{o}}^{\Ad}(\iota(\mathsf{geom})) \neq 0$ (\cref{cor:torsion_geom_nonzero}).

\begin{table}[t]
  \centering
  \begin{tabular}{cl}
    {Type} & {Exponents} \\
    \hline
    ${A}_n$ & $1, 2, 3, \ldots, n$ \\
    ${B}_n$ & $1, 3, 5, \ldots, 2n - 1$ \\
    ${C}_n$ & $1, 3, 5, \ldots, 2n - 1$ \\
    ${D}_n$ & $1, 3, 5, \ldots, 2n - 3, n - 1$ \\
    ${E}_6$ & $1, 4, 5, 7, 8, 11$ \\
    ${E}_7$ & $1, 5, 7, 9, 11, 13, 17$\\
    ${E}_8$ & $1, 7, 11, 13, 17, 19, 23, 29$\\
    ${F}_4$ & $1, 5, 7, 11$ \\
    ${G}_2$ & $1, 5$ \\
  \end{tabular}
  \caption{Exponents associated with simple Lie algebras of given types.}
  \label{table:exponents}
\end{table}

The torsion functions in the right-hand side of \eqref{eq:intro ad tor decomp} coincide with 
those associated with $\SL_2$-representations
defined by Menal-Ferrer and Porti~\cite{MFP_torsion}.
Their values at the geometric local system can be computed using \texttt{hyperbolic\_SLN\_torsion} method
in SnapPy~\cite{SnapPy}.

\subsubsection{Adjoint torsion of a local system not arising from $\PGL_2$-local systems via principal embeddings}
We compute the adjoint torsion of an explicit example of a boundary-adjoint-regular local system for the figure-eight knot complement and the projective symplectic group $G=\PGSp_4$ (i.e., the split adjoint group of type $C_2$) that does not arise from a $\PGL_2$-local system via any principal embedding, where we don't have a formula like \eqref{eq:intro ad tor decomp} and thus the torsion cannot be computed from known $\PGL_2$-torsions.

\begin{introex}
  Let $K$ be a number field defined by
  \begin{equation*}
    K \coloneqq \mathbb{Q} [\omega] / (\omega^6 - \omega^5 + 3\omega^4 - 5\omega^3 + 8\omega^2 - 6\omega + 8).
  \end{equation*}
  Zickert~\cite{Zickert} found a boundary-unipotent $\PGSp_4(K)$-local system $\mathcal{L}$
  on the figure-eight knot complement
  by using the method of the Ptolemy relations.
  We note that the Ptolemy relations are interpreted as the relation among 
  cluster coordinates in higher Teichm\"uller theory~\cite{FockGoncharov, LeIan, GoncharovShen}.
  One can verify that the local system $\mathcal{L}$ is boundary-adjoint-regular but does 
  not arise from a $\PGL_2$-local system via any principal embedding.
  For a meridian $\mu$, we compute
  \begin{equation*}
    \torsion_{\PGSp_4, S^3 \setminus 4_1, \mu}^{\Ad} (P) = \frac{85}{16} \omega^{5} - \frac{33}{8} \omega^{4} + \frac{217}{16} \omega^{3} - \frac{99}{8} \omega^{2} + \frac{321}{8} \omega - 11.
  \end{equation*}
  which has the following complex embeddings:
  \begin{align*}
    \torsion_{\PGSp_4, S^3 \setminus 4_1, \mu} (P) \mapsto
    \begin{cases}
    -3.459966243820608\\
    1.104983121910304 + 38.11233948347826 \sqrt{-1}\\
    1.104983121910304 - 38.11233948347826 \sqrt{-1}.
    \end{cases}
  \end{align*}
\end{introex}

\subsubsection{Related works}
\begin{itemize}
    \item Kitayama-Terashima \cite{KitayamaTerashima} provided a method to compute the adjoint $\PGL_n$-torsion by using the Fock-Goncharov coordinates \cite{FockGoncharov} on the moduli stack of framed $\PGL_n$-local systems when $M$ is a mapping torus over a punctured surface. This method is effective in explicit computations in coordinates.
    \item Naef and Safronov \cite{NaefSafronov} describe the adjoint torsion function of closed oriented $3$-manifolds $M$ from the viewpoint of derived algebraic geometry. Their torsion is defined on the moduli stack of $G$-local systems on $M$, which is viewed as a $(-1)$-shifted symplectic stack. They provide an application to the cohomological Donaldson-Thomas invariants of the moduli stacks of local systems on $M$. 
    They also provide insights to the relation with the $1$-loop determinant in the cotangent AKSZ theory. 
\end{itemize}

\subsection{Future work: computation of adjoint torsion by cluster coordinates}
When $M$ is a mapping torus of a homeomorphism $f$ on a punctured surface $\Sigma$, a $G$-local system on $M$ corresponds (generically one-to-one) to a $G$-local system on $\Sigma$ which is invariant under $f$. Kitayama and Terashima \cite{KitayamaTerashima} provide a method to compute the adjoint torsion for $G=\PGL_n$ by using the Fock-Goncharov coordinates on the moduli stack $\X_{G,\Sigma}$ of framed $G$-local systems on $\Sigma$ \cite{FockGoncharov} for $G=\PGL_n$. See also \cite{NTY,TerashimaYamazaki} for a background on this method. There is also a cluster theoretic interpretation \cite{Mizuno} of the $1$-loop invariant of Dimofte-Garoufalidis \cite{DimofteGaroufalidis}.

Later, cluster coordinates on $\X_{G,\Sigma}$ for general semisimple algebraic group $G$ of adjoint type is found by Goncharov-Shen \cite{GoncharovShen}, generalizing and providing a uniform explanation of earlier works of Zickert \cite{Zickert} and Le \cite{LeIan}. Then it is straightforward to generalize the result of Kitayama-Terashima for such $G$. 

In a future work, we plan to proceed further. For any $3$-manifold within the class of consideration in this paper, we may define a cluster-like coordinates on the moduli stack $\X_{G,M}$ upon choosing a 3D ideal triangulation. The idea is similar to the \emph{gluing variety} of Garoufalidis-Goerner-Zickert \cite{GaroufalidisGoernerZickert} and Dimofte-Gabella-Goncharov \cite{DimofteGabellaGoncharov} for $G=\PGL_n$, while we lose the nice symmetry existing in type $A_n$ in general. We will develop a theory of gluing variety associated with a mutation network and then compute the adjoint torsion in the associated coordinate system. 

\subsection*{Acknowledgements}
Considerable part of this work is substantialized during the first-named authors' visit to University College Cork in October 2025. He is grateful to Robert Osburn for his hospitality. 
Tsukasa Ishibashi is partially supported by JSPS KAKENHI Grant Number JP24K16914. Yuma Mizuno is partially supported by the Irish Research Council Advanced Laureate Award IRCLA/2023/1934. 


\section{Determinants of complexes}
\label{sec:determinant functor}
In this section, we recollect a general framework to compute the determinants of complexes with functorial control over signs. An explicit computation using the sign convention in \cite{Nicolaescu} will be carried out in \cref{sec:computation}. 

\subsection{Graded invertible modules}
A \emph{groupoid} is a category whose morphisms are all isomorphisms. 
Let $A$ be a commutative ring. 
We denote by $\mathcal{P}_A$ the groupoid of graded invertible $A$-modules.
An object of $\mathcal{P}_A$ is a pair $(L, \alpha)$, where $L$ is a projective $A$-module 
of rank $1$ and $\alpha : \Spec A \to \mathbb{Z}$ is a locally constant function.
The set of morphisms between two objects are given by
\begin{align}
  \Hom_{\mathcal{P}_A} ((L, \alpha), (M, \beta)) \coloneqq
  \begin{cases}
  \Iso(L, M) &\text{if $\alpha = \beta$}\\
    \emptyset &\text{otherwise}
  \end{cases}
\end{align}
We will often write $L \in \mathcal{P}_A$, only indicating the underlying module $L$, and denote by $\lvert L \rvert : \Spec A \to \mathbb{Z}$ the locally constant function $\alpha$.
The groupoid $\mathcal{P}_A$ is equipped with a symmetric monoidal structure where
the tensor product is given by
\begin{align}
  (L, \alpha) \otimes (M, \beta) \coloneqq (L \otimes M, \alpha + \beta),
\end{align}
and the unit object is given by $1 \coloneqq 1_A \coloneqq (A, 0)$.
For $L, M \in \mathcal{P}_A$,
the braiding 
\begin{equation}\label{eq:braiding def}
  \beta_{L, M} : L \otimes M \to M \otimes L
\end{equation}
is locally given by
\begin{align}\label{eq:sign braiding}
  \ell \otimes m \mapsto (-1)^{\lvert L \rvert \cdot \lvert M \rvert} m \otimes \ell
\end{align}
for $\ell \in L$ and $m \in M$,
and the unit object is given by $1 \coloneqq (A, 0)$.

An \emph{inverse structure} \cite[Definition A.16]{Knudsen} on $\mathcal{P}_A$ is a pair 
$((-)^{\otinv}, \varepsilon)$, where $(-)^{\otinv} : \mathcal{P}_A \to \mathcal{P}_A$ is a 
braided monoidal functor, and $\varepsilon : \id \otimes (-)^{\otinv} \to 1$ is a monoidal natural 
transformation.
In other words, an inverse structure consists of the following data:
\begin{enumerate}
  \item for each $L \in \mathcal{P}_A$, an object $L^{\otinv} \in \mathcal{P}_A$,
  \item for each morphism $f : L \to M$ in $\mathcal{P}_A$, a morphism 
    $f^{\otinv} : L^{\otinv} \to M^{\otinv}$ in $\mathcal{P}_A$,
  \item for each $L, M \in \mathcal{P}_A$, a morphism
    $\theta_{L, M} : (L \otimes M)^{\otinv} \to L^{\otinv} \otimes M^{\otinv}$ in $\mathcal{P}_A$,
  \item a morphism $1^{\otinv} \to 1$ in $\mathcal{P}_A$,
  \item for each $L \in \mathcal{P}_A$, a morphism
    $\varepsilon_L : L \otimes L^{\otinv} \to 1$ in $\mathcal{P}_A$,
\end{enumerate}
such that the following diagrams commute:
\begin{align}\label{eq:theta natural}
    \begin{tikzpicture}[scale=1.3, commutative diagrams/every diagram, baseline={(current bounding box.center)}]
      \node (0) at (0,0) {$(L \otimes M)^{\otinv}$};
      \node (1) at (4,0) {$(L^{\otinv} \otimes M^{\otinv})$};
      \node (2) at (0,-1) {$(L' \otimes M')^{\otinv}$};
      \node (3) at (4,-1) {$L'^{\otinv} \otimes M'^{\otinv}$};
      \draw[morphism] (0) edge node[above] {$\theta_{L, M}$} (1);
      \draw[morphism] (0) edge node[left] {$(f \otimes g)^{\otinv}$} (2);
      \draw[morphism] (2) edge node[below] {$\theta_{L', M'}$} (3);
      \draw[morphism] (1) edge node[right] {$f^{\otinv} \otimes g^{\otinv}$} (3);
    \end{tikzpicture}
\end{align}
  \begin{align}\label{eq:theta assoc}
    \begin{tikzpicture}[scale=1.3, commutative diagrams/every diagram, baseline={(current bounding box.center)}]
      \node (0) at (0,0) {$((L \otimes M) \otimes N)^{\otinv}$};
      \node (1) at (4,0) {$(L \otimes M)^{\otinv} \otimes N^{\otinv}$};
      \node (2) at (8,0) {$(L^{\otinv} \otimes M^{\otinv}) \otimes N^{\otinv}$};
      \node (3) at (0,-1) {$(L \otimes (M \otimes N))^{\otinv}$};
      \node (4) at (4,-1) {$L^{\otinv} \otimes (M \otimes N)^{\otinv}$};
      \node (5) at (8,-1) {$L^{\otinv} \otimes (M^{\otinv} \otimes N^{\otinv})$};
      \draw[morphism] (0) edge node[above] {$\theta_{L \otimes M, N}$} (1);
      \draw[morphism] (1) edge node[above] {$\theta_{L ,M} \otimes \id$} (2);
      \draw[morphism] (0) edge node[left] {$\mathsf{assoc}^{\otinv}$} (3);
      \draw[morphism] (3) edge node[below] {$\theta_{L, M \otimes N}$} (4);
      \draw[morphism] (4) edge node[below] {$\id \otimes \theta_{M ,N}$} (5);
      \draw[morphism] (2) edge node[right] {$\mathsf{assoc}$} (5);
    \end{tikzpicture}
  \end{align}
  \begin{align}\label{eq:theta unit}
    \begin{tikzpicture}
      [scale=1.3, xscale = 0.6, commutative diagrams/every diagram, baseline={(current bounding box.center)}]
      \node (0) at (0,0) {$(1 \otimes L)^{\otinv}$};
      \node (1) at (4,0) {$1^{\otinv} \otimes L^{\otinv}$};
      \node (2) at (4,-1) {$1 \otimes L^{\otinv}$};
      \node (3) at (0,-1) {$L^{\otinv}$};
      \draw[morphism] (0) edge node[above] {$\theta_{1, L}$} (1);
      \draw[morphism] (0) edge node[left] {} (3);
      \draw[morphism] (2) edge node[below] {} (3);
      \draw[morphism] (1) edge node[right] {} (2);
    \end{tikzpicture}\quad
    \begin{tikzpicture}
      [scale=1.3, xscale = 0.6, commutative diagrams/every diagram, baseline={(current bounding box.center)}]
      \node (0) at (0,0) {$(L \otimes 1)^{\otinv}$};
      \node (1) at (4,0) {$L^{\otinv} \otimes 1^{\otinv}$};
      \node (2) at (4,-1) {$L^{\otinv} \otimes 1$};
      \node (3) at (0,-1) {$L^{\otinv}$};
      \draw[morphism] (0) edge node[above] {$\theta_{L, 1}$} (1);
      \draw[morphism] (0) edge node[left] {} (3);
      \draw[morphism] (2) edge node[below] {} (3);
      \draw[morphism] (1) edge node[right] {} (2);
    \end{tikzpicture}
  \end{align}
  \begin{align}\label{eq:theta comm}
    \begin{tikzpicture}[scale=1.3, commutative diagrams/every diagram, baseline={(current bounding box.center)}]
      \node (0) at (0,0) {$(L \otimes M)^{\otinv}$};
      \node (1) at (4,0) {$(L^{\otinv} \otimes M^{\otinv})$};
      \node (2) at (0,-1) {$(M \otimes L)^{\otinv}$};
      \node (3) at (4,-1) {$M^{\otinv} \otimes L^{\otinv}$};
      \draw[morphism] (0) edge node[above] {$\theta_{L, M}$} (1);
      \draw[morphism] (0) edge node[left] {$\beta_{L, M}^{\otinv}$} (2);
      \draw[morphism] (2) edge node[below] {$\theta_{M, L}$} (3);
      \draw[morphism] (1) edge node[right] {$\beta_{L^{\otinv}, M^{\otinv}}$} (3);
    \end{tikzpicture}
  \end{align}
  \begin{align}\label{eq:epsilon naturality}
    \begin{tikzpicture}[scale=1.3, commutative diagrams/every diagram, baseline={(current bounding box.center)}]
      \node (0) at (0,0) {$L \otimes L^{\otinv}$};
      \node (1) at (3,0) {$1$};
      \node (2) at (0,-1) {$M \otimes M^{\otinv}$};
      \node (3) at (3,-1) {$1$};
      \draw[morphism] (0) edge node[above] {$\varepsilon_L$} (1);
      \draw[morphism] (0) edge node[left] {$f \otimes f^{\otinv}$} (2);
      \draw[morphism] (2) edge node[below] {$\varepsilon_M$} (3);
      \draw[morphism] (1) edge node[right] {$\cong$} (3);
    \end{tikzpicture}
  \end{align}
  \begin{align}
    \begin{tikzpicture}\label{eq:epsilon unit}
      [scale=1.3, commutative diagrams/every diagram, baseline={(current bounding box.center)}]
      \node (0) at (0,0) {$1 \otimes 1^{\otinv}$};
      \node (1) at (2,0) {$1$};
      \node (2) at (0,-1) {$1 \otimes 1$};
      \draw[morphism] (0) edge node[above] {$\varepsilon_1$} (1);
      \draw[morphism] (0) edge node {} (2);
      \draw[morphism] (2) edge node[below] {$\cong$} (1);
    \end{tikzpicture}
  \end{align}
  \begin{align}\label{eq:epsilon monoidal}
    \begin{tikzpicture}[scale=1.3, xscale=1.1, commutative diagrams/every diagram, baseline={(current bounding box.center)}]
      \node (0) at (0,0) {$(L \otimes M) \otimes (L \otimes M)^{\otinv}$};
      \node (1) at (3.7,0) {$(L \otimes M) \otimes (L^{\otinv} \otimes M^{\otinv})$};
      \node (2) at (8,0) {$(L \otimes L^{\otinv}) \otimes (M \otimes M^{\otinv})$};
      \node (3) at (0,-1) {$1$};
      \node (4) at (8,-1) {$1 \otimes 1$};
      \draw[morphism] (0) edge node[above] {$\id \otimes \theta$} (1);
      \draw[morphism] (1) edge node[above] {$\id \otimes \beta \otimes \id$} (2);
      \draw[morphism] (0) edge node[left] {$\varepsilon_{L \otimes M}$} (3);
      \draw[morphism] (2) edge node[right] {$\varepsilon_L \otimes \varepsilon_M$} (4);
      \draw[morphism] (3) edge node[below] {$\cong$} (4);
    \end{tikzpicture}
  \end{align}  

By $\varepsilon_L : L \times L^{\otinv} \to 1$, the object $L^{\otinv}$ is a right dual of $L$.
We regard $L^{\otinv}$ as the left dual of $L$ as well,
using the map
\begin{align}\label{eq:left dual}
  L^{\otinv} \otimes L \xrightarrow{\beta_{L^{\otinv}, L}} 
  L \otimes L^{\otinv} \xrightarrow{\varepsilon_L} 1.
\end{align}

For inverse structures $(\sigma, \varepsilon)$ and $(\sigma', \varepsilon')$,
a morphism between them is a monoidal natural transformation $\alpha : \sigma \to \sigma'$
such that $\varepsilon = \varepsilon' \circ (\id \otimes \alpha)$.
We denote by $\InvStr (\mathcal{P}_A)$ the category of inverse structures on $\mathcal{P}_A$.

\begin{lemma}\label{lemma:smul otinv}
  Let $((-)^{\otinv}, \varepsilon)$ be an inverse structure on $\mathcal{P}_A$.
  For $f : L \to M$ in $\mathcal{P}_A$ and $a \in A^{\times}$,
  we have $(a \cdot f)^{\otinv} = a^{-1} \cdot f^{\otinv}$.
\end{lemma}
\begin{proof}
  This follows from the naturality \eqref{eq:epsilon naturality} of $\varepsilon$ 
  applied for $f$ and $a \cdot f$.
\end{proof}

\subsection{Determinant functors}
We recall the notion of determinant functors following \cite{Deligne, Knudsen}.
\begin{definition}
  Let $\mathcal{E}$ be an exact category,
  and $w$ be a set of morphisms in $\mathcal{E}$ containing all isomorphisms
  and being closed under composition.
  We denote by $\mathcal{E}_w$ the subcategory of $\mathcal{E}$
  whose morphisms are those in $w$.
  A \emph{determinant functor} on $(\mathcal{E}, w)$ consists of the following data:
  \begin{enumerate}
    \item a functor $\det : \mathcal{E}_w \to \mathcal{P}_A$,
    \item for each short exact sequence $\Sigma : 0 \to A \to B \to C \to 0$ in $\mathcal{E}$,
      a morphism $\det(\Sigma) : \det(B) \to \det(A) \otimes \det(C)$ in $\mathcal{P}_A$,
  \end{enumerate}
  such that they satisfy naturality with respect to morphisms in $w$ of short exact sequences,
  associativity, and commutativity.
  We refer to \cite[Definition 2.4]{Breuning} for the precise definition.
  We denote by $\Det (\mathcal{E}, w)$ the category of determinant functors on $(\mathcal{E}, w)$.
\end{definition}

In this paper, we only consider determinants on the following exact categories.
\begin{example}
  Let $A$ be a commutative ring.
  \begin{enumerate}
    \item $(\mathsf{proj}(A), \mathrm{iso})$: $\mathcal{E}$ is the category of 
      finitely generated projective $A$-modules,
      and $w$ is the set of isomorphisms.
    \item $(C^{\mathrm{b}}(\mathsf{proj}(A)), \mathrm{qis})$: $\mathcal{E}$ is the category of
      bounded complexes of finitely generated projective $A$-modules,
      and $w$ is the set of quasi-isomorphisms.
  \end{enumerate}
\end{example}

For categories $C$ and $D$,
we denote by $\Equiv (C, D)$ the category of equivalences from $C$ to $D$.

\begin{theorem}[{\cite[Theorem 2.3]{Knudsen}}]\label{theorem:det functor}
  We have a functor 
  \begin{align*}
    \detExtend :
    \InvStr (\mathcal{P}_A) \to \Equiv
    \bigl(
      \Det (\mathsf{proj}(A), \mathrm{iso}),
      \Det (C^{\mathrm{b}} (\mathsf{proj}(A)), \mathrm{qis})
    \bigr)
  \end{align*}
  such that the quasi-inverse $(\detExtend_\sigma)^{-1}$ for $\sigma \in \InvStr (\mathcal{P}_A)$ is defined
  to be the restriction functor.
\end{theorem}

Consequently, choices of an inverse structure on $\mathcal{P}_A$
and a determinant functor on $(\mathsf{proj}(A), \mathrm{iso})$
determine a determinant functor on $(C^{\mathrm{b}} (\mathsf{proj}(A)), \mathrm{qis})$.
Roughly speaking, these input data correspond to the ``sign convention'',
and the above theorem states that we have a determinant functor on complexes
once we fix a sign convention.

\begin{theorem}[{\cite[Definition 3.1]{Knudsen}}]
  \label{theorem:det C to det H}
  Let $(C^i, d^i)_{i \in \mathbb{Z}} \in C^b (\mathsf{proj}(A))$.
  Suppose that $\Img d^i$, $\Ker d^i$, and $H^i (C)$ are finitely 
  generated projective $A$-modules for any $i$. 
  Let $\det$ be a determinant functor on $(C^{\mathrm{b}} (\mathsf{proj}(A)), \mathrm{qis})$.
  Then we have an isomorphism
  \begin{align}
    \Phi : \det C \cong \det H (C),
  \end{align}
  where $H (C) = (H^i (C))_i$ is a complex whose differentials are all zero.
\end{theorem}

We will use \cref{theorem:det C to det H} for complexes
with decreasing indices. 
To do so,
for a complex $C_\bullet = (C_i, \partial_i)_{i \in \mathbb{Z}}$
where the differentials $\partial_i : C_i \to C_{i-1}$ decrease the indices,
we regard it as a complex with increasing indices by setting $C^i \coloneqq C_{-i}$
and $d^i \coloneqq \partial_{-i}$.

\section{Moduli stack of local systems and twisted complexes}\label{sec:torsion}
Throughout this section, let $k$ be a field.
We denote by $\Alg_k$ the category of $k$-algebras.
We denote by $\Grpd$ the 2-category of groupoids.
A \emph{stack} is a pseudofunctor $\Alg_k \to \Grpd$ 
that satisfies the sheaf conditions with respect to
the fppf topology. For example, given a scheme $X$,
\begin{align*}
    \Alg_k \to \Grpd, \quad A \mapsto X(A)=\Hom_{\mathsf{Sch}}(\Spec A, X)
\end{align*}
defines a stack, where the set $X(A)$ is regarded as a discrete groupoid.

For a stack $X$, 
we will often denote the extension of scalar functor 
$X(f) : X(A) \to X(B)$ associated with a map $f : A \to B$ by $(-) \otimes_A B$, with given $f$ understood.

\subsection{Moduli stack of $G$-local systems}
\label{sec:moduli of G-local systems}
In this section, we review basic definitions and properties of $G$-local systems.
Let $G$ be a smooth affine algebraic group over $k$.
Let $M$ be a finite CW complex.
We denote by $\Pi_1 M$ the fundamental groupoid of $M$.
It is a groupoid whose objects are the points of $M$,
and morphisms from $x$ to $y$ are homotopy classes of paths from $x$ to $y$.
For any category $C$, and morphisms $f : x \to y$ and $g : y \to z$ in $C$,
we denote by $f \gg g$ the ``path ordered notation'' of the composition.
We have $f \gg g = g \circ f$ in the usual notation.
We will use this notation particularly when $C$ is a fundamental groupoid.

We denote by $\BG{G}$ the classifying stack of $G$.
Explicitly, for a $k$-algebra $A$, the groupoid $\BG{G}(A)$ 
is the category of $G_A$-torsors. 

For stacks $X$ and $Y$, the \emph{mapping stack} $\iHom(X, Y)$ is defined by
\begin{align}
  \iHom(X, Y) (A) \coloneqq \Hom_{\Hom(\Alg_k, \Grpd)} (X \times \Hom(A, -), Y).
\end{align}

\begin{definition}[Moduli stack of $G$-local systems]
  \label{def:moduli of G-local systems}
  We define the moduli stack $\Loc_{G, M}$ of $G$-local systems on $M$ as the following mapping stack:
  \begin{equation}\label{eq:G-loc stack}
    \Loc_{G, M} = \iHom (\underline{\Pi_1 M}, \BG{G})
  \end{equation}
  where $\underline{\Pi_1 M}$ is the constant stack associated with the fundamental groupoid $\Pi_1 M$.
\end{definition}

\begin{remark}\label{rem:moduli of G-local systems}
  We have the following equivalence:
  \begin{alignat}{1}
    \Loc_{G, M}(A) &\simeq \Hom_{\Grpd}(\Pi_1 M, \BG{G}(A))
  \end{alignat}
  where the right-hand side is the functor category.
\end{remark}

By Remark \ref{rem:moduli of G-local systems},
we see that an object of $\Loc_{G, M}(A)$ consists of the following data:
\begin{enumerate}
  \item for any $x \in M$, a $G_A$-torsor $P_x$,
  \item for any $x, y \in M$ and any (homotopy class of) path $\gamma : x \to y$,
  an isomorphism of $G_A$-torsors $P_\gamma : P_x \to P_y$
  such that
  \begin{equation*}
    P_{\id_x} = \id_{P_x}, \quad
    P_{\gamma_1 \gg \gamma_2} = P_{\gamma_2} \circ P_{\gamma_1}.
  \end{equation*}
\end{enumerate}
A morphism from $P$ to $P'$ in $\Loc_{G, M}(A)$
is a family of isomorphisms of $G_A$-torsors
$\varphi_x : P_x \to P'_x$ for any $x \in M$
such that for any path $\gamma : x \to y$,
the following diagram commutes:
\begin{equation}\label{eq:morphism of G-local systems}
  \begin{tikzpicture}[scale=1.3, commutative diagrams/every diagram, baseline={(current bounding box.center)}]
    \node (a) at (0,0) {$P_x$};
    \node (b) at (2,0) {$P_y$};
    \node (c) at (0,-1) {$P'_x$};
    \node (d) at (2,-1) {$P'_y$};
    \draw[morphism] (a) to node[above] {$P_\gamma$} (b);
    \draw[morphism] (a) to node[left] {$\varphi_x$} (c);
    \draw[morphism] (b) to node[right] {$\varphi_y$} (d);
    \draw[morphism] (c) to node[below] {$P'_\gamma$} (d);
  \end{tikzpicture}
\end{equation}

\begin{definition}[Moduli stack of marked $G$-local systems]
  \label{def:moduli of marked $G$-local systems}
  Let $S \subset M$ be a finite set such that each connected component of $M$ contains at least one point of $S$.
  We regard $S$ as the full subcategory of $\Pi_1 M$ whose objects are the points in $S$.
  We define the moduli stack 
  \begin{align}
    \Loc_{G, M, S} : \Alg_k \to \Grpd
  \end{align}
  of \emph{marked $G$-local systems}
  so that
  for any $k$-algebra $A$, 
  an object of $\Loc_{G, M, S}(A)$ consists of 
  an object $P \in \Loc_{G, M}(A)$ and
  a section $p_s \in P_s(A)$ for any $s \in S$; 
  a morphism from $(P, p)$ to $(P', p')$ is a morphism $\varphi : P \to P'$ 
  in $\Loc_{G, M}(A)$ such that $\varphi_s(p_s) = p'_s$ for any $s \in S$.
\end{definition}

\begin{lemma}\label{lemma:no auto in Loc G M S}
  Any object in $\Loc_{G, M, S}(A)$ has a trivial automorphism group.
\end{lemma}
\begin{proof}
  Let $\varphi : (P, p) \to (P, p)$ be an automorphism.
  For any $s \in S$,
  we have $\varphi_s (p_s) = p_s$, which implies that $\varphi_s = \id_{P_s}$.
  For any $x \in M$, we choose a point $s_x \in S$ and a path $\gamma_x : s_x \to x$.
  Noting that $\varphi_{s_x} = \id_{P_{s_x}}$,
  the commutative diagram \eqref{eq:morphism of G-local systems} for $\gamma_x$
  shows that $\varphi_x = \id_{P_x}$.
\end{proof}

For a group $\Gamma$, we denote by $\bullet_\Gamma$ the groupoid with a single object 
whose automorphism group is $\Gamma$.

\begin{lemma}[Monodromy correspondence]\label{lemma:monodromy of Loc G M S}
We have a bijection
\begin{equation}
  \label{eq:mon G}
  \mon_{G, M}(A) : \Obj (\Loc_{G, M, S}(A)) / \mathord{\cong} \xlongrightarrow{\cong} 
  \Obj \left(\Hom_{\Grpd} \left(S,\ \bullet_{G(A)}\right)\right)
\end{equation}
where $\Obj (C)$ is the set of objects of a category $C$.
\end{lemma}
\begin{proof}
For an object $(P, p) \in \Loc_{G, M, S}(A)$,
the functor $\mon_{G}(A, P)$ is defined
by sending a path $\gamma : s \to s'$ for $s, s' \in S$ to the unique element $g \in G(A)$ 
satisfying
\begin{align}
  P_\gamma (A) (p_s) = p_{s'} \cdot g.
\end{align}
Conversely, suppose that we have a functor $\alpha : S \to \bullet_{G(A)}$.
For any point $x \in M$, we define $P_x \coloneqq G_A$.
For each $x \in M$, we can choose a point $s_x \in S$ and a path $\gamma_x : s_x \to x$.
For each $\gamma : x \to y$ in $M$, we define a path $\tilde{\gamma} : s_x \to s_y$ by 
$\tilde{\gamma} \coloneqq \gamma_x \gg \gamma \gg \gamma_y^{-1}$.
Then we define $P_\gamma: P_x \to P_y$ to be the left multiplication by $\alpha (\tilde{\gamma}) \in G(A)$.
For $s \in S$, we define
$p_s \coloneqq e \in G(A)$, where $e$ is the identity.
\end{proof}

By Lemma \ref{lemma:monodromy of Loc G M S}, 
for any marked local system $P \in \Loc_{G, M, S}(A)$,
any points $x, y \in S$,
and any path $\gamma : x \to y$ in $M$,
we have an element $\mon_P(\gamma) \coloneqq \mon_{G, M} (A, P, \gamma) \in G(A)$, which satisfies 
\begin{equation*}
  \mon_P (\id_x) = e, \quad
  \mon_P (\gamma_1 \gg \gamma_2) = \mon_P (\gamma_2) \mon_P (\gamma_1).
\end{equation*}
The element $\mon_P (\gamma)$ is called the \emph{monodromy} of $P$ along $\gamma$.

\begin{lemma}\label{lemma:monodromy natural}
  The map $\mon_{G, M}(A)$
  is natural in $A$ with respect to the extension of scalars.
\end{lemma}
\begin{proof}
  For a map $f : A \to B$ of $k$-algebras,
  we can see that $G(f) (\mon_P(\gamma)) = \mon_{P \otimes_A B} (\gamma)$.
\end{proof}

\begin{lemma}
  The presheaf $\Loc_{G, M, S}$ is represented by a $k$-algebra.
\end{lemma}
\begin{proof}
  By Lemma \ref{lemma:no auto in Loc G M S}, \ref{lemma:monodromy of Loc G M S},
  and \ref{lemma:monodromy natural},
  it suffices to show that the presheaf
  \begin{equation*}
    F : \Alg_k \to \Set, \quad A \mapsto \Obj \left(\Hom_{\Grpd} \left(S,\ \bullet_{G(A)} \right)\right)
  \end{equation*}
  is representable.
  Let $\mathcal{O}(G)$ be the coordinate ring of $G$.
  We denote by $\Delta : \mathcal{O}(G) \to \mathcal{O}(G) \otimes_k \mathcal{O}(G)$
  the comultiplication, and by $\varepsilon : \mathcal{O}(G) \to k$ the counit.
  Let $\Lambda$ be the set of morphisms in $S$.
  Let $\Fin (\Lambda)$ be the set of finite subsets of $\Lambda$.
  Let 
  \begin{equation*}
    \mathcal{O}(G)^{\otimes \Lambda} \coloneqq \colim_{N \in \Fin (\Lambda)} \bigotimes_{\gamma \in N} \mathcal{O}(G)
  \end{equation*}
  be the filtered colimit of the finite tensor products for each $N \in \Fin(\Lambda)$.
  For $\gamma \in \Lambda$,
  we denote by $\iota_\gamma : \mathcal{O}(G) \to \mathcal{O}(G)^{\otimes \Lambda}$
  the map $t \mapsto \cdots \otimes 1 \otimes t \otimes 1 \otimes \cdots$
  where $t$ is in the $\gamma$-th factor.
  We denote by $\mu : 
  \mathcal{O}(G)^{\otimes \Lambda} \otimes \mathcal{O}(G)^{\otimes \Lambda} \to \mathcal{O}(G)^{\otimes\Lambda}$
  the multiplication map,
  and by $\eta : k \to {\mathcal{O}(G)^{\otimes \Lambda}}$ the unit map.
  Let $I_{\mathrm{mul}}$ and $I_{\mathrm{unit}}$ be the ideals of $\mathcal{O}(G)^{\otimes \Lambda}$
  defined by
  \begin{align*}
    I_{\mathrm{mul}} &\coloneqq \langle \iota_{\alpha \gg \beta} (t) - (\mu \circ 
      (\iota_\beta \otimes \iota_\alpha) \circ \Delta)(t) \mid 
      x,y,z \in S,\, \alpha : x \to y,\, \beta : y \to z, \, 
      t \in \mathcal{O}(G) \rangle, \\
    I_{\mathrm{unit}} &\coloneqq \langle \iota_{\mathrm{id}_x} (t) - (\eta \circ \varepsilon) (t) \mid 
      x \in S,\, t \in \mathcal{O}(G) \rangle.
  \end{align*}
  We can see that the algebra $\mathcal{O}(G)^{\otimes \Lambda} / (I_{\mathrm{mul}} + I_{\mathrm{unit}})$
  represents $F$.
\end{proof}

\begin{lemma}\label{lemma:loc G as quotient}
  We have an equivalence of stacks:
  \begin{equation}
    \label{eq:loc G as quotient}
    \Loc_{G, M} \simeq [\Loc_{G, M, S} \mathrel{/} G^S].
  \end{equation}
  The right-hand side in \eqref{eq:loc G as quotient} is a quotient stack,
  where the action of $G^S$ on $\Loc_{G, M, S}$
  is given by changing the markings.
  We also have a $2$-isomorphism:
  \begin{equation}
    \begin{tikzpicture}[scale=1.3, commutative diagrams/every diagram, baseline={(current bounding box.center)}]
      \node (a) at (0,0) {$\Loc_{G, M, S}$};
      \node (b) at (2,0) {$\Loc_{G, M}$};
      \node (d) at (2,-1) {$[\Loc_{G, M, S} \mathrel{/} {G^S}]$};
      \draw[morphism] (a) to node[above] {} (b);
      \draw[morphism] (b) to node[right] {} (d);
      \draw[morphism] (a) to node[below] {} (d);
      \node at (barycentric cs:a=1,b=1,d=1) {$\scriptstyle \cong$};
    \end{tikzpicture}
  \end{equation}
  where the vertical arrow is the equivalence in \eqref{eq:loc G as quotient},
  and the top horizontal arrow is the $1$-morphism forgetting the markings.
\end{lemma}
\begin{proof}
  This follows because,
  for any $A \in \Alg_k$ and any $P \in \Loc_{G, M}(A)$,
  the fiber of the forgetful map $\Loc_{G, M, S}(A) \to \Loc_{G, M}(A)$
  over $P$ is a $G(A)^S$-torsor, that is,
  two markings of $P$ differ by a unique element of $G(A)^S$.
\end{proof}

\begin{corollary}
  $\Loc_{G, M}$ is an algebraic stack.
\end{corollary}

\begin{remark}\label{rem:loc as rep/G}
When $M$ is connected, the automorphism group of $\Pi_1 M$ at a point $x_0 \in M$
is isomorphic to $\pi_1(M, x_0)^{\mathrm{op}}$, the opposite group of the fundamental group.
The opposite here ensures that the composition of a fundamental groupoid is compatible with
a common convention of the multiplication of a fundamental group that reads loops from left to right,
that is, $\gamma_1 \cdot \gamma_2 = \gamma_2 \circ \gamma_1$ for loops $\gamma_1, \gamma_2 : x_0 \to x_0$.
When $S$ consists of one point $x_0$,
Lemma~\ref{lemma:monodromy of Loc G M S} says that
the scheme $\Loc_{G, M, S}$ is isomorphic to the representation scheme
$\underline{\Hom} (\pi_1(M, x_0)^{\mathrm{op}}, G)$,
which represents the functor
\begin{equation}\label{eq:rep scheme}
  A\mapsto \Hom_{\mathsf{Grp}} (\pi_1(M, x_0)^{\mathrm{op}}, G(A)),
  \quad
  (f : A \to B) \mapsto (\rho \mapsto G(f) \circ \rho).
\end{equation}
By Lemma \ref{lemma:loc as quotient}, we have an equivalence of stacks
\begin{equation}\label{eq:char_stack}
  \Loc_{G, M} \simeq [\underline{\Hom} (\pi_1(M, x_0)^{\mathrm{op}},G) \mathrel{/} G ]
\end{equation}
where $G$ acts on $\underline{\Hom} (\pi_1(M, x_0)^{\mathrm{op}},G)$ by conjugation.
When $G$ is linearly reductive,
the GIT quotient $\underline{\Hom} (\pi_1(M, x_0)^{\mathrm{op}},G) \sslash G$
is called the character scheme.
See \cite{Sikora} for the character scheme (called the character variety therein),
and its applications to $3$-dimensional topology. The quotient stack \eqref{eq:char_stack} is called the \emph{character stack}, for example in \cite{BZBJ}.
\end{remark}

\begin{example}
Let $G=\SL_2$, whose coordinate ring is given by
\begin{align*}
    \cO(G) = k[x_{11},x_{12},x_{21},x_{22}]/ ( x_{11}x_{22} - x_{12}x_{21} - 1).
\end{align*}
Let $M=T^2$ be a torus, and $S\coloneq\{x_0\}$ be a singleton. Then $\Loc_{G,M,S}$ is an affine scheme with coordinate ring generated by $a_{11},a_{12},a_{21},a_{22},b_{11},b_{12},b_{21},b_{22}$, subject to the relations $a_{11}a_{22} - a_{12}a_{21} = 1$, $b_{11}b_{22} - b_{12}b_{21} = 1$, and 
\begin{align*}
    \mtx{a_{11} & a_{12} \\ a_{21} & a_{22}}\mtx{b_{11} & b_{12} \\ b_{21} & b_{22}} = \mtx{b_{11} & b_{12} \\ b_{21} & b_{22}}\mtx{a_{11} & a_{12} \\ a_{21} & a_{22}}.
\end{align*}
Here $a_{ij}$ and $b_{ij}$ correspond to the matrix entries of the monodromies along the generators of $\pi_1(M,x_0) \cong \bZ^2$. 
\end{example}

\subsection{Moduli stack of linear local systems}
\label{sec:moduli of lin local systems}
We denote by $\mathsf{proj}_{\GL}$ the stack of finitely generated projective modules,
and by $\mathsf{proj}_{\SL}$ the stack of finitely generated projective modules with volume forms.
Explicitly, for a $k$-algebra $A$, the groupoids are given as follows:
\begin{enumerate}
  \item $\mathsf{proj}_{\GL}(A)$ is the category whose objects are finitely generated projective $A$-modules,
    and morphisms are isomorphisms between them,
  \item $\mathsf{proj}_{\SL}(A)$ is the category of finitely generated projective $A$-modules $V$ equipped with
    a volume form, namely
    an isomorphism $\mathsf{vol}_V : \det V \cong A$,
    and a morphism from $(V, \mathsf{vol}_V)$ to $(W, \mathsf{vol}_W)$ is an isomorphism $f : V \to W$
    such that $\mathsf{vol}_W \circ \det f = \mathsf{vol}_V$.
\end{enumerate}

Let $M$ be a finite CW complex.
The following definition is a parallel to Definition~\ref{def:moduli of G-local systems}.

\begin{definition}[Moduli stack of local systems]
  \label{def:moduli of lin local systems}
  We define the moduli stack $\Loc_{\GL, M}$ of local systems on $M$,
  the moduli stack $\Loc_{\SL, M}$ of local systems on $M$ with volume forms as 
  the following mapping stacks:
  \begin{equation}\label{eq:loc stack}
    \Loc_{\GL, M} = \iHom (\underline{\Pi_1 M}, \mathsf{proj}_{\GL}), \qquad
    \Loc_{\SL, M} = \iHom (\underline{\Pi_1 M}, \mathsf{proj}_{\SL}).
  \end{equation}
\end{definition}

\begin{remark}\label{rem:lin loc as GL-loc}
  For each $n\geq 0$, let $\mathsf{proj}_{\GL}^{(n)} \subset \mathsf{proj}_{\GL}$
  and $\mathsf{proj}_{\SL}^{(n)} \subset \mathsf{proj}_{\SL}$
  denote the substacks of objects of rank $n$.
  Then we have a canonical equivalences $\mathsf{proj}_{\GL}^{(n)} \simeq \BG{\GL_n}$ and
  $\mathsf{proj}_{\SL}^{(n)} \simeq \BG{\SL_n}$.
  For $V \in \mathsf{proj}_{\GL}^{(n)}(A)$,
  the corresponding $(\GL_n)_A$-torsor
  is given by the set of bases of $V \otimes_A B$ for each $f : A \to B$.
  For $(V, \mathsf{vol}_V) \in \mathsf{proj}_{\SL}^{(n)}(A)$,
  the corresponding $(\SL_n)_A$-torsor
  is given by the set of bases of $V \otimes_A B$ whose wedge product 
  is mapped to $1$ via the volume form for each $f : A \to B$.
  Thus, for a connected $M$, we have decompositions 
  \begin{equation}
    \Loc_{\GL, M} \simeq \coprod\nolimits_{n \in \mathbb{N}} \Loc_{\GL_n, M},
    \qquad
    \Loc_{\SL, M} \simeq \coprod\nolimits_{n \in \mathbb{N}} \Loc_{\SL_n, M}.
  \end{equation}
\end{remark}

\begin{definition}[Moduli stack of marked linear local systems]
  \label{def:moduli of marked lin local systems}
  Let $S \subset M$ be a finite set such that each connected component of $M$ contains at least one point of $S$.
  Let $n=(n_x)_{x \in S} \in \mathbb{N}^S$. 
  We define the marked version of the moduli stacks defined in Definition \ref{def:moduli of lin local systems} 
  as functors
  \begin{align}
    \Loc_{\GL, M, S, n},\ \Loc_{\SL, M, S, n} : \Alg_k \to \Grpd,
  \end{align}
  which are defined as follows:
  for any $k$-algebra $A$, we define the following categories:
  \begin{enumerate}
    \item An object of $\Loc_{\GL, M, S, n}(A)$ consists of 
      an object $\mathcal{L}$ of $\Loc_{\GL, M}(A)$ and
      a basis $b_x = (b_{x, 1}, \dots, b_{x, n_x})$ of $\mathcal{L}_x$ for any $x \in S$.
      A morphism from $(\mathcal{L}, b)$ to $(\mathcal{L}', b')$ is a morphism $\varphi : \mathcal{L} \to \mathcal{L}'$
      in $\Loc_{\GL, M}(A)$ such that $\varphi_x \circ \lin b_x = \lin b'_x$ for any $x \in S$.
    \item An object of $\Loc_{\SL, M, S, n}(A)$ consists of 
      an object $\mathcal{L}$ of $\Loc_{\SL, M}(A)$ and
      a basis $b_x = (b_{x, 1}, \dots, b_{x, n_x})$ of $\mathcal{L}_x$ for any $x \in S$
      such that $\mathsf{vol}_{\mathcal{L}_x} (\wedge b_x) = 1$.
      A morphism from $(\mathcal{L}, b)$ to $(\mathcal{L}', b')$ is a morphism $\varphi : \mathcal{L} \to \mathcal{L}'$
      in $\Loc_{\SL, M}(A)$ such that $\varphi_x \circ \lin b_x = \lin b'_x$ for any $x \in S$.
  \end{enumerate}
  We also define 
  \begin{align}
    \Loc_{\GL, M, S} \coloneqq \coprod\nolimits_{n \in \mathbb{N}^S} \Loc_{\GL, M, S, n}, \quad
    \Loc_{\SL, M, S} \coloneqq \coprod\nolimits_{n \in \mathbb{N}^S} \Loc_{\SL, M, S, n}.
  \end{align}
\end{definition}

Lemmas \ref{lemma:no auto}--\ref{lemma:loc as quotient} below can be proved analogously to the corresponding lemmas 
in the previous section (or can be reduced to them via Remark \ref{rem:lin loc as GL-loc}).

\begin{lemma}\label{lemma:no auto}
  Any object in $\Loc_{\GL, M, S, n}(A)$ has a trivial automorphism group.
  The same assertion holds for $\Loc_{\SL, M, S, n}(A)$.
\end{lemma}

\begin{lemma}[Monodromy correspondence]\label{lemma:monodromy}
We have bijections
\begin{alignat}{1}
  \label{eq:mon GL}
  \mon_{\GL, M}(A) : \Obj (\Loc_{\GL, M, S}(A)) / \mathord{\cong} &\xlongrightarrow{\cong}
  \Obj \left(\Hom_{\Grpd} \left(S,\ \coprod\nolimits_{n \in \mathbb{N}} \bullet_{\GL_n(A)}\right)\right) \\
  \label{eq:mon SL}
  \mon_{\SL, M}(A) : \Obj (\Loc_{\SL, M, S}(A)) / \mathord{\cong} &\xlongrightarrow{\cong} 
  \Obj \left(\Hom_{\Grpd} \left(S,\ \coprod\nolimits_{n \in \mathbb{N}} \bullet_{\SL_n(A)}\right)\right)
\end{alignat}
\end{lemma}

\begin{lemma}
  The maps $\mon_{\SL, M}(A)$ and $\mon_{G, M}(A)$
  are natural in $A$ with respect to the extension of scalars.
\end{lemma}

\begin{lemma}
  The presheaves $\Loc_{\GL, M, S}$ and $\Loc_{\SL, M, S}$
  are represented by $k$-algebras.
\end{lemma}

For a finite set $S \subset M$ and $n \in \mathbb{N}^S$,
we define
\begin{align}
  \GL_n^S \coloneqq \prod_{x \in S} \GL_{n_x}, \quad
  \SL_n^S \coloneqq \prod_{x \in S} \SL_{n_x}.
\end{align}

\begin{lemma}\label{lemma:loc as quotient}
  We have equivalences of stacks
  \begin{align}
    \label{eq:loc GL as quotient}
    \Loc_{\GL, M} &\simeq 
      \coprod\nolimits_{n \in \mathbb{N}^S}
      [\Loc_{\GL, M, S, n} \mathrel{/} \GL_n^S] \\
    \label{eq:loc SL as quotient}
    \Loc_{\SL, M} &\simeq 
      \coprod\nolimits_{n \in \mathbb{N}^S}
      [\Loc_{\SL, M, S, n} \mathrel{/} \SL_n^S] 
  \end{align}
  The group actions for the quotient stacks in the right-hand sides in \eqref{eq:loc GL as quotient}--\eqref{eq:loc G as quotient}
  are given by changing the markings.
  We also have 2-isomorphisms:
  \begin{equation}
    \begin{tikzpicture}[scale=1.3, commutative diagrams/every diagram, baseline={(current bounding box.center)}]
      \node (a) at (0,0) {$\Loc_{\GL, M, S}$};
      \node (b) at (2,0) {$\Loc_{\GL, M}$};
      \node (d) at (2,-1) {$\coprod_n [\Loc_{\GL, M, S, n} \mathrel{/} {\GL_n^S}]$};
      \draw[morphism] (a) to node[above] {} (b);
      \draw[morphism] (b) to node[right] {} (d);
      \draw[morphism] (a) to node[below] {} (d);
      \node at (barycentric cs:a=1,b=1,d=1) {$\scriptstyle \cong$};
    \end{tikzpicture}
    \begin{tikzpicture}[scale=1.3, commutative diagrams/every diagram, baseline={(current bounding box.center)}]
      \node (a) at (0,0) {$\Loc_{\SL, M, S}$};
      \node (b) at (2,0) {$\Loc_{\SL, G, M}$};
      \node (d) at (2,-1) {$\coprod_n [\Loc_{\SL, M, S, n} \mathrel{/} {\SL_n^S}]$};
      \draw[morphism] (a) to node[above] {} (b);
      \draw[morphism] (b) to node[right] {} (d);
      \draw[morphism] (a) to node[below] {} (d);
      \node at (barycentric cs:a=1,b=1,d=1) {$\scriptstyle \cong$};
    \end{tikzpicture}
  \end{equation}
  where the vertical arrows are the equivalences,
  and the top horizontal arrows are the 1-morphisms forgetting the markings.
\end{lemma}

\begin{corollary}
  $\Loc_{\GL, M}$ and $\Loc_{\SL, M}$ are algebraic stacks.
\end{corollary}

\subsection{Linear representations of algebraic groups}
\label{sec:representations of algebraic groups}
Let $G$ be a smooth affine algebraic group over $k$.
We denote by $\bullet_{G} : \Alg_k \to \Grpd$ the pseudofunctor defined by
$\bullet_{G}(A) \coloneqq \bullet_{G(A)}$.
We define the categories of linear representations of $G$ as
\begin{alignat}{3}
  \label{eq:rep GL}
  \Rep_{\GL}(G) &\coloneqq \Hom (\BG{G}, \mathsf{proj_{\GL}}) 
  &&\simeq \Hom (\bullet_{G}, \mathsf{proj_{\GL}}), \\
  \label{eq:rep SL}
  \Rep_{\SL}(G) &\coloneqq \Hom(\BG{G}, \mathsf{proj_{\SL}}) 
  &&\simeq \Hom (\bullet_{G}, \mathsf{proj_{\SL}}).
\end{alignat}
The equivalences in \eqref{eq:rep GL} and \eqref{eq:rep SL} follow from the fact that
$\BG{G}$ is a stackification of $\bullet_{G}$.
We see that
giving an object of $\Rep_{\GL}(G)$ (resp. $\Rep_{\SL}(G)$) is equivalent to giving 
a pair $(V, \rho)$ (resp. a triple $(V, \rho, \omega)$) where
$V$ is finitely generated projective $k$-module, and $\rho$ is a 
family of group homomorphisms $\rho_{A} : G(A) \to \GL_A (V \otimes_k A)$
(resp. $\rho_{A} : G(A) \to \SL_A (V \otimes_k A)$)
for each $k$-algebra $A$ such that the left (resp. the right) diagram
\begin{equation}\label{eq:rep naturality}
\begin{tikzpicture}[scale=1.3, commutative diagrams/every diagram, baseline={(current bounding box.center)}]
    \node (A) at (0,1) {$G(A)$};
    \node (A') at (2,1) {$\GL_A (V \otimes_k A)$};
    \node (B) at (0,0) {$G(B)$};
    \node (B') at (2,0) {$\GL_B (V \otimes_k B)$};
    \draw[morphism] (A) -- (A') node[midway, above] {};
    \draw[morphism] (B) -- (B') node[midway, above] {};
    \draw[morphism] (A) -- (B) node[midway, left] {$G(f)$};
    \draw[morphism] (A') -- (B') node[midway, right] {};
\end{tikzpicture}\qquad
\begin{tikzpicture}[scale=1.3, commutative diagrams/every diagram, baseline={(current bounding box.center)}]
  \node (A) at (0,1) {$G(A)$};
  \node (A') at (2,1) {$\SL_A (V \otimes_k A)$};
  \node (B) at (0,0) {$G(B)$};
  \node (B') at (2,0) {$\SL_B (V \otimes_k B)$};
  \draw[morphism] (A) -- (A') node[midway, above] {};
  \draw[morphism] (B) -- (B') node[midway, above] {};
  \draw[morphism] (A) -- (B) node[midway, left] {$G(f)$};
  \draw[morphism] (A') -- (B') node[midway, right] {};
\end{tikzpicture}
\end{equation}
commutes for each $k$-algebra morphism $f:A \to B$.
Here, for a $A$-module $V$ with a volume form $\omega  : \det V \to A$, 
we define
\begin{equation*}
  \SL_A (V) \coloneqq \{ g \in \GL_A (V) \mid \omega \circ \det g = \omega \}.
\end{equation*}

\begin{definition}\label{def:associated lin loc}
  We have a functor
  \begin{equation}\label{eq:associated lin loc}
    \begin{split}
      (-) \times^G (-) : \Rep_{\GL}(G) &\to \Hom (\Loc_{G, M}, \Loc_{\GL, M}), \\
      V &\mapsto ((A, P) \mapsto P \times^G V\coloneq V \circ P) \\
      (\alpha : V \to W) &\mapsto ((A, P) \mapsto \alpha \boxminus \id_P)
    \end{split}
  \end{equation}
  where we regard 
  $V$ as an object of $\Hom(\BG{G}, \mathsf{proj}_{\GL})$, and
  $P$ as an object of $\Hom(\underline{\Pi_1 M} \times A, \BG{G})$.
  The symbol $\boxminus$ in \eqref{eq:associated lin loc} denotes the horizontal composition of $2$-morphisms.
  The $2$-isomorphism 
  \begin{equation}
    \begin{tikzpicture}[scale=1.5, commutative diagrams/every diagram, baseline={(current bounding box.center)}]
      \node (a) at (0,0) {$\Loc_{G, M} (A)$};
      \node (b) at (2,0) {$\Loc_{\GL, M} (A)$};
      \node (d) at (2,-1) {$\Loc_{\GL, M} (B)$};
      \node (c) at (0,-1) {$\Loc_{G, M} (B)$};
      \draw[morphism] (a) to node[above] {$(-)\times^G V$} (b);
      \draw[morphism] (b) to node[right] {$(-) \otimes_A B$} (d);
      \draw[morphism] (a) to node[left] {$(-) \otimes_A B$} (c);
      \draw[morphism] (c) to node[below] {$(-)\times^G V$} (d);
      \node at (barycentric cs:a=1,b=1,c=1,d=1) {$\scriptstyle \cong$};
    \end{tikzpicture}
  \end{equation}
  is given by the corresponding 2-isomorphism for each point $x \in M$.
  Similarly, we have a functor
  \begin{equation}
    \begin{split}
      (-) \times^G (-) : \Rep_{\SL}(G) &\to \Hom (\Loc_{G, M}, \Loc_{\SL, M})
    \end{split}
  \end{equation}
  defined in the same way.
\end{definition}

\begin{definition}\label{def:associated marked lin loc}
  Let $V \in \Rep_{\GL}(G)$,
  and let $v = (v_i)_{i=1}^{n}$ be a basis of $V$.
  We define a map 
  \begin{equation}
    (-) \times^G (V, v) : \Loc_{G, M, S} \to \Loc_{\GL, M, S, n}
  \end{equation}
  where in the right-hand side we regard $n$ as the constant map $n : S \to \mathbb{N}$ with value $n$,
  as follows.
  For a marked $G$-local system $(P, p) \in \Loc_{G, M, S}(A)$,
  we send it to the marked linear local system
  $(P \times^G V, b)$,
  where the basis $b_x$ of $(P \times^G V)_x$
  at each $x \in S$ is defined as follows.
  The marking $p_x \in P_x(A)$ induces an isomorphism of 
  $G_A$-torsor $P_x \cong G_A$,
  which in turn induces an isomorphism of $A$-modules
  $(P \times^G V)_x \cong V \otimes_k A$.
  The basis $b_x$ is defined as the one corresponding to the basis $(v_i \otimes 1 )_{i=1}^{n}$
  of $V \otimes_k A$ via this isomorphism.
  Similarly, for $V \in \Rep_{\SL}(G)$ and a basis $(v_i)_{i=1}^{n}$ of $V$
  such that $\mathsf{vol}_V (\wedge_{i=1}^n v_i) = 1$,
  we have a map 
  \begin{equation}
    (-) \times^G (V, v) : \Loc_{G, M, S} \to \Loc_{\SL, M, S, n}
  \end{equation}
  defined in the same way.
\end{definition}

\subsection{Twisted complex and geometric bases}
\label{sec:twisted complex and geometric bases}
Let $M$ be a Hausdorff space. Recall that an $n$-cell in $M$ is a pair $(e,\varphi)$, where $e \subset M$ is a subset, and $\varphi: D^n \to M$ is a continuous map whose restriction to $\interior D^n$ gives a homeomorphism onto $e$. We call $\varphi$ the characteristic map of $e$. We call $n$ the dimension of the cell $(e,\varphi)$.
A finite CW structure on $M$ is a finite collection $\{(e_\sigma,\varphi_\sigma)\}_{\sigma \in \mathsf{cell}}$ of cells satisfying the conditions
\begin{itemize}
    \item $M=\bigsqcup_{\sigma \in \mathsf{cell}} e_\sigma$;
    \item $\varphi_\sigma(\partial D^{\dim \sigma}) \subset M^{(\dim \sigma -1)}$ for any $\sigma \in \mathsf{cell}$, where $\dim \sigma$ denotes the dimension of $(e_\sigma,\varphi_\sigma)$, and $M^{(n)}\coloneq \bigsqcup\{ e_\tau \mid \dim \tau \leq n\}$ denotes the $n$-skeleton of $M$.
\end{itemize}
We call $(M,\{(e_\sigma,\varphi_\sigma)\}_{\sigma \in \mathsf{cell}})$ a finite CW complex. 
Let $\mathsf{cell}_i\coloneq\{ \sigma \in \mathsf{cell} \mid \dim \sigma = i\}$.
If the CW structure is clear from the context, we simply call $M$ a finite CW complex. 

Recall that the cellular chain complex $C_\bullet(M;\mathcal{L})$ with coefficient $\mathcal{L}$ is defined as the relative singular homology, as 
follows. Define
\begin{equation}
  C_k(M; \mathcal{L}) \coloneq H_k^{\mathrm{sing}}(M^{(k)},M^{(k-1)};\mathcal{L})
\end{equation} 
for $k \geq 0$. 
The boundary map
\begin{align*}
    \partial_k: C_k(M;\mathcal{L}) \to C_{k-1}(M;\mathcal{L})
\end{align*}
is defined to be the connecting homomorphism for the short exact sequence
\begin{align*}
    0 \to C_k^{\mathrm{sing}}(M^{(k-1)},M^{(k-2)};\mathcal{L}) \to C_k^{\mathrm{sing}}(M^{(k)},M^{(k-2)};\mathcal{L}) \to C_k^{\mathrm{sing}}(M^{(k)},M^{(k-1)};\mathcal{L}) \to 0.
\end{align*}

The chain complex $C_\bullet(M;\mathcal{L})$ can be described more explicitly, as follows \cite[Section 9]{Steenrod}. 
For each $\sigma \in \mathsf{cell}_i$, the characteristic map induces 
\begin{align*}
    \varphi_{\sigma,\ast}: H_i^{\mathrm{sing}}(D^i,S^{i-1};\mathcal{L}) \to C_i(M;\mathcal{L}),
\end{align*}
whose images form a basis. 
Let $\varodot_\sigma\coloneq\varphi_\sigma(0) \in M$ denote the central point of $e_\sigma$. Since $D^i$ is contractible, we have $H_i^{\mathrm{sing}}(D^i,S^{i-1};\mathcal{L}) \cong H_i^{\mathrm{sing}}(D^i,S^{i-1};\bZ) \otimes \mathcal{L}_{\varodot_\sigma} \cong \mathcal{L}_{\varodot_\sigma}$. Here, the latter isomorphism is given by the fundamental class $[D^i] \in H_i^{\mathrm{sing}}(D^i,S^{i-1};\bZ)$. 
Therefore we get 
\begin{alignat}{2}
  &C_i(M; \mathcal{L}) \cong \bigoplus_{\sigma \in \mathsf{cell}_i} \mathcal{L}_{\varodot_\sigma}.
\end{alignat}
We will sometimes write $\mathcal{L}_\sigma$ instead of $\mathcal{L}_{\varodot_\sigma}$ for simplicity.
To describe the differential,
consider the map
\begin{align}\label{eq:attaching_sphere}
  S^{i-1} \xrightarrow{\varphi_\tau|_{S^{i-1}}} M^{(i-1)} \hookrightarrow (M^{(i)},M^{(i-1)}) \xrightarrow{\pi_\sigma} (S^{i-1},\ast),
\end{align}
for $\tau \in \mathsf{cell}_i$ and $\sigma \in \mathsf{cell}_{i-1}$,
where $\pi_\sigma$ is the unique map such that
\begin{equation*}
\begin{tikzcd}
   & (S^{i-1},\ast)= (D^{i-1}/\partial D^{i-1}, \ast) \ar[d,"\approx"] \\
  (M^{(i-1)},M^{(i-2)}) \ar[ur,"\pi_\sigma"] \ar[r] & (M^{(i-1)}/ (M^{(i-2)} \cup \bigcup_{\nu \in \mathsf{cell}_{i-1} \setminus \{\sigma\}} \varphi_\nu(D^{i-1})), \ast)
\end{tikzcd}
\end{equation*}
commutes. The vertical map is induced by $\varphi_\sigma$, and the horizontal map is the quotient map. See \cref{fig:attaching_sphere}. 
\begin{figure}[t]
  \centering
\begin{tikzpicture}
\begin{scope}[xshift=-5cm]
\draw[->-] (-1,0,0) arc(-180:0:1cm and 0.3cm);
\draw[dashed] (-1,0,0) arc(180:0:1cm and 0.3cm);
\draw (-1,0,0) arc(180:0:1cm and 0.8cm);
\node at (0,1.1,0) {$D^i$};
\draw[->,thick] (1.3,0.2,0) --node[midway,above]{$\varphi_\tau$} ++(2.5,0,0);
\end{scope}

\draw(-1,0,0) -- (0,0,1);
\draw[red,thick,->-={0.7}{}] (0,0,1) --node[midway,below]{$\sigma$} (1,0,-1);
\draw[dashed] (-1,0,0) -- (1,0,-1);
\draw(2,0,-2) -- (2,0,2) node[right]{$M^{(i-1)}$} -- (-2,0,2) -- (-2,0,-2) --cycle;
\draw (-1,0,0) to[bend left=70]node[midway,above]{$\tau$} (1,0,-1);
\foreach \i in {{-1,0,0},{0,0,1},{1,0,-1}} \fill (\i) circle(1.5pt);

\begin{scope}[xshift=6cm]
\draw[red,thick,->-={0.8}{}] (-1,0,0) ellipse(1cm and 0.3cm);
\fill (-2,0,0) circle(1.5pt);
\node[red] at (0.5,0,0) {$S^{i-1}$};
\draw[red,<-,thick] (-2.2,0,0) --node[midway,below]{$\pi_\sigma$} ++(-2.5,0,0);
\end{scope}

\end{tikzpicture}
  \caption{The map \eqref{eq:attaching_sphere}.}
  \label{fig:attaching_sphere}
\end{figure}
Then the differential $\partial_i=\partial_i(M; \mathcal{L})$ is written as
\begin{align}\label{eq:twisted_differential}
  \partial_i : C_i(M; \mathcal{L}) \to C_{i-1}(M; \mathcal{L}), \quad
  \partial_i f (\sigma) =
  \sum_{\tau \in \mathsf{cell}_{i}} \sum_{p \in S_{\tau,\sigma}} \varepsilon_p \cdot \mathcal{L}_{\varodot_\tau \to p} (f(\tau))
\end{align}
for $f=(f(\tau))_{\tau \in \mathsf{cell}_i} \in C_i(M,\mathcal{L})$ and $\sigma \in \mathsf{cell}_{i-1}$. 
Here,
\begin{itemize}
    \item
    We assume that the map \eqref{eq:attaching_sphere} for $\tau$ and $\sigma$
    is a local homeomorphism for any $\tau \in \mathsf{cell}_i$, and
    we write $S_{\tau,\sigma} \coloneqq \varphi_\tau^{-1}(\varodot_\sigma) \subset \partial D^i$,
    which is a finite set.
    \item $\mathcal{L}_{\varodot_\tau \to p}$ denotes the parallel-transport of the local system $\varphi_\tau^\ast \mathcal{L}$ along the path from the center to $p$ in $D^i$ unique up to homotopy. 
    \item $\varepsilon_p = \pm 1$ is the local mapping degree at $p$ of the map \eqref{eq:attaching_sphere}.
\end{itemize}
See \cref{ex:attaching} for a typical situation. 

When $\mathcal{L}$ is a trivial local system, the formula \eqref{eq:twisted_differential} recovers the differential for the untwisted chain complex by the localization formula of mapping degree: the incidence number is given by $[\tau:\sigma]=\sum_{p \in S_{\tau,\sigma}} \varepsilon_p$.
The cellular cochain complex with coefficient $\mathcal{L}$ is introduced in the dual way. 
We have 
\begin{equation*}
  C^i(M; \mathcal{L}) \cong \prod_{\sigma \in \mathsf{cell}_i} \mathcal{L}_{\varodot_\sigma},
\end{equation*}
and the differential $d^i=d^i(M,\mathcal{L})$ is described as 
\begin{align*}
    d^i : C^i(M; \mathcal{L}) \to C^{i+1}(M; \mathcal{L}), \quad
  d^i f (\sigma) =
  \sum_{\tau \in \mathsf{cell}_{i}}
  \sum_{p \in S_{\sigma,\tau}} \varepsilon_p \cdot \mathcal{L}_{p \to \varodot_\sigma} (f(\tau))
\end{align*}
for $f \in C^i(M,\mathcal{L})$ and $\sigma \in \mathsf{cell}_{i+1}$, under the same assumption on the map \eqref{eq:attaching_sphere}. 

The $i$-th \emph{twisted homology} $H_i(M; \mathcal{L})$ 
and the $i$-th \emph{twisted cohomology} $H^i(M; \mathcal{L})$ 
are the $A$-modules defined as
\begin{align}
  &H_i(M; \mathcal{L}) \coloneqq \Ker (\partial_i (M; \mathcal{L})) / \Img (\partial_{i+1} (M; \mathcal{L})) \\
  &H^i(M; \mathcal{L}) \coloneqq \Ker (d^i (M; \mathcal{L})) / \Img (d^{i-1} (M; \mathcal{L})).
\end{align}
The relative versions $H_i (M, A; \mathcal{L})$ and $H^i (M, A; \mathcal{L})$
for a CW pair $(M, A)$ are defined in a similar way. 

\begin{example}\label{ex:attaching}
Consider an annulus $M$, whose CW structure is given by identifying an opposite sides of a square. Let $\tau$ be the identified edge ($1$-cell), and $\sigma$ the unique $2$-cell with attaching map $\varphi_\sigma: D^2 \to M$. See \cref{fig:attach_example}. 
Then $S_{\sigma,\tau}$ consists of two points $p,p'$, which have local degrees $\varepsilon_p=+1$ and $\varepsilon_{p'}=-1$, respectively. The formula \eqref{eq:twisted_differential} becomes
\begin{align*}
    \partial_2 f (\sigma) = \mathcal{L}_{\varodot_\tau \to p}(f(\tau)) - \mathcal{L}_{\varodot_\tau \to p'}(f(\tau)) = v - g(v),
\end{align*}
where $v\coloneq\mathcal{L}_{\varodot_\tau \to p}(f(\tau))$ and $g\coloneq\mathrm{mon}_{\mathcal{L}}(\gamma)$.   
\end{example}

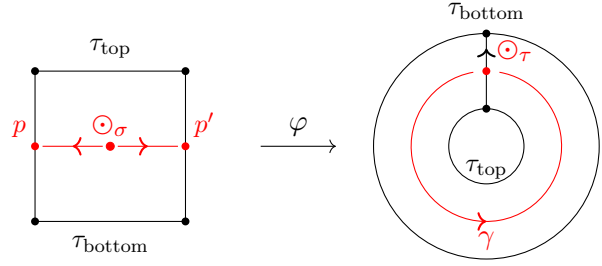
\begin{figure}[ht]
    \centering
\begin{tikzpicture}
\draw (-1,-1) -- (1,-1) -- (1,1) -- (-1,1) --cycle;
\foreach \i in {-1,1} \foreach \j in {-1,1} \fill (\i,\j) circle(1.5pt);
\fill[red] (1,0) circle(1.5pt) node[above right,scale=0.9]{$p'$};
\fill[red] (-1,0) circle(1.5pt) node[above left,scale=0.9]{$p$};
\filldraw[red] (0,0) circle(1.5pt) node[above]{$\varodot_\sigma$};
\draw[red,->-] (0.1,0) -- (0.9,0);
\draw[red,->-] (-0.1,0) -- (-0.9,0);
\node[scale=0.9] at (0,1.3) {$\tau_\mathrm{top}$};
\node[scale=0.9] at (0,-1.3) {$\tau_\mathrm{bottom}$};
\draw[->] (2,0) --node[midway,above]{$\varphi$} ++(1,0);
\begin{scope}[xshift=5cm]
\draw(0,0) circle(1.5cm);
\draw(0,0) circle(0.5cm);
\draw[->-={0.8}{}] (0,0.5) -- (0,1.5);
\fill (0,1.5) circle(1.5pt);
\fill (0,0.5) circle(1.5pt);
\fill[red] (0,1) circle(1.5pt) node[above right]{$\varodot_\tau$};
\node[scale=0.9] at (0,-0.3) {$\tau_\mathrm{top}$};
\node[scale=0.9] at (0,1.8) {$\tau_\mathrm{bottom}$};
\draw[red,->-] (100:1) arc(100:360+80:1) node[midway,below]{$\gamma$};
\end{scope}
\end{tikzpicture}
    \caption{An example of $2$-dimensional CW copmlex.}
    \label{fig:attach_example}
\end{figure}

\begin{remark}\label{rem:H^0_invariant_vector}
Suppose that $M$ is connected, and let $x_0 \in M$.
For a marked linear local system $\mathcal{L} \in \Loc_{\GL, M, \{x_0\}} (A)$,
we have $H^0(M; \mathcal{L}) \cong (\mathcal{L}_{x_0})^{\Img (\mon_{\mathcal{L}})}$, where $\mon_{\mathcal{L}}:\pi_1(M,x_0)^{\mathrm{op}} \to GL(\mathcal{L}_{x_0})$ is the monodromy homomorphism of $\mathcal{L}$ given by \eqref{eq:mon GL}.
\end{remark}

\begin{lemma}\label{lemma:tensor chain cochain}
  Suppose that $\mathcal{L} \in \Loc_{\GL, M} (A)$, and $f : A \to B$ is a morphism of $k$-algebras.
  We have a canonical isomorphism of $B$-modules:
  \begin{alignat}{1}
    \label{eq:tensor chain}
    &C_i(M; \mathcal{L}) \otimes_A B \cong C_i(M; \mathcal{L} \otimes_A B), \\
    \label{eq:tensor cochain}
    &C^i(M; \mathcal{L}) \otimes_A B \cong C^i(M; \mathcal{L} \otimes_A B).
  \end{alignat}
\end{lemma}
\begin{proof}
  The isomorphism \eqref{eq:tensor chain} follows from the fact that 
  the extension of scalars commutes with direct sums.
  Since $\mathsf{cell}_i$ is a finite set, 
  the direct product is the same as the direct sum, which gives
  the isomorphism \eqref{eq:tensor cochain}.
\end{proof}
The following lemma is easily verified from the definitions. 

\begin{lemma}\label{lemma:diff naturality}
  We have commutative diagrams of $B$-modules:
  \begin{align}
    \label{eq: diff naturality hain}
    &\begin{tikzpicture}[scale = 1.5, commutative diagrams/every diagram, baseline={(current bounding box.center)}]
      \node (A) at (0,1) {$C_i(M; \mathcal{L}) \otimes_A B$};
      \node (A') at (3,1) {$C_{i-1}(M; \mathcal{L}) \otimes_A B$};
      \node (B) at (0,0) {$C_i(M; \mathcal{L} \otimes_A B)$};
      \node (B') at (3,0) {$C_{i-1}(M; \mathcal{L} \otimes_A B)$};
      \draw[morphism] (A) -- (A') node[midway, above] {$\partial_i \otimes B$};
      \draw[morphism] (B) -- (B') node[midway, above] {$\partial_i$};
      \draw[morphism] (A) -- (B) node[midway, left] {$\cong$};
      \draw[morphism] (A') -- (B') node[midway, right] {$\cong$};
    \end{tikzpicture}\\
    \label{eq: diff naturality cochain}
    &\begin{tikzpicture}[scale = 1.5, commutative diagrams/every diagram, baseline={(current bounding box.center)}]
      \node (A) at (0,1) {$C^i(M; \mathcal{L}) \otimes_A B$};
      \node (A') at (3,1) {$C^{i+1}(M; \mathcal{L}) \otimes_A B$};
      \node (B) at (0,0) {$C^i(M; \mathcal{L} \otimes_A B)$};
      \node (B') at (3,0) {$C^{i+1}(M; \mathcal{L} \otimes_A B)$};
      \draw[morphism] (A) -- (A') node[midway, above] {$d^i \otimes B$};
      \draw[morphism] (B) -- (B') node[midway, above] {$d^i$};
      \draw[morphism] (A) -- (B) node[midway, left] {$\cong$};
      \draw[morphism] (A') -- (B') node[midway, right] {$\cong$};
    \end{tikzpicture}
  \end{align}
  where the vertical isomorphisms are given by Lemma \ref{lemma:tensor chain cochain}.
\end{lemma}

From Lemmas \ref{lemma:tensor chain cochain}--\ref{lemma:diff naturality},
the following definition makes sense.
\begin{definition}\label{def:chain O-module}
  The assignment $(A, \mathcal{L}) \mapsto C_i(M; \mathcal{L})$ defines a
  quasi-coherent module on $\Loc_{\GL, M}$, which we denote by $C_i (M; -)$.
  We also define a chain complex of quasi-coherent modules on $\Loc_{\GL, M}$ by:
  \begin{equation}
    C_{\bullet} (M; -)
    \coloneqq (C_i(M; -), \partial_i(M; -))_i,
  \end{equation}
  where $\partial_i(M; -) : C_i (M; -) \to C_{i-1} (M; -)$ is a morphism of $\mathcal{O}_{\Loc_{\GL, M}}$-modules
  associated with the assignment $(A, \mathcal{L}) \mapsto (\partial_i (M; \mathcal{L}) : C_i(M; \mathcal{L}) \to C_{i-1}(M; \mathcal{L}))$.
  We also define the cochain complex of quasi-coherent modules $C^{\bullet} (M; -)$ in a similar way.
\end{definition}

\begin{definition}\label{def:fixed betti moduli}
  Let $n$ be a natural number,
  and $r = (r_0, \dots, r_n)$ be a sequence of natural numbers.
  We define a pseudofunctor $\Loc_{\GL, M}^{b_* = r} : \Alg_k \to \Grpd$ by
  defining $\Loc_{\GL, M}^{b_* = r}(A)$ for each $k$-algebra $A$ as
  the full subgroupoid of $\Loc_{\GL, M}(A)$ consisting of local systems $\mathcal{L}$ satisfying 
  \begin{equation}
    \begin{split}      
      \text{$H_i(M; \mathcal{L} \otimes_A B)$ is projective of constant rank $r_i$}
    \end{split}
  \end{equation}
  for any $f : A \to B$ and any $i=0,\dots,n$.
  We also define the marked version
  \begin{equation}
    \Loc_{\GL, M, S}^{b_* = r} \coloneqq \Loc_{\GL, M, S} \times_{\Loc_{\GL, M}} \Loc_{\GL, M}^{b_* = r}. 
  \end{equation}
\end{definition}

\begin{lemma}\label{lemma:fixed betti substack}
  $\Loc_{\GL, M, S}^{b_* = r}$ is represented by a 
  locally closed subscheme of $\Loc_{\GL, M, S}$, and
  $\Loc_{\GL, M}^{b_* = r}$ is a 
  locally closed substack of $\Loc_{\GL, M}$.
\end{lemma}
\begin{proof}
  The assertion for $\Loc_{\GL, M, S}^{b_* = r}$ follows from \cite[Proposition 23.130]{GortzWedhorn2}.
  This implies the assertion for $\Loc_{\GL, M}^{b_* = r}$,
  as the condition is invariant under a change of markings.
\end{proof}

We define the \emph{$i$-th betti function} for $M$ by 
\begin{equation}
  \betti_i (M; -) : \lvert \Loc_{\GL, M} \rvert \to \mathbb{N}, \quad
  (K, \mathcal{L}) \mapsto
  \dim_{K} H_i (M; \mathcal{L}).
\end{equation}
As a set, we have
\begin{equation}
  \begin{split}
    \lvert \Loc_{\GL, M}^{b_* = r} \rvert \coloneqq  
    \{ \mathcal{L} \in \lvert \Loc_{\GL, M} \rvert \mid 
    \text{$\betti_i (M; \mathcal{L}) = r_i$ for $i=0,\dots,n$} \}.
  \end{split}
\end{equation}
by~\cite[Proposition 23.130]{GortzWedhorn2},
which justifies the notation.

We also define the following variant:
\begin{definition}\label{def:fixed coBetti moduli}
  Let $n$ be a natural number,
  and $r = (r^0, \dots, r^n)$ be a sequence of natural numbers.
  We define a locally closed substack 
  $\Loc_{\GL, M}^{b^* = r} \subset \Loc_{\GL, M}$
  by the condition that
  $\mathcal{L} \in \Loc_{\GL, M}^{b^* = r}(A)$
  if and only if 
  \begin{equation}
    \begin{split}      
      \text{$H^i(M; \mathcal{L} \otimes_A B)$ is projective of constant rank $r^i$}
    \end{split}
  \end{equation}
  for any $f : A \to B$ and any $i=0,\dots,n$.
\end{definition}

We will use \cref{def:fixed coBetti moduli} for a single sequence $(r)$,
which only specifies the rank of $H^0$.
In this case, we simply write $\Loc_{\GL, M}^{b^0 = r}$ for $\Loc_{\GL, M}^{b^* = (r)}$.

\subsection{Poincar\'{e} duality}
Let $n \geq 0$.
Let $X$ be a finite CW complex, and let $[X] \in H_n(X;\bZ)$.
We call the pair $(X,[X])$ a \emph{Poincar\'e complex of formal dimension $n$} if the map
\begin{align}\label{eq:poincare duality isom}
    \cap [X]: H^p(X;\mathcal{L}) \to H_{n-p}(X;\mathcal{L})
\end{align}
is an isomorphism for any $A \in \Alg_k$, $\mathcal{L} \in \Loc_{\GL, M}(A)$, and $p$.

There is also a pair version of Poincar\'e complexes.
Let $X$ and $Y$ be finite CW complexes such that $Y \subset X$ is a subcomplex,
and let $[X] \in H_n(X, Y; \bZ)$ and $[Y] \in H_{n-1}(Y; \bZ)$.
We call the tuple $(X, Y, [X], [Y])$ a \emph{Poincar\'e pair of formal dimension $n$} 
if the following conditions hold:
\begin{enumerate}
  \item $(X, [X])$ is a Poincar\'e complex of formal dimension $n$,
  \item $(Y, [Y])$ is a Poincar\'e complex of formal dimension $n-1$,
  \item $\partial_n [X] = [Y]$,
  \item the map
  \begin{align}\label{eq:poincare-lefschetz duality isom}
      \cap [X]: H^p(X;\mathcal{L}) \xrightarrow{\cong} H_{n-p}(X, Y;\mathcal{L}).
  \end{align}
  is an isomorphism
  for any $A \in \Alg_k$, $\mathcal{L} \in \Loc_{\GL, M}(A)$, and $p$.
\end{enumerate}

\begin{remark}
Our definitions of Poincar\'e complexes and Poincar\'e pairs are those given in \cite{Wall}, 
except that we work over $k$ rather than $\mathbb{Z}$ and consider only finitely generated projective modules. 
Furthermore, we set the orientation character $w^1(X) \in H^1(X; \mathbb{Z}_2)$ in \cite{Wall} 
to be trivial, as we restrict our attention to orientable manifolds in this paper.
\end{remark}

Given a Poincar\'e complex $M = (X, [X])$, we often write $M$ for the underlying complex $X$.
Similarly, given a Poincar\'e pair $M = (X, Y, [X], [Y])$, we often write $M$ for the 
underlying complex $X$ and $\partial M$ for the subcomplex $Y$. Moreover, we let $\mathsf{cell}(M)$ denote the index set for the CW structure of $X$.

\begin{dfn}\label{dfn:link exterior complex}
  We say that a Poincar\'e pair $M = (X, Y, [X], [Y])$ of formal dimension $3$ is 
  a \emph{link exterior complex} if 
  the Euler characteristic of each connected component of $\partial M$ is zero.
  In addition, we say that a pair $\gamma = (Z, [Z])$ is a \emph{peripheral loop} in $M$
  if it is a Poincar\'e complex of formal dimension $1$
  and $\gamma$ is a subcomplex of $\partial M$.
\end{dfn}

A typical example is given by:
\begin{example}\label{ex:3-mfd complex}
  Let $M$ be an oriented $3$-manifold with boundary consisting of $m$ tori $\partial M=\bigsqcup_{i=1}^m T_i$,
  and $\gamma \coloneqq \bigsqcup_{i=1}^m \gamma_i$ be the disjoint union of $m$ oriented simple closed curves $\gamma_i \subset T_i$ which represent non-trivial homology cycles for $i=1,\dots,m$.
  Then $(M, \partial M, [M], \partial_3 [M] )$ is a link exterior complex,
  and $(\gamma, [\gamma])$ is a peripheral loop in $M$,
  where the fundamental classes $[M]$ and $[\gamma]$ 
  are determined by the orientations of $M$ and $\gamma$, respectively.
\end{example}

\section{Reidemeister torsion on \texorpdfstring{$\Loc_{\SL, M}$}{Loc(SL,M)}}
Let $k$ be a field of characteristic $0$.
We fix an inverse structure $(-)^{\otinv}$ on $\mathcal{P}_A$ for any 
$k$-algebra $A$ pseudonaturally in $A$,
and also fix a determinant functor on $(\mathsf{proj}(A), \mathrm{iso})$ for any $k$-algebra $A$
pseudonaturally in $A$.
See \cref{sec:determinant functor} for the definitions of 
inverse structures and determinant functors.

\subsection{Geometric volume form}
\label{sec:geometric volume form}
Let $A$ be a $k$-algebra.
For a projective $A$-module $L$ of finite rank,
a morphism $\omega : \det L \to (A, \lvert L \rvert)$ in 
$\mathcal{P}_A$ is called a \emph{volume form} on $L$.
For a volume form $\omega$ on $L$,
we denote by $1_{\omega}$ the element $\omega^{-1}(1) \in \det L$.

Let $M$ be a finite CW complex.
We assume that $\mathsf{cell}_i$ is linearly ordered for any $i$.
Let $\mathcal{L} \in \Loc_{\SL, M}(A)$.
We define a volume form $\mathsf{vol}_{M, \mathcal{L}, i}$ on $\det C_i (M; \mathcal{L})$ as 
the composition
\begin{align}\label{eq:geometric volume form at i}
  \det C_i (M; \mathcal{L}) \cong \bigotimes_{\sigma \in \mathsf{cell}_i} \det \mathcal{L}_{\sigma} 
  \xrightarrow{\otimes_{\sigma \in \mathsf{cell}_i} \mathsf{vol}_{\mathcal{L}_{\sigma}}} 
  \bigotimes_{\sigma \in \mathsf{cell}_i} (A, \lvert \mathcal{L}_{\sigma} \rvert)
  \cong
  \biggl(A, \sum_{\sigma \in \mathsf{cell}_i} \lvert \mathcal{L}_{\sigma} \rvert \biggr),
\end{align}
where the last isomorphism is defined by using the 
order $\mathsf{cell}_i = \{\sigma_1, \dots, \sigma_n\}$ as follows:
\begin{align}
  \bigotimes_{\sigma \in \mathsf{cell}_i} (A, \lvert \mathcal{L}_{\sigma} \rvert) \cong 
  (A, \lvert \mathcal{L}_{\sigma_1} \rvert) \otimes \cdots \otimes (A, \lvert \mathcal{L}_{\sigma_n} \rvert) &\cong
  \biggl(A, \sum_{j=1}^n \lvert \mathcal{L}_{\sigma_j} \rvert \biggr)
\end{align}
where the second isomorphism is given by 
$1 \otimes \dots \otimes 1 \mapsto 1 : A \otimes \dots \otimes A \to A$.

Now suppose that 
$\sum_{i \in \mathbb{Z}} (-1)^i \sum_{\sigma \in \mathsf{cell}_i} \lvert \mathcal{L}_{\sigma} \rvert = 0$.
Then we define a volume form $\mathsf{vol}_{M, \mathcal{L}}$ on $\det C_\bullet (M; \mathcal{L})$ as follows:
\begin{align}\label{eq:vol complex}
  \det C_\bullet (M; \mathcal{L}) &\cong
  \bigotimes_{i : \text{even}} \det C_i (M; \mathcal{L})
  \otimes
  \bigotimes_{i : \text{odd}} \bigl( \det C_i (M; \mathcal{L}) \bigr)^{\otinv} \\
  \label{eq:vol complex cell contract}
  &\cong 
  \bigotimes_{i : \text{even}}
  \biggl(A, \sum_{\sigma \in \mathsf{cell}_i} \lvert \mathcal{L}_{\sigma} \rvert\biggr)
  \otimes
  \bigotimes_{i : \text{odd}}
  \biggl(A, \sum_{\sigma \in \mathsf{cell}_i} \lvert \mathcal{L}_{\sigma} \rvert \biggr)^{\otinv} \\
  &\cong 
  \bigotimes_{i : \text{even}}
  \biggl(A, \sum_{\sigma \in \mathsf{cell}_i} \lvert \mathcal{L}_{\sigma} \rvert\biggr)
  \otimes
  \biggl(\bigotimes_{i : \text{odd}}
  \biggl( A, \sum_{\sigma \in \mathsf{cell}_i} \lvert \mathcal{L}_{\sigma} \rvert \biggr) \biggr)^{\otinv} \\
  \label{eq:vol complex degree contract}
  &\cong
  \biggl(A, \sum_{i : \text{even}} \sum_{\sigma \in \mathsf{cell}_i} \lvert \mathcal{L}_{\sigma} \rvert \biggr)
  \otimes
  \biggl(A, \sum_{i : \text{odd}} \sum_{\sigma \in \mathsf{cell}_i} \lvert \mathcal{L}_{\sigma} \rvert \biggr)^{\otinv} \\
  \label{eq:vol complex end}
  &\cong
  1,
\end{align}
where 
\begin{itemize}
  \item \eqref{eq:vol complex cell contract} is given by $\mathsf{vol}_{M, \mathcal{L}, i}$ defined in \eqref{eq:geometric volume form at i},
  \item \eqref{eq:vol complex degree contract} is given by $ \cdots \otimes 1 \otimes 1 \otimes 1 \otimes \cdots \mapsto 1$
  for each even and odd part,
  after ordering the tensor products in the decreasing order on $i$ in the following way:
    \begin{align*}
      &\cdots \otimes \biggl(A, \sum_{\sigma \in \mathsf{cell}_4} \lvert \mathcal{L}_{\sigma} \rvert \biggr)
      \otimes
      \biggl(A, \sum_{\sigma \in \mathsf{cell}_2} \lvert \mathcal{L}_{\sigma} \rvert \biggr)
      \otimes
      \biggl(A, \sum_{\sigma \in \mathsf{cell}_0} \lvert \mathcal{L}_{\sigma} \rvert \biggr) \otimes \cdots
      \cong
      \biggl(A, \sum_{i : \text{even}} \sum_{\sigma \in \mathsf{cell}_i} \lvert \mathcal{L}_{\sigma} \rvert \biggr), \\
      &\cdots \otimes \biggl(A, \sum_{\sigma \in \mathsf{cell}_5} \lvert \mathcal{L}_{\sigma} \rvert \biggr)
      \otimes
      \biggl(A, \sum_{\sigma \in \mathsf{cell}_3} \lvert \mathcal{L}_{\sigma} \rvert \biggr)
      \otimes
      \biggl(A, \sum_{\sigma \in \mathsf{cell}_1} \lvert \mathcal{L}_{\sigma} \rvert \biggr) \otimes \cdots
      \cong
      \biggl(A, \sum_{i : \text{odd}} \sum_{\sigma \in \mathsf{cell}_i} \lvert \mathcal{L}_{\sigma} \rvert \biggr), \\
    \end{align*}
  \item other isomorphisms are given by the braiding and the inverse structures on $\mathcal{P}_A$.
\end{itemize}
We call $\mathsf{vol}_{M, \mathcal{L}}$ the \emph{geometric volume form} associated with $\mathcal{L}$.
When $M$ is clear from the context, we simply write 
$\mathsf{vol}_{\mathcal{L}, i}$ and $\mathsf{vol}_{\mathcal{L}}$ instead of
$\mathsf{vol}_{M, \mathcal{L}, i}$ and $\mathsf{vol}_{M, \mathcal{L}}$.

\subsection{Homology orientations}
Let $M$ be a finite CW complex.

\begin{definition}\label{def:parity local system}
  For a linear local system $\mathcal{L} \in \Loc_{\GL, M}(A)$,
  we define a local system 
  $\parity \mathcal{L} \in \Loc_{\SL, M}(\mathbb{Q})$ as follows.
  If $\mathcal{L}_x$ is of constant rank for any $x \in M$,
  we define
  \begin{equation}\label{eq:parity local system}
    (\parity \mathcal{L})_x \coloneqq 
    \begin{cases}
      0 & \text{if $\rank_A \mathcal{L}_x$ is even}, \\
      \mathbb{Q} & \text{if $\rank_A \mathcal{L}_x$ is odd},
    \end{cases}
  \end{equation}
  and define $(\mathsf{vol}_{\parity \mathcal{L}})_x$ to be the 
  the canonical map $\det 0 \cong (\mathbb{Q}, 0)$ in the even case and
  $\det \mathbb{Q} \cong (\mathbb{Q}, 1)$ in the odd case.
  For an arbitrary $\mathcal{L} \in \Loc_{\GL, M}(A)$,
  we apply \eqref{eq:parity local system} locally on $\Spec A$.
  We call $\parity \mathcal{L}$ the \emph{parity local system} associated with $\mathcal{L}$.
\end{definition}

We call a volume form on $\det H_{\bullet} (M; \mathbb{Q})$ a 
\emph{homology orientation} of $M$.

\begin{definition}
  Let $\mathfrak{o}$ be a homology orientation of $M$.
  Let $\mathcal{L} \in \Loc_{\SL, M}(A)$.
  We define a volume form $\mathfrak{o}_{\mathcal{L}}$ on 
  $\det H_{\bullet} (M; \parity \mathcal{L})$ as follows.
  If $\mathcal{L}_x$ is of constant rank for any $x \in M$,
  we define $\mathfrak{o}_{\mathcal{L}}$ to be the composition
  \begin{equation}\label{eq:homology orientation odd}
    \begin{split}
    \det H_\bullet (M; \parity \mathcal{L})
    &\cong 
    \det H_\bullet (M_{\mathrm{even}}; 0) \otimes \det H_\bullet (M_{\mathrm{odd}}; \mathbb{Q}) \\
    &\xrightarrow{\id \otimes \mathfrak{o}|_{M_{\mathrm{odd}}}}
    1 \otimes (\mathbb{Q}, \chi (M_{\mathrm{odd}})) \\
    &\cong (\mathbb{Q},  \chi (M_{\mathrm{odd}}))
    \end{split}
  \end{equation}
  where $M = M_{\mathrm{even}} \sqcup M_{\mathrm{odd}}$ with
  \begin{equation*}
  M_{\mathrm{even}}=\{x\in M \mid \rank_A \mathcal{L}_x \text{ is even}\},\qquad
  M_{\mathrm{odd}}=\{x\in M \mid \rank_A \mathcal{L}_x \text{ is odd}\},
  \end{equation*}
  and $\chi (M_{\mathrm{odd}})$ is the Euler characteristic of $M_{\mathrm{odd}}$.
  For an arbitrary $\mathcal{L} \in \Loc_{\GL, M}(A)$,
  we apply \eqref{eq:homology orientation odd} locally on $\Spec A$.
\end{definition}

\subsection{Regular local systems}
Let $M$ be a link exterior complex,
and $\gamma$ be a peripheral loop in $M$
(\cref{dfn:link exterior complex}).
Let $i_1: \gamma \to \partial M$, $i_2: \partial M \to M$ denote the inclusions, which induce
\begin{align*}
    \Loc_{\GL, M} \xrightarrow{i_2^\ast} \Loc_{\GL, \partial M} \xrightarrow{i_1^\ast} \Loc_{\GL, \gamma}.
\end{align*}
For any $\mathcal{L} \in \Loc_{\GL, M}$, we write $\mathcal{L}|_{\partial M} \coloneqq i_2^\ast(P)$ and $\mathcal{L}|_\gamma \coloneqq i_1^\ast(P|_{\partial M})$. 
Then we have isomorphisms
\begin{align}
    \cap[\partial M]&: H^0(\partial M;\mathcal{L}|_{\partial M}) \xrightarrow{\cong} H_2(\partial M;\mathcal{L}|_{\partial M}), \label{eq:Poincare_duality_T} \\
    \cap[\gamma]&: H^0(\gamma;\mathcal{L}|_\gamma) \xrightarrow{\cong} H_1(\gamma;\mathcal{L}|_\gamma). \label{eq:Poincare_duality_gamma}
\end{align}
By \eqref{eq:Poincare_duality_T} and \eqref{eq:Poincare_duality_gamma},
we obtain the maps
\begin{align}
  \label{eq:H^0 to H_2}
  &H^0 (\partial M; \mathcal{L}|_{\partial M}) \xlongrightarrow{\cong} H_2 (\partial M; \mathcal{L}|_{\partial M}) 
    \longrightarrow H_2 (M; \mathcal{L}), \\
  \label{eq:H^0 to H_1}
  &H^0 (\partial M; \mathcal{L}|_{\partial M}) \longrightarrow H^0 (\gamma; \mathcal{L}|_{\gamma})
  \xlongrightarrow{\cong} H_1 (\gamma; \mathcal{L}|_{\gamma})
  \longrightarrow H_1 (M; \mathcal{L}).
\end{align}

\begin{definition}\label{def:boundary-regular}
  Let $r$ be a natural number.
  For a local system $\mathcal{L} \in \Loc_{\GL, M}(A)$,
  we consider the following conditions:
  \begin{enumerate}
    \item
      \label{item:boundary-regular locally finite}
      $\mathcal{L} \in \Loc_{\GL, M}^{b_* = (0,r,r,0)}(A)$ (see Definition \ref{def:fixed betti moduli}),
    \item
      \label{item:boundary-regular boundary locally finite}
      $\mathcal{L}|_{\partial M} \in \Loc_{\GL, \partial M}^{b^0 = r} (A)$ (see Definition \ref{def:fixed coBetti moduli}),
    \item
      \label{item:boundary-regular isom}
      the map \eqref{eq:H^0 to H_2} is an isomorphism.
    \item
      \label{item:gamma-regular isom}
      the map \eqref{eq:H^0 to H_1} is an isomorphism.
  \end{enumerate}
  We say that $\mathcal{L}$ is \emph{boundary-regular
  with betti number $r$}
  if the conditions 
  \eqref{item:boundary-regular locally finite}--\eqref{item:boundary-regular isom} are satisfied,
  and we say that $\mathcal{L}$ is \emph{$\gamma$-regular with betti number $r$} if the conditions
  \eqref{item:boundary-regular locally finite}--\eqref{item:gamma-regular isom} are satisfied.
  We denote by $\Loc_{\GL, M}^{\reg{\partial}, (r)}(A)$ the
  category of boundary-regular local systems,
  and by $\Loc_{\GL, M}^{\reg{\gamma}, (r)}(A)$ the
  category of $\gamma$-regular local systems.
\end{definition}

\begin{lemma}\label{lem:boundary regular substack}
  $\Loc_{\GL, M}^{\reg{\partial}, (r)}$ and $\Loc_{\GL, M}^{\reg{\gamma}, (r)}$
  are represented by locally closed substacks of $\Loc_{\GL, M}$.
\end{lemma}
\begin{proof}
  Set
  \begin{equation*}
    \Loc_{\GL, M}^{\partial b^0 = r}
    \coloneqq
    \Loc_{\GL, \partial M}^{b^0 = r} \times_{\Loc_{\GL, \partial M}} \Loc_{\GL, M},
  \end{equation*}
  and consider the locally closed substack
  \begin{equation*}
    X
    \coloneqq
    \Loc_{\GL, M}^{b_* = (0,r,r,0)} \times_{\Loc_{\GL, M}} \Loc_{\GL, M}^{\partial b^0 = r}
    \ \subset \ \Loc_{\GL, M},
  \end{equation*}
  which represents the conditions \eqref{item:boundary-regular locally finite} and \eqref{item:boundary-regular boundary locally finite}.
  By the defining property of $Y \coloneqq \Loc_{\GL, M}^{b_* = (0,r,r,0)}$,
  we see that $H_1(M; -)$ and $H_2(M; -)$ are locally free $\mathcal{O}_Y$-modules of rank $r$,
  and their pullbacks to $X$ are
  locally free $\mathcal{O}_{X}$-modules of rank $r$.
  Likewise, by the defining property of $Z \coloneqq \Loc_{\GL, \partial M}^{b^0 = r}$,
  the sheaf $H^0(\partial M; (-)|_{\partial M})$ is locally free 
  $\mathcal{O}_Z$-module of rank $r$, and its pullback to $X$ is a locally free $\mathcal{O}_{X}$-module of rank $r$.

  The maps \eqref{eq:H^0 to H_2} and \eqref{eq:H^0 to H_1} therefore define morphisms of
  locally free $\mathcal{O}_X$-modules of the same rank $r$.
  For a morphism of locally free sheaves of finite rank, the locus where it is an isomorphism
  is open.
  Hence the conditions \eqref{item:boundary-regular isom} and \eqref{item:gamma-regular isom}
  cut out open substacks of $X$.
\end{proof}

We define the stack of 
boundary-regular local system by
\begin{equation}
  \Loc_{\GL, M}^{\reg{\partial}} \coloneqq
  \coprod_{r \in \mathbb{N}} \Loc_{\GL, M}^{\reg{\partial}, (r)},
\end{equation}
and its $\SL$-version by
\begin{equation}
  \Loc_{\SL, M}^{\reg{\partial}} \coloneqq
  \Loc_{\SL, M} \times_{\Loc_{\GL, M}}
  \Loc_{\GL, M}^{\reg{\partial}}.
\end{equation}

\subsection{Porti form}
\label{sec:Porti form}
We continue to assume that $M$ is a link exterior complex and 
$\gamma$ is a peripheral loop in $M$.
We have $H_n (M; \mathcal{L}) = 0$ for any $\mathcal{L} \in \Loc_{\GL, M}(A)$ unless $n = 0, 1, 2, 3$, 
since $M$ is of formal dimension $3$. Thus we have
\begin{equation}\label{eq:det unless 0-3}
  \det H_n (M; \mathcal{L}) \cong 1_A
\end{equation}
unless $n = 0, 1, 2, 3$.

Let $\mathcal{L} \in \Loc_{\GL, M}^{\reg{\partial}}(A)$. 
By condition \eqref{item:boundary-regular locally finite} of \cref{def:boundary-regular},
we have
\begin{equation}\label{eq:det H0 H3}
  \det H_0 (M; \mathcal{L}) \cong 1_A, \quad
  \det H_3 (M; \mathcal{L}) \cong 1_A.
\end{equation}
By condition \eqref{item:boundary-regular isom} of Definition \ref{def:boundary-regular},
we have a map
\begin{equation*}\label{eq:H2 to H1}
  \bar{h}_{\gamma, \mathcal{L}} : H_2 (M; \mathcal{L}) \to H_1 (M; \mathcal{L})
\end{equation*}
by composing the inverse of the map \eqref{eq:H^0 to H_2} and
the map \eqref{eq:H^0 to H_1},
which is an $A$-module homomorphism, but not necessarily an isomorphism in general.
By taking the determinant, we get a map
\begin{equation}\label{eq:det H2 to H1}
  \det \bar{h}_{\gamma, \mathcal{L}} : \det H_2 (M; \mathcal{L}) \to \det H_1 (M; \mathcal{L}).
\end{equation}
When $\bar{h}_{\gamma, \mathcal{L}}$ is not an isomorphism,
we regard $\det \bar{h}_{\gamma, \mathcal{L}}$ as its top degree exterior power.
We define a map
\begin{equation}\label{eq:Porti form}
  h_{\gamma, \mathcal{L}} : \det H_{\bullet} (M; \mathcal{L}) \to A
\end{equation}
by composing 
\begin{equation}\label{eq:det H to H1 H2}
  \det H_{\bullet} (M; \mathcal{L})
  \cong
  \det H_2 (M; \mathcal{L}) \otimes (\det H_1 (M; \mathcal{L}))^{\otinv}
\end{equation}
and
\begin{equation}\label{eq:Porti form core}
  \det H_2 (M; \mathcal{L}) \otimes (\det H_1 (M; \mathcal{L}))^{\otinv}
  \xrightarrow{\det \bar{h}_{\gamma, \mathcal{L}} \otimes \id}
  \det H_1 (M; \mathcal{L}) \otimes (\det H_1 (M; \mathcal{L}))^{\otinv}
  \cong
  A,
\end{equation}
where the isomorphism \eqref{eq:det H to H1 H2} is by 
\eqref{eq:det unless 0-3} and \eqref{eq:det H0 H3}.
We call $h_{\gamma, \mathcal{L}}$ the \emph{Porti form} for $\gamma$.

\begin{lemma}\label{lemma:Porti form of gamma-regular}
  Let $\mathcal{L} \in \Loc_{\GL, M}^{\reg{\partial}}(A)$.
  Then the Porti form $h_{\gamma, \mathcal{L}}$ is a volume form on $\det H_{\bullet} (M; \mathcal{L})$ 
  (i.e., it is an isomorphism, and thus is a morphism in $\mathcal{P}_A$) if and only if
  $\mathcal{L}$ is $\gamma$-regular.
\end{lemma}
\begin{proof}
  $\mathcal{L}$ is $\gamma$-regular if and only if
  the map $\bar{h}_{\gamma, \mathcal{L}}$ in \eqref{eq:H2 to H1} is an isomorphism,
  which is equivalent to that $\det \bar{h}_{\gamma, \mathcal{L}} : \det H_2 (M; \mathcal{L}) \to \det H_1 (M; \mathcal{L})$ is an isomorphism.
  Since the evaluation map in \eqref{eq:Porti form core} is an isomorphism,
  we see that $\det \bar{h}_{\gamma, \mathcal{L}}$ is an isomorphism if and only if $h_{\gamma, \mathcal{L}}$ is a volume form on $\det H_{\bullet} (M; \mathcal{L})$.
\end{proof}

\subsection{Torsion function on the moduli stack of local systems}
We continue to assume that $M$ is a link exterior complex and 
$\gamma$ is a peripheral loop in $M$.

For $\mathcal{L} \in \Loc_{\GL, M}(A)$,
we have an isomorphism
\begin{equation}\label{eq:det C to A}
  \Phi_{\mathcal{L}} : \det C_{\bullet} (M; \mathcal{L}) \cong 
  \det H_{\bullet} (M; \mathcal{L})
\end{equation}
by Definition \ref{theorem:det C to det H}.

\begin{definition}\label{def:torsion function}
  Let $A$ be a $k$-algebra, and $\mathcal{L} \in \Loc_{\SL, M}^{\reg{\partial}}(A)$.
  Let $\mathfrak{o} : \det H_{\bullet} (M; \mathbb{Q}) \cong \mathbb{Q}$ be a homology orientation.
  The \emph{torsion function} on the category of boundary-regular local systems
  is a functor
  \begin{equation}\label{eq:torsion at A}
    \torsion_{M, \gamma, \mathfrak{o}}(A) : \Loc_{\SL, M}^{\reg{\partial}} (A) \to A
  \end{equation}
  defined by
  \begin{equation}\label{eq:def_torsion}
    \torsion_{M, \gamma, \mathfrak{o}}(A, \mathcal{L}) \coloneqq 
    \epsilon_{\mathfrak{o}, \mathcal{L}} \cdot
      h_{\gamma, \mathcal{L}} \bigl (\Phi_{\mathcal{L}} (1_{\mathsf{vol}_{\scriptstyle \mathcal{L}}}) \bigr),
  \end{equation}
  where 
  \begin{itemize}
    \item $1_{\mathsf{vol}_{\mathcal{L}}} \in \det C_{\bullet} (M; \mathcal{L})$
  is the unit element associated with the geometric volume form defined in \eqref{eq:vol complex}-\eqref{eq:vol complex end},
    \item $h_{\gamma, \mathcal{L}} : \det H_\bullet (M; \mathcal{L}) \to A$ is the Porti form defined in \eqref{eq:Porti form}, and
    \item $\epsilon_{\mathfrak{o}, \mathcal{L}} = \pm 1$ is a sign defined by
  \begin{equation}\label{eq:orientation sign}
    \epsilon_{\mathfrak{o}, \mathcal{L}} \coloneqq       
    \sign\bigl(
      \mathfrak{o}_{\mathcal{L}} \bigl(
      \Psi_{\parity \mathcal{L}} (1_{\mathsf{vol}_{\scriptstyle \parity \mathcal{L}}}) \bigr)\bigr) 
  \end{equation}
  where $\parity \mathcal{L} \in \Loc_{\SL, M}(\mathbb{Q})$
  is the parity local system defined in \eqref{eq:parity local system},
  and $\mathfrak{o}_{\mathcal{L}}$ is the volume form on 
  $\det H_{\bullet} (M; \parity \mathcal{L})$ 
  defined by \eqref{eq:homology orientation odd}.
  \end{itemize}
\end{definition}

In the codomain of \eqref{eq:torsion at A}, we regard $A$ as a discrete category whose set of objects is $A$.
In particular, 
the functoriality means $\torsion_{M, \gamma, \mathfrak{o}}(A, \mathcal{L}) = \torsion_{M, \gamma, \mathfrak{o}}(A, \mathcal{L}')$
if $\mathcal{L} \cong \mathcal{L}'$. 

\begin{lemma}\label{lemma:naturality torsion}
  Suppose that $f : A \to B$ is a morphism of $k$-algebras.
  We have a commutative diagram:
  \begin{equation}\label{eq:naturality torsion}
    \begin{tikzpicture}[scale=1.7, commutative diagrams/every diagram,  baseline={(current bounding box.center)}]
      \node (RA) at (0, 0) {$\Loc_{\SL, M}^{\reg{\partial}}(A)$};
      \node (A) at (3, 0) {$A$};
      \node (RB) at (0, -1) {$\Loc_{\SL, M}^{\reg{\partial}}(B)$};
      \node (B) at (3, -1) {$B$};
      \draw[morphism] (RA) -- (A) node[midway, above] {$\torsion_{M, \gamma, \mathfrak{o}}(A)$};
      \draw[morphism] (RA) -- (RB) node[midway, left] {$(-) \otimes_A B$};
      \draw[morphism] (A) -- (B) node[midway, right] {$f$};
      \draw[morphism] (RB) -- (B) node[midway, above] {$\torsion_{M, \gamma, \mathfrak{o}}(B)$};
    \end{tikzpicture}
  \end{equation}
\end{lemma}
\begin{proof}
  The geometric volume form $\mathsf{vol}_{\mathcal{L}}$ commutes with extension of scalars
  since the map \eqref{eq:det exact} has this property.
  The map $\Phi_{\mathcal{L}}$ commutes with extension of scalars by results in \cite{Knudsen}.
  The Porti form $h_{\gamma, \mathcal{L}}$ also commutes with extension of scalars
  since taking cap product (used in the Poincar\'e dualities) commutes with extension of scalars.
  Finally, the sign $\epsilon_{\mathfrak{o}, \mathcal{L}}$ is invariant under extension of scalars
  since the rank of $\mathcal{L}$ commutes with extension of scalars.
\end{proof}

By the naturality \eqref{eq:naturality torsion} of the torsion function, we have:
\begin{definition}\label{def:torsion function morphism}
  The \emph{torsion function} on the moduli stack of boundary-regular linear local systems 
  with volume forms is the morphism of stacks over $k$:
  \begin{equation}\label{eq:torsion function morphism}
    \torsion_{M, \gamma, \mathfrak{o}} : \Loc_{\SL, M}^{\reg{\partial}} \to \mathbb{A}^1_k
  \end{equation}
  whose values in a $k$-algebra $A$
  are given by Definition \ref{def:torsion function}.
\end{definition}

In \eqref{eq:torsion function morphism}, we denote by
$\mathbb{A}^1_k$ the affine line over $k$.
Viewed as a functor $\mathbb{A}^1_k : \Alg_k \to \Set$, it sends a $k$-algebra $A$ to its underlying set.

In what follows,
we will denote $\torsion_{M, \gamma, \mathfrak{o}} (A, \mathcal{L})$ by 
$\torsion_{M, \gamma, \mathfrak{o}} (\mathcal{L})$ when there is no risk of confusion.

\begin{lemma}\label{lemma:torsion indep order}
  The torsion function $\torsion_{M, \gamma, \mathfrak{o}}$ does not depend on a 
  choice of the ordering of the cells of $M$.
\end{lemma}
\begin{proof}
  Let $\mathcal{L} \in \Loc_{\SL, M}^{\reg{\partial}}(A)$.
  We may assume that $\mathcal{L}_x$ is of constant rank for any $x \in M$.
  We compare the sign changes of the volume forms $\mathsf{vol}_{\mathcal{L}}$
  and $\mathsf{vol}_{\parity \mathcal{L}}$.
  When we swap two adjacent cells $\sigma$ and $\sigma'$ in the ordering, 
  both volume forms change by the factor 
  $(-1)^{n n'}$, where $n = \rank_A \mathcal{L}_{\sigma}$ and 
  $n' = \rank_A \mathcal{L}_{\sigma'}$.
  Therefore, the sign changes cancel, and the torsion remains invariant.
\end{proof}

\begin{lemma}\label{lemma:torsion of even local system}
  Let $\mathcal{L} \in \Loc_{\SL, M}^{\reg{\partial}}(A)$
  be a local system such that $\mathcal{L}_x$ is of even rank for any $x \in M$.
  Then we have $\epsilon_{\mathfrak{o}, \mathcal{L}} = 1$.
  In particular, the torsion $\torsion_{M, \gamma, \mathfrak{o}} (\mathcal{L})$
  does not depend on the homology orientation $\mathfrak{o}$.
\end{lemma}
\begin{proof}
  This follows from the fact that $M_{\mathrm{odd}}$ in \eqref{eq:homology orientation odd}
  is empty in this case.
\end{proof}

\begin{prop}\label{lemma:torsion of gamma regular}
  Let $\mathcal{L} \in \Loc_{\SL, M}^{\reg{\partial}}(A)$.
  Then $\mathcal{L}$ is $\gamma$-regular if and only if
  the torsion $\torsion_{M, \gamma, \mathfrak{o}} (\mathcal{L})$ is a unit in $A$.
\end{prop}
\begin{proof}
  This follows from \cref{lemma:Porti form of gamma-regular}.
\end{proof}

\begin{prop}\label{lemma:torsion of direct sum}
  Let $\mathcal{L}, \mathcal{L}' \in \Loc_{\SL, M}^{\reg{\partial}}(A)$.
  We have
  \begin{equation}\label{eq:torsion of direct sum}
    \torsion_{M, \gamma, \mathfrak{o}} (\mathcal{L} \oplus \mathcal{L}')
    =
    \torsion_{M, \gamma, \mathfrak{o}} (\mathcal{L}) \cdot
    \torsion_{M, \gamma, \mathfrak{o}} (\mathcal{L}').
  \end{equation}
\end{prop}
\begin{proof}
  The assertion follows from
  the properties of the determinant functor
  on complexes given in~\cite[Section 3]{Knudsen}.
\end{proof}

\subsection{Topological invariance of the torsion}
\label{sec:topological invariance}
Here we discuss the invariance of the torsion function  under homeomorphisms. By Chapman's theorem \cite[Theorem 1]{Chapman}, any homeomorphism between two compact CW complexes is a simple homotopy equivalence. Therefore it is enough to prove the invariance under simple homotopy equivalences. 
We refer the reader to \cite{Cohen,Kozlov} for details on simple homotopy equivalence.

First note that if $f: M \to M'$ and $g: M' \to M$ form a homotopy equivalence between two link exterior complexes $M$ and $M'$, we have an equivalence 
$\Loc_{\SL, M} \simeq \Loc_{\SL, M'}$ by sending a local system $\mathcal{L} \in \Loc_{\SL, M}(A)$ to the pull-back local system $g^\ast \mathcal{L} \in \Loc_{\SL, M'}(A)$ and sending a local system $\mathcal{L}' \in \Loc_{\SL, M'}(A)$ to the pull-back local system $f^\ast \mathcal{L}' \in \Loc_{\SL, M}(A)$.
This equivalence restricts to the equivalence
$\Loc_{\SL, M}^{\reg{\partial}} \simeq \Loc_{\SL, M'}^{\reg{\partial}}$ between the stacks of boundary-regular local systems 
since the conditions in \cref{def:boundary-regular} are homotopy invariant.
We are going to verify that the torsion function remain unchanged under simple homotopy equivalences.

\begin{definition}[{\cite[Section 4]{Cohen}}]
Let $M$ and $M'$ be finite CW complexes.
An \emph{elementary expansion} from 
$M$ to $M'$ is a tuple $f = (i, \iota_{\pm}, \varphi)$, where $i \in \mathbb{N}$ with $i \ge 1$,
$\iota_{\pm} : D^{i-1} \to \partial D^i$ is a continuous map,
and $\varphi: D^i \to M'$ is a continuous map, such that these satisfy 
the following conditions:
\begin{enumerate}
    \item $M$ is a subcomplex of $M'$ such that 
    $\mathsf{cell}_j (M) = \mathsf{cell}_j (M')$ for any $j \ne i-1, i$, 
    and the sets $\mathsf{cell}_{i-1} (M') \setminus \mathsf{cell}_{i-1} (M)$ and
    $\mathsf{cell}_i (M') \setminus \mathsf{cell}_i (M)$ both consist of a single element, 
    which we denote by $\sigma^{i-1}$ and $\sigma^i$, respectively.
    \item $\iota_{\pm}$ are homeomorphisms onto their images such that
    $\partial D^i = \iota_+ (D^{i-1}) \cup \iota_- (D^{i-1})$
    and $\iota_+ (D^{i-1}) \cap \iota_- (D^{i-1}) = \iota_+ (\partial D^{i-1}) = \iota_- (\partial D^{i-1})$,
    \item $\varphi$ and $\varphi|_{\partial D^{i}} \circ \iota_+$ are characteristic maps for 
    $\sigma^i$ and $\sigma^{i-1}$, respectively, and 
    $\Img (\varphi|_{\partial D^{i}} \circ \iota_-) \subset M^{(i-1)}$.
\end{enumerate}
We write $f : M \nearrow M'$ to mean that $f$ is an elementary expansion from $M$ to $M'$.
\end{definition}

An \emph{elementary collapse} from $M$ to $M'$ is defined to be an elementary expansion from 
$M'$ to $M$. We write $f : M \searrow M'$ to mean that $f$ is an elementary collapse from $M$ to $M'$.
For an elementary expansion $f : M \nearrow M'$, by abuse of notation, we denote by $f : M \to M'$ the inclusion map
between CW complexes.
For an elementary collapse $f : M \searrow M'$, again by abuse of notation, we denote by $f : M \to M'$ the retraction
between CW complexes.
A map $f : M \to M'$ is said to be a \emph{simple homotopy equivalence} if it is a finite composition of 
elementary expansions and collapses.

\begin{dfn}\label{def:expansion_triple}
Let $M$ and $M'$ be link exterior complexes,
and let $\gamma$ and $\gamma'$ be peripheral loops in $M$ and $M'$, respectively.
An elementary expansion from $(M,\gamma)$ to $(M',\gamma')$ is an elementary 
expansion $f : M \nearrow M'$ that sends the fundamental classes
$[M]$, $[\partial M]$, and $[\gamma]$ to $[M']$, $[\partial M']$, and $[\gamma']$, respectively, 
and either of the following conditions holds:
\begin{enumerate}
    \item $f$ restricts to both $\partial M \nearrow \partial M'$
    and $\gamma \nearrow \gamma'$
    \item $f$ restricts to $\partial M \nearrow \partial M'$, and $\gamma = \gamma'$,
    \item $\partial M = \partial M'$, and $\gamma = \gamma'$.
\end{enumerate}
For an elementary expansion $f : (M,\gamma) \nearrow (M',\gamma')$, by abuse of notation, we have a map 
$f : (M,\gamma) \to (M',\gamma')$ of link exterior complexes with peripheral loops
by 
The fundamental classes are induced by the inclusions. 
We write $f : (M,\gamma) \nearrow (M',\gamma')$ to mean that $f$ is an elementary expansion from $(M,\gamma)$ to $(M',\gamma')$,
\end{dfn}

An \emph{elementary collapse} from $(M, \gamma)$ to $(M', \gamma')$ is defined to be an elementary expansion from 
$(M', \gamma')$ to $(M, \gamma)$. We write $f : (M, \gamma) \searrow (M', \gamma')$ to mean that $f$ is an elementary collapse from $(M, \gamma)$ to $(M', \gamma')$.
For an elementary expansion $f : (M, \gamma) \nearrow (M', \gamma')$, by abuse of notation, 
we denote by $f : (M, \gamma) \to (M', \gamma')$ the inclusion map between link exterior complexes with peripheral loops.
For an elementary collapse $f : M \searrow M'$, again by abuse of notation, we denote by $f : M \to M'$ the retraction
between link exterior complexes with peripheral loops.
If a map $f : (M, \gamma) \to (M', \gamma')$ is a simple homotopy equivalence on $M$, $\partial M$, and $\gamma$, 
then $f$ can be written as a finite composition of elementary expansions and collapses.

Let $(M, \gamma)$ and $(M', \gamma')$ be link exterior complexes with peripheral loops,
and let $f : (M, \gamma) \nearrow (M', \gamma')$.
In particular, the inclusion $\iota: M \to M'$ is a homotopy equivalence, and hence induces an equivalence
$\iota^\ast: \Loc_{M',SL}^{\reg{\partial}} \simeq \Loc_{M,SL}^{\reg{\partial}}$
between the stacks of unimodular local systems. 
For any $\mathcal{L}' \in \Loc_{M'}^{\reg{\partial}}(A)$, we have a quasi-isomorphism
\begin{align*}
    \iota_\ast: C_\bullet (M; \iota^\ast\mathcal{L}') \to C_\bullet (M'; \mathcal{L}').
\end{align*}

We set $\mathcal{L} \coloneqq \iota^\ast \mathcal{L}'$,
and consider the diagram
\begin{equation}\label{eq:torsion simple homotopy}
  \begin{tikzpicture}[scale=1.3, xscale = 1.1, commutative diagrams/every diagram,  baseline={(current bounding box.center)}]
    \node (A0) at (-2.5, 0) {$A$};
    \node (A1) at (-2.5, -1) {$A$};
    \node (RA) at (0, 0) {$\det C_\bullet(M; \mathcal{L})$};
    \node (A) at (2.5, 0) {$\det H_\bullet(M; \mathcal{L})$};
    \node (RB) at (0, -1) {$\det C_\bullet(M'; \mathcal{L}')$};
    \node (B) at (2.5, -1) {$\det H_\bullet(M'; \mathcal{L}')$};
    \node (C0) at (4.5, 0) {$A$};
    \node (C1) at (4.5, -1) {$A$};
    \draw[morphism] (A0) -- (RA) node[midway, above] {$(\mathsf{vol}_{M, \mathcal{L}})^{-1}$};
    \draw[morphism] (A1) -- (RB) node[midway, below] {$(\mathsf{vol}_{M', \mathcal{L}'})^{-1}$};
    \draw[morphism] (A0) -- (A1) node[midway, left] {$\epsilon$};
    \draw[morphism] (RA) -- (A) node[midway, above] {$\Phi_{\mathcal{L}}$};
    \draw[morphism] (RA) -- (RB) node[midway, left] {$\det \iota_\ast$};
    \draw[morphism] (A) -- (B) node[midway, left] {$\det \iota_\ast$};
    \draw[morphism] (RB) -- (B) node[midway, below] {$\Phi_{\mathcal{L}'}$};
    \draw[morphism] (A) -- (C0) node[midway, above] {$h_{\mathcal{L}, \gamma} $};
    \draw[morphism] (B) -- (C1) node[midway, below] {$h_{\mathcal{L}', \gamma'} $};
    \draw[morphism] (C0) -- (C1) node[midway, left] {$\mathrm{id}$};
  \end{tikzpicture}
\end{equation}
The middle square commutes by \cite[Proposition 3.3]{Knudsen}, and the right square commutes by the naturality of taking cap products.
The left square commutes for an appropriate sign $\epsilon$,
which we determine below.

\begin{lemma}\label{lemma:sign of elementary expansion}
  Let $M$ and $M'$ be finite CW complexes such that $M'$ is an elementary expansion of $M$.
  We choose the linear ordering on $\mathsf{cell}(M')$ as 
  an extension of $\mathsf{cell}(M)$ such that 
  $\sigma^{i-1}$ and $\sigma^i$ are the last in 
  $\mathsf{cell}_{i-1}(M')$ and $\mathsf{cell}_{i}(M')$, respectively.
  Let $\mathcal{L}' \in \Loc_{\SL, M'}(A)$ and set 
  $\mathcal{L} \coloneqq \iota^\ast \mathcal{L}' \in \Loc_{\SL, M}(A)$,
  where $\iota : M \to M'$ is the inclusion and $\iota^\ast : \Loc_{\SL, M'}(A) \to \Loc_{\SL, M}(A)$
  is the induced equivalence.
  Suppose that $\mathcal{L}_\tau$ is of constant rank $n_\tau$ for any $\tau \in \mathsf{cell}(M)$,
  and $\mathcal{L}'_{\sigma^{i-1}}$ and $\mathcal{L}'_{\sigma^i}$ are of constant rank $n$.
  Suppose also that $\sum_j (-1)^j \sum_{\sigma \in \mathsf{cell}_j(M)} n_\sigma = 0$.
  Then the diagram
  \begin{equation}\label{eq:torsion simple homotopy sign}
    \begin{tikzpicture}[scale=1.3, xscale = 1.1, commutative diagrams/every diagram,  baseline={(current bounding box.center)}]
      \node (A0) at (-2.5, 0) {$A$};
      \node (A1) at (-2.5, -1) {$A$};
      \node (RA) at (0, 0) {$\det C_\bullet(M; \mathcal{L})$};
      \node (RB) at (0, -1) {$\det C_\bullet(M'; \mathcal{L}')$};
      \draw[morphism] (A0) -- (RA) node[midway, above] {$(\mathsf{vol}_{M, \mathcal{L}})^{-1}$};
      \draw[morphism] (A1) -- (RB) node[midway, below] {$(\mathsf{vol}_{M', \mathcal{L}'})^{-1}$};
      \draw[morphism] (A0) -- (A1) node[midway, left] {$\epsilon$};
      \draw[morphism] (RA) -- (RB) node[midway, left] {$\det \iota_\ast$};
    \end{tikzpicture}
  \end{equation}
  commutes with 
  \begin{align}\label{eq:sign of elementary expansion}
    \epsilon = 
    [\sigma^i:\sigma^{i-1}]^{(-1)^i} \cdot (-1)^\alpha, 
  \end{align}
  where $[\sigma^i:\sigma^{i-1}]$ is the incidence number of $\sigma^i$ and $\sigma^{i-1}$,
  and
  \begin{equation}\label{eq:sign of elementary expansion 1}
    \alpha = 
    \sum_{j : \mathrm{odd}} \sum_{\sigma \in \mathsf{cell}_j (M)} n_\sigma \cdot n
    + \sum_{\substack{j : \mathrm{even} \\ j \le i}} \sum_{\sigma \in \mathsf{cell}_j (M)} n_\sigma \cdot n
    + \sum_{\substack{j : \mathrm{odd} \\ j < i}} \sum_{\sigma \in \mathsf{cell}_j (M)} n_\sigma \cdot n
    + i n^2.
  \end{equation}
\end{lemma}
\begin{proof}
  We compute the map $\det \iota_\ast$ using Definition 2.24 in \cite{Knudsen}.
  To this end, take $Q_\bullet$ such that $0 \to C_\bullet(M) \xrightarrow{\iota_\ast} C_\bullet(M') \to Q_\bullet \to 0$.
  We have $Q_j=0$ for $j \neq i-1,i$,
  and $Q_{i-1} = \mathcal{L}'_{\sigma^{i-1}}$ and $Q_{i} = \mathcal{L}'_{\sigma^i}$.
  Hence, $Q_\bullet$ is the two-term complex 
  $0 \to Q_i \xrightarrow{\partial_i} Q_{i-1} \to 0$, where 
  $\partial_i(v) = [\sigma^i:\sigma^{i-1}] \cdot \alpha (v)$ for $v \in \mathcal{L}'_{\sigma^i}$,
  where $\alpha : \mathcal{L}'_{\sigma^i} \to \mathcal{L}'_{\sigma^{i-1}}$ is the isomorphism 
  given by the local system $\mathcal{L}'$.
  Set 
  \begin{alignat*}{2}
    L_{j} \coloneqq (\det \mathcal{L}'_{\sigma^{j}})^{(\otinv)^{j}}, \quad
    &&\omega_{j} \coloneqq (\mathsf{vol}_{\mathcal{L}'_{\scriptstyle \sigma^{j}}})^{(\otinv)^{j}}
  \end{alignat*}
  for $j = i-1,i$.
  We define a volume form $\mathsf{vol}_{Q}$ on $Q_\bullet$ as the following composition:
  \begin{equation*}
  \det Q_\bullet 
  \cong L_i \otimes L_{i-1}
  \xrightarrow{\omega_i \otimes \omega_{i-1}}
  (A^{(\vee)^i}, (-1)^i n) \otimes (A^{(\vee)^{i-1}}, (-1)^{i-1} n) \to 1.
  \end{equation*}
  We consider the following diagram:
  \begin{equation*}
    \begin{tikzpicture}[scale=1.3, xscale=1.1, commutative diagrams/every diagram,  baseline={(current bounding box.center)}]
      \node (0) at (0, 0) {$\det C_\bullet(M')$};
      \node (1) at (2.5, 0) {$\det C_\bullet (M) \otimes \det Q_\bullet $};
      \node (2) at (5, 0) {$\det C_\bullet (M) \otimes 1$};
      \node (3) at (7.5, 0) {$\det C_\bullet (M)$};
      \node (00) at (0, -1) {$1$};
      \node (10) at (2.5, -1) {$1 \otimes 1$};
      \node (11) at (2.5, -2) {$1$};
      \node (20) at (5, -1) {$1 \otimes 1$};
      \node (21) at (5, -2) {$1$};
      \node (30) at (7.5, -1) {$1$};
      \draw[morphism] (0) -- (1) node[] {};
      \draw[morphism] (1) -- (2) node[] {};
      \draw[morphism] (2) -- (3) node[] {};
      \draw[morphism] (0) -- (00) node[midway, left] {$\mathsf{vol}_{M'}$};
      \draw[morphism] (1) -- (10) node[midway, left] {$\mathsf{vol}_{M} \otimes \mathsf{vol}_{Q}$};
      \draw[morphism] (10) -- (11) node[] {};
      \draw[morphism] (00) -- (11) node[midway, below left] {$\epsilon_1$};
      \draw[morphism] (2) -- (20) node[midway, right] {$\mathsf{vol}_{M} \otimes \id$};
      \draw[morphism] (20) -- (21) node[] {};
      \draw[morphism] (3) -- (30) node[midway, left] {$\mathsf{vol}_{M}$};
      \draw[morphism] (10) -- (20) node[midway, below] {$\id \otimes \epsilon_2$};
      \draw[morphism] (11) -- (21) node[midway, below] {$\epsilon_2$};
      \draw[morphism] (21) -- (30) node[midway, below right] {$\id$};
      \node at (barycentric cs:0=1,1=1,00=1,11=1) {$\scriptstyle \text{(1)}$};
      \node at (barycentric cs:1=1,2=1,10=1,20=1) {$\scriptstyle \text{(2)}$};
    \end{tikzpicture}
  \end{equation*}
  The composition along the top row is the inverse of $\det \iota_\ast$ by \cite[Definition 2.24 (d*)]{Knudsen},
  where in the middle map we use the isomorphism $\det Q_\bullet \cong 1$
  for an acyclic complex given by \cite[Definition 2.24 (c)]{Knudsen}.
  $\epsilon_1,\epsilon_2 \in \{\pm 1\}$ are appropriate signs. 
  The diagram clearly commutes, except possibly in the two squares labeled (1) and (2).
  The square (1) commutes if $\epsilon_1 = (-1)^\alpha$ with $\alpha$ given in \eqref{eq:sign of elementary expansion 1}.
  The first two terms in \eqref{eq:sign of elementary expansion 1} come from braiding for $\sigma^i$ in 
  \eqref{eq:vol complex degree contract}.
  The third term comes from braiding for $\sigma^{i-1}$ in
  \eqref{eq:vol complex degree contract}.
  The last term is nonzero modulo $2$ if and only if $i$ is odd,
  which comes from braiding of $\sigma^{i-1}$ and $\sigma^i$.
  The square (2) commutes if the 
  following diagram commutes:
  \begin{equation*}
    \begin{tikzpicture}[scale=1.3, xscale=1.55, commutative diagrams/every diagram,  baseline={(current bounding box.center)}]
      \node (0) at (0, 0) {$\det Q_\bullet$};
      \node (1) at (2.5, 0) {$L_{i-1}^{\otinv} \otimes L_{i-1}$};
      \node (2) at (5, 0) {$1$};
      \node (00) at (0, -1) {$L_{i} \otimes L_{i-1}$};
      \node (01) at (0, -2) {$A_i \otimes A_{i-1}$};
      \node (02) at (0, -3) {$1$};
      \node (10) at (2.5, -1) {$A_i \otimes A_{i-1}$};
      \node (11) at (2.5, -2) {$1$};
      \draw[morphism] (0) -- (1) node[] {};
      \draw[morphism] (1) -- (2) node[] {};
      \draw[morphism] (0) -- (00) node[midway, left] {};
      \draw[morphism] (00) -- (01) node[midway, left] 
        {$\omega_{i} \otimes \omega_{i-1}$};
      \draw[morphism] (01) -- (02) node[midway, left] {};
      \draw[morphism] (1) -- (10) node[midway, right] 
        {$\omega_{i-1}^{\otinv} \otimes \omega_{i-1}$};
      \draw[morphism] (10) -- (11) node[] {};
      \draw[morphism] (00) -- (1) node[midway, fill=white, auto=false] {$(\det \partial_i)^{(\otinv)^i} \otimes \id$};
      \draw[morphism] (01) -- (10) node[midway, below right] {$\epsilon_2 \otimes \id$};
      \draw[morphism] (02) -- (11) node[midway, below right] {$\epsilon_2$};
      \draw[morphism] (11) -- (2) node[midway, below right] {$\id$};
      \node at (barycentric cs:00=1,1=1,01=1,10=1) {$\scriptstyle \text{(3)}$};
      \node at (barycentric cs:1=0.5,2=1,11=0.5) {$\scriptstyle \text{(4)}$};
    \end{tikzpicture}
  \end{equation*}
  where we write $A_j \coloneqq (A^{(\vee)^j}, (-1)^j n)$ for $j=i-1,i$.
  Note that the square (4) commutes by \eqref{eq:epsilon naturality},
  and the whole diagram commutes except possibly in the square (3).
  If $i$ is even, 
  the (3) commutes if
  \begin{equation}\label{eq:sign of elementary expansion 2}
      \epsilon_2 = (\mathsf{vol}_{\mathcal{L}'_{\scriptstyle \sigma^{i-1}}} \circ 
      \det \partial_i \circ  
      \mathsf{vol}_{\mathcal{L}'_{\scriptstyle \sigma^i}}^{-1}) (1),
  \end{equation}
  and the right-hand side of \eqref{eq:sign of elementary expansion 2}
  is computed as
  \begin{align*}
      (\mathsf{vol}_{\mathcal{L}'_{\scriptstyle \sigma^{i-1}}} \circ 
      \det \partial_i \circ  
      \mathsf{vol}_{\mathcal{L}'_{\scriptstyle \sigma^i}}^{-1}) (1) 
      &= \mathsf{vol}_{\mathcal{L}'_{\scriptstyle \sigma^{i-1}}}
      (\det \partial_i (1_{\mathsf{vol}_{\scriptstyle \mathcal{L}'_{\scriptstyle \sigma^i}}} )) \\
      &= [\sigma^i : \sigma^{i-1}] \cdot
        \mathsf{vol}_{\mathcal{L}'_{\scriptstyle \sigma^{i-1}}} 
        (\alpha (1_{\mathsf{vol}_{\scriptstyle \mathcal{L}'_{\scriptstyle \sigma^i}}} )) \\
      &= [\sigma^i : \sigma^{i-1}] \cdot
        \mathsf{vol}_{\mathcal{L}'_{\scriptstyle \sigma^{i-1}}} (1_{\mathsf{vol}_{\scriptstyle \mathcal{L}'_{\scriptstyle \sigma^{i-1}}}} ) \\
      &= [\sigma^i : \sigma^{i-1}]. 
  \end{align*}
  If $i$ is odd, the (3) commutes if
  $\epsilon_2 = [\sigma^i : \sigma^{i-1}]^{-1}$.
  Thus in general, the (3) commutes if $\epsilon_2 = [\sigma^i : \sigma^{i-1}]^{(-1)^i}$.
\end{proof}

\begin{thm}[Topological invariance]\label{thm:topological_invariance}
  Let $f: (M, \gamma) \to (M', \gamma')$ be a simple homotopy equivalence
  (e.g., a homeomorphism) between link exterior complexes with peripheral loops.
  Then we have the following commutative diagram:
  \begin{equation}
    \begin{tikzpicture}[scale=1.5, commutative diagrams/every diagram, baseline={(current bounding box.center)}]
      \node (a) at (0,0) {$\Loc_{\SL, M}^{\reg{\partial}}$};
      \node (b) at (0,-1) {$\Loc_{\SL, M'}^{\reg{\partial}}$};
      \node (d) at (2,0) {$\mathbb{A}_k^1$};
      \draw[morphism] (a) to node[left] {$\simeq$} (b);
      \draw[morphism] (b) to node[pos=0.35, below right] {$\torsion_{M', \gamma', \mathfrak{o}'}$} (d);
      \draw[morphism] (a) to node[midway, above] {$\torsion_{M, \gamma, \mathfrak{o}}$} (d);
    \end{tikzpicture}
  \end{equation}
  where $\mathfrak{o}' = f_\ast \mathfrak{o}$.
\end{thm}
\begin{proof}
  It suffices to show the assertion when $f$ is an elementary expansion.
  This follows from the commutative diagram \eqref{eq:torsion simple homotopy}
  and \cref{lemma:sign of elementary expansion}, since
  the sign change in \eqref{eq:sign of elementary expansion} depends only 
  on the parity of the rank of the local system and hence common in the two terms in \eqref{eq:def_torsion}. 
\end{proof}

\section{Adjoint torsion}
Let $k$ be a field of characteristic zero,
and $G$ a connected semisimple algebraic group over $k$.

\subsection{Lie algebras}
In this section, we review the Lie algebras of algebraic groups.
See, e.g., \cite[Chapter II, Section 4]{DemazureGabriel} for more details.
For any $k$-algebra $A$, we denote by
$A(\varepsilon) \coloneqq A[\varepsilon]/(\varepsilon^2)$ the algebra of dual numbers.
We have morphisms $i : A \to A(\varepsilon)$ and 
$p : A(\varepsilon) \to A$ defined by 
$i(1) = 1$, $p(1) = 1$, and $p(\varepsilon) = 0$.
We denote by $\Lie (G) (A)$ the additive group 
isomorphic to the (multiplicative) group $\Ker (G(p))$.
For an element $x \in \Lie (G) (A)$,
the corresponding element in $\Ker (G(p)) \subset G(A(\varepsilon))$ is denoted by 
$e^{\varepsilon x}$.
We define the scalar product $A \times \Lie (G) (A) \to \Lie (G) (A)$ by 
\begin{equation}
e^{\varepsilon (a x)} \coloneqq 
G(\varepsilon \mapsto a \cdot \varepsilon : A(\varepsilon) \to A(\varepsilon)) (e^{\varepsilon x})
\end{equation}
for $a \in A$ and $x \in \Lie (G) (A)$.
Then $\Lie (G) (A)$ is an $A$-module with this scalar multiplication.

The assignment $A \mapsto \Lie (G) (A)$ extends to a functor $\Lie (G) : \Alg_k \to \Set$.
We have an action
\begin{equation}
\Ad : G \times \Lie (G) \to \Lie (G)
\end{equation} defined by
\begin{align}
  e^{\varepsilon (\Ad_A (g) x)} \coloneqq G(i)(g) \cdot  e^{\varepsilon x} \cdot G(i)(g^{-1})
\end{align}
for $g \in G(A)$ and $x \in \Lie(G)(A)$.
The $k$-module $\Lie(G)(k)$ is called the \emph{Lie algebra} of $G$,
and is denoted by $\mathfrak{g}$.
We have a canonical isomorphism $\mathfrak{g} \otimes_k A \cong \Lie (G) (A)$ (since $G$ is smooth over $k$, see~\cite[II, \S 4, 4.8]{DemazureGabriel}). 
We fix a volume form $\omega_{\mathfrak{g}}$ on $\mathfrak{g}$.
It is known that $(\mathfrak{g}, \omega_{\mathfrak{g}}) \in \Rep_{\SL}(G)$ 
for connected semisimple $G$, where the action is given by $\Ad$.

\subsection{Adjoint torsion function}
\label{sec:adjoint torsion function}
Let $M$ be a link exterior complex (see \cref{dfn:link exterior complex}).
We say that a $G$-local system $P \in \Loc_{G, M}(A)$ is \emph{boundary-adjoint-regular} if
the linear local system $P \times^G \mathfrak{g}$ is boundary-regular with 
betti number $m \cdot \rank G$ in the sense of Definition \ref{def:boundary-regular}.
We define the locally closed substack
\begin{equation}\label{eq:adjoint regular substack}
  \Loc_{G, M}^{\reg{\partial\text{-}\mathrm{Ad}}} \coloneqq
    \Loc_{G, M} \times_{\Loc_{\GL, M}} \Loc_{\GL, M}^{\reg{\partial}, (m \cdot \rank G)}
  \ \subset \ \Loc_{G, M},
\end{equation}
where the pullback is taken by the map $\Loc_{G, M} \to \Loc_{\GL, M}$ induced by the adjoint representation.

\begin{definition}\label{def:adjoint torsion function}
  Let $\mathfrak{o}$ be a homology orientation of $M$.
  We define the \emph{adjoint torsion function}
  \begin{equation}\label{eq:adjoint torsion function}
    \torsion_{G, M, \gamma, \mathfrak{o}}^{\mathrm{Ad}} : \Loc_{G, M}^{\reg{\partial\text{-}\mathrm{Ad}}} \to \mathbb{A}^1_k
  \end{equation}
  by $\torsion_{G, M, \gamma, \mathfrak{o}}^{\mathrm{Ad}} (P) \coloneqq 
  \torsion_{M, \gamma, \mathfrak{o}} (P \times^G \mathfrak{g})$,
  where the right-hand side is the torsion function defined in Definition \ref{def:torsion function}.
\end{definition}

\begin{lemma}\label{lemma:adjoint torsion independence on choices}
  The adjoint torsion function \eqref{eq:adjoint torsion function} does not depend on the choice of a volume form $\omega_{\mathfrak{g}}$ on $\mathfrak{g}$.
  It also does not depend on the choice of a homology orientation $\mathfrak{o}$ of $M$ if $\rank G$ is even.
\end{lemma}
\begin{proof}
  Let $S$ be the set of the centers of cells of $M$.
  We take markings on $S$.
  Then we have a trivialization
  $(P \times^G \mathfrak{g})_\sigma \cong \mathfrak{g} \otimes_k A$ for each $\sigma \in \mathsf{cell}_k$.
  Under this trivialization, the $i$-th component of the geometric volume \eqref{eq:geometric volume form at i} is given by 
  \begin{equation*}
    \otimes_{\sigma \in \mathsf{cell}_i} \mathsf{vol}_{\mathcal{L}_\sigma}
    = \otimes_{\sigma \in \mathsf{cell}_i} (\omega_{\mathfrak{g}} \otimes 1)
  \end{equation*}
  Since the Euler characteristic of each connected component of $M$ is zero,
  we have $\sum_i (-1)^i \lvert \mathsf{cell}_i \rvert = 0$.
  Thus the map \eqref{eq:vol complex cell contract} in the definition of the geometric volume
  form does not depend on a choice of $\omega_{\mathfrak{g}}$ by \cref{lemma:smul otinv}.
  The assertion on the homology orientation follows from \cref{lemma:torsion of even local system}.
\end{proof}

\subsection{Local systems over fields}
A \emph{point} of $\Loc_{G, M}$ is an equivalence class of pairs $(K, P)$,
where $K$ is a field extension of $k$ and $P \in \Loc_{G, M}(K)$.
Two pairs $(K, P)$ and $(K', P')$ are equivalent
if there exists a field extension $\Omega$ of both $K$ and $K'$ such that
$\Omega \otimes_K P \cong \Omega \otimes_{K'} P'$ in $\Loc_{G, M}(\Omega)$.
The set of points of $\Loc_{G, M}$ is denoted by $\lvert \Loc_{G, M} \rvert$.

\begin{prop}\label{prop:adjoint regularity criterion}
  Let $(K, P) \in \lvert \Loc_{G, M} \rvert$.
  Then $P$ is boundary-adjoint-regular if and only if
  the following conditions hold:
  \begin{enumerate}
    \item 
    \label{item:adjoint reg H0}
    $H^0 (M; P \times^G \mathfrak{g}) = 0$,
    \item 
    \label{item:adjoint reg H1}
    $\dim_K H^1 (M; P \times^G \mathfrak{g}) = \frac{1}{2} \dim_K H^1 (\partial M; P|_{\partial M} \times^G \mathfrak{g})$,
    \item 
    \label{item:adjoint reg H0 boundary}
    $\dim_K H^0 (\partial M; P|_{\partial M} \times^G \mathfrak{g}) = m \cdot \rank G$.
  \end{enumerate}
\end{prop}
\begin{proof}
  For simplicity, we write $H^i(M)$ for $H^i(M; P \times^G \mathfrak{g})$ in the proof,
  and similarly for the homology.
  First, we note that 
  $\dim_K H^i (M) = \dim_K H_i (M)$ 
  since we have the Killing form on $\mathfrak{g}$,
  which is non-degenerate, that enables us to 
  have a perfect pairing between $H^i (M)$ and $H_i (M)$.
  We also have the long exact sequence
  \begin{align}\label{eq:long exact around 0}
    \cdots \to H^0 (M) \to H^0 (\partial M) \xrightarrow{\alpha} H^1 (M, \partial M) 
    \to H^1 (M) \to H^1 (\partial M) \to \cdots.
  \end{align}
  We now prove the if part.
  By condition \eqref{item:adjoint reg H0}, the Poincar\'e-Lefschetz duality, 
  and the fact that the Euler characteristic of $\partial M$ and $M$ are zero,
  we see that 
  \begin{align*}
    &(\dim_K H^0 (M), \dim_K H^1 (M), \dim_K H^2 (M), \dim_K H^3 (M)) = 
    (0, a, a, 0)\\
    &(\dim_K H^0 (\partial M), \dim_K H^1 (\partial M), \dim_K H^2 (\partial M)) =
    (b, 2b, b)
  \end{align*}
  for some $a, b$.
  To see that $\mathcal{L}$ is boundary-adjoint-regular,
  we need to show the map $H_2 (\partial M) \to H_2 (M)$ is an isomorphism.
  Via the Poincar\'e-Lefschetz duality, the latter map in \eqref{eq:H^0 to H_2} corresponds to the connecting homomorphism $\alpha : H^0 (\partial M) \to H^1 (M, \partial M)$. Hence it suffices to show that $\alpha$ is an isomorphism.
  Since the map $\alpha$ is injective by \eqref{eq:long exact around 0} and condition (1),
  it suffices to show that $\dim_K H^0 (\partial M) = \dim_K H^1 (M, \partial M)$,
  which is equivalent to $a = b$ since 
  we have $\dim_K H^1 (M, \partial M) = \dim_K H_2 (M) = \dim_K H^2 (M) = a$.
  The equality $a = b$ follows from condition \eqref{item:adjoint reg H1}.
  By condition \eqref{item:adjoint reg H0 boundary}, we have $a = b = m \cdot \rank G$.
  This completes the proof of the if part.
  The only if part is proved similarly.
\end{proof}

\begin{remark}\label{rem:adjoint regularity criterion}
  By the exact sequence \eqref{eq:long exact around 0},
  one can show that
  under the conditions \eqref{item:adjoint reg H0},
  the condition \eqref{item:adjoint reg H1} is equivalent to the following:
  \begin{enumerate}
    \item[(2')] the restriction map 
      $H^1 (M; P \times^G \mathfrak{g}) \to H^1 (\partial M; P|_{\partial M} \times^G \mathfrak{g})$
      is injective.
  \end{enumerate}
  Moreover, under the conditions \eqref{item:adjoint reg H0} and \eqref{item:adjoint reg H0 boundary},
  the condition \eqref{item:adjoint reg H1} is equivalent to the following:
  \begin{enumerate}
    \item[(2'')] $\dim_K H^1 (M; P \times^G \mathfrak{g}) = m \cdot \rank G$.
  \end{enumerate}
\end{remark}

\begin{corollary}
  Let $(K, P) \in \lvert \Loc_{G, M}^{\reg{\partial\text{-}\mathrm{Ad}}} \rvert$.
  Then the dimensions of homology groups are given as follows:
  \begingroup
  \renewcommand{\arraystretch}{1.1}
  \begin{align*}
    &\begin{array}{c|cccc}
    i & 0 & 1 & 2 & 3 \\ \hline
    \dim_K H_i(M; P \times^G \mathfrak{g}) & 0 & m \cdot \rank G & m \cdot \rank G & 0
    \end{array}
    \\[0.5ex]
    &\begin{array}{c|ccc}
    i & 0 & 1 & 2 \\ \hline
    \dim_K H_i(\partial M; P|_{\partial M} \times^G \mathfrak{g}) & m \cdot \rank G & 2m \cdot \rank G & m \cdot \rank G
    \end{array}
  \end{align*}
  \endgroup
\end{corollary}

\subsection{Principal embedding}
In this section, we assume that $k$ is algebraically closed.
We also assume that $G$ has a trivial center. 
For any $k$-point $g \in G(k)$, its centralizer
$Z_{G}(g)=\{ x \in G(k) \mid xg=gx\}$ satisfies $\dim_{k} Z_{G}(g) \geq \rank G$. We call $g$ a \emph{regular element} if the equality $\dim_{k} Z_{G}(g) = \rank G$ holds \cite{Steinberg}.

\begin{dfn}
A morphism $\iota: \PGL_2 \to G$ of algebraic groups is called a \emph{principal embedding} if $\iota(k): \PGL_2(k) \to G(k)$ sends the element $\bmtx{1 & 1 \\ 0 & 1}$ to a regular unipotent element.
\end{dfn}
It is known that regular unipotent elements always exist, and unique up to inner automorphisms. 
For the case where $k$ is not necessarily algebraically closed, we say that $\iota : \PGL_2 \to G$ is a principal embedding if its base change to $\bar{k}$ is a principal embedding.

For any integer $n \geq 1$, let $(V_n,\varsigma_n)$ denote the $n$-dimensional complex irreducible representation of $\SL_2$, which is unique up to isomorphism. It provides a morphism $\varsigma_n: \SL_2 \to \SL(V_n)$. It descends to $\varsigma_n: \PGL_2 \to \SL_n(V_n)$ if and only if $n$ is odd. 

For any principal embedding $\iota:\PGL_2 \to G$, let
$\iota^\ast \mathfrak{g}$ be the adjoint representation of $G$, viewed as an $\PGL_2$-module via $\iota$. 

\begin{prop}[{\cite[Corollary 8.7]{Kostant}}]\label{prop:adjoint_decomp}
For any principal embedding $\iota:\PGL_2 \to G$, we have an irreducible decomposition 
\begin{align}\label{eq:adjoint_decomp}
    \iota^\ast \mathfrak{g} \cong \bigoplus_{i=1}^{\rank G} V_{2m_i+1},
\end{align}
where the numbers $(m_i)_{i=1}^{\rank G}$ are exponents of $\mathfrak{g}$ (Table \ref{table:exponents}). 
\end{prop}

A principal embedding $\iota: \PGL_2 \to G$ induces a morphism
\begin{align*}
    \iota: \Loc_{\PGL_2,M} \to \Loc_{G,M}. 
\end{align*}
In the next subsection, we investigate the image of the geometric representation under $\iota$. 

\subsection{Regularity of the geometric local system}
In this section, we take the base field $k=\bC$.
Let $M$ be a compact $3$-manifold whose interior is an orientable complete hyperbolic $3$-manifold. 
Let $\mathsf{geom} \in \Loc_{\PGL_2,M}(\bC)$ denote the $\PGL_2$-local system 
associated with the complete hyperbolic structure on the 
interior of $M$.

We investigate the following local systems on $M$:
\begin{itemize}
    \item $\mathcal{L}_{\hol}^{(n)}\coloneq \mathsf{geom} \times^{\PGL_2} V_n \in \Loc_{\SL,M}$ for $n \geq 3$ odd, where $V_n$ is viewed as an $\PGL_2$-module via $\varsigma_n$. 
    \item $\mathfrak{g}_{\hol}\coloneq \mathsf{geom} \times^{\PGL_2} \iota^\ast\mathfrak{g} =\iota(\mathsf{geom}) \times^{G} \mathfrak{g}$ for a principal embedding $\iota$. 
\end{itemize}

\begin{thm}[{\cite[Theorem 0.1]{MFP_tw}}]\label{thm:MFP_Lagrangian}
  Let $M$ be a compact $3$-manifold whose interior is an orientable complete, non-elementary hyperbolic 3-manifold.
  Let $n \geq 3$ odd. Then the inclusion $\partial M \to M$ induces
  an injection,
  \begin{equation*}
    H^1(M;\mathcal{L}_{\hol}^{(n)}) \to H^1(\partial M;\mathcal{L}_{\hol}^{(n)})
  \end{equation*}
  with $\dim_\bC H^1(M;\mathcal{L}_{\hol}^{(n)}) = \frac{1}{2} \dim_\bC H^1(\partial M;\mathcal{L}_{\hol}^{(n)})$,
  and an isomorphism
  \begin{equation*}
    H^2(M;\mathcal{L}_{\hol}^{(n)}) \xrightarrow{\cong} H^2(\partial M;\mathcal{L}_{\hol}^{(n)}).
  \end{equation*}
\end{thm}

\begin{thm}[{\cite[Corollary 3.7]{MFP_tw}}]\label{thm:MFP_H0}
  Let $M$, $n$ be as in \cref{thm:MFP_Lagrangian}. Then
  \begin{align*}
    H^0 (M;\mathcal{L}_{\hol}^{(n)}) = 0.
  \end{align*}
\end{thm}

\begin{thm}[{\cite[Corollary 3.6]{MFP_tw}}]\label{thm:MFP_H0_boundary}
  Let $M$, $n$ be as in \cref{thm:MFP_Lagrangian}, and suppose $\partial M$ consists of $m$ tori. Then
  \begin{align*}
    \dim_\bC H^0(\partial M;\mathcal{L}_{\hol}^{(n)}) = m .
  \end{align*}
\end{thm}

In~\cite{MFP_tw}, there is also an explicit description of the cases when $n$ is even,
where they consider a $\SL_2$-local system $\widetilde{\hol}$ that is a lift of $\mathsf{geom}$.
Also, we focus on the case where $\partial M$ consists only of torus cusps. See \cite{MFP_tw} for the general case.
We get the following from \cref{thm:MFP_Lagrangian,thm:MFP_H0}:
\begin{cor}\label{cor:twisted_dimension}
  Let $M$, $n$ be as in \cref{thm:MFP_Lagrangian}. Then
  \begin{enumerate}
    \item the map
    \begin{equation*}
      H^1(M;\mathfrak{g}_{\hol}) \to H^1(\partial M;\mathfrak{g}_{\hol})
    \end{equation*} 
    is injective with 
    $\dim_\bC H^1(M;\mathfrak{g}_{\hol}) = \frac{1}{2} \dim_\bC H^1(\partial M;\mathfrak{g}_{\hol})$,
    \item $H^0(M;\mathfrak{g}_{\hol}) = 0$,
    \item if $\partial M$ consists of $m$ tori, we have
      \begin{align*}
        \dim_\bC H^0(\partial M;\mathfrak{g}_{\hol}) = m \cdot \rank G.
      \end{align*}
  \end{enumerate}
\end{cor}
\begin{proof}
The first assertion follows from the \cref{prop:adjoint_decomp,thm:MFP_Lagrangian}.
The second follows from \cref{prop:adjoint_decomp,thm:MFP_H0}.
By \cref{prop:adjoint_decomp,thm:MFP_H0_boundary}, we get
\begin{align*}
    \dim_\bC H^0(\partial M;\mathfrak{g}_{\hol}) = \sum_{i=1}^{\rank G} \dim_\bC H^0(\partial M;\mathcal{L}_{\hol}^{(2m_i+1)}) =  \sum_{i=1}^{\rank G} m = m \cdot \rank G,
\end{align*}
as desired.
\end{proof}

\begin{thm}\label{thm:geom_regular}
$\iota(\mathsf{geom})$ is boundary-adjoint-regular. 
\end{thm}
\begin{proof}
The conditions in \cref{prop:adjoint regularity criterion} hold by \cref{cor:twisted_dimension}.
\end{proof}

\begin{thm}\label{thm:adjoint_torsion_hol}
We have
\begin{align*}
  \torsion_{G, M,\gamma,\mathfrak{o}}^{\Ad}(\iota(\mathsf{geom})) 
  = \prod_{i=1}^{\rank G}
  \torsion_{M,\gamma,\mathfrak{o}}(\mathcal{L}_{\hol}^{(2m_i + 1)}).
\end{align*}
\end{thm}
\begin{proof}
  This follows from the decomposition of $\iota^* \mathfrak{g}$ (\cref{prop:adjoint_decomp}) and the multiplicativity of torsion (\cref{lemma:torsion of direct sum}).
\end{proof}

\begin{cor}\label{cor:torsion_geom_nonzero}
  Suppose that each component of $\gamma$ is not null-homotopic in $M$.
  Then 
  \begin{equation*}
    \torsion_{G, M,\gamma,\mathfrak{o}}^{\Ad}(\iota(\mathsf{geom})) \neq 0.
  \end{equation*}
\end{cor}
\begin{proof}
  By \cref{thm:adjoint_torsion_hol}, it suffices to show that $\torsion_{M,\gamma,\mathfrak{o}}(\mathcal{L}_{\hol}^{(2m_i +1)}) \neq 0$ for any $i$.
  Since every positive integer occurs as an exponent of $\mathfrak{pgl}_n$ for some $n$
  (see Table \ref{table:exponents}),
  it suffices to show that the assertion in the case $\mathfrak{g} = \mathfrak{pgl}_n$.
  This is proved in \cite[Theorem 3.4]{KitayamaTerashima}.
\end{proof}

\section{Computation}\label{sec:computation}
\subsection{Determinant of complexes in terms of bases}
To compute torsions explicitly,
we will write down the isomorphism in Definition~\ref{theorem:det C to det H} explicitly in terms of bases.
A well-known formula for the torsion of complexes is reproduced (\cref{theorem:torsion explicit}), with an appropriate sign correction
determined by our conventions.
See Lemma \ref{lemma:det homology explicit} and \ref{lemma:det homology explicit sign} below.

We first fix an inverse structures on $\mathcal{P}_A$
following the sign convention in \cite[\S 1.2]{Nicolaescu}.
For $(L, \alpha) \in \mathcal{P}_A$, we define $(L, \alpha)^{\otinv} \coloneqq (L^{\vee}, -\alpha)$ where
$L^{\vee} \coloneqq \Hom_A (L, A)$.
For $L \in \mathcal{P}_A$, the map $\varepsilon_L : L \otimes L^{\otinv} \to 1$
is defined as the evaluation map $L \otimes L^{\vee} \cong A$ 
multiplied by the sign $(-1)^{\frac{1}{2} \lvert L \rvert (\lvert L \rvert + 1)}$.
For a morphism $f : L \to M$, we define $f^{\otinv} : L^{\otinv} \to M^{\otinv}$ by
$f^{\otinv} \coloneqq (f^{\vee})^{-1}$,
where $f^{\vee} : M^{\vee} \to L^{\vee}$ is the dual map of $f$.
For $L, M \in \mathcal{P}_A$,
the map $\theta_{L,M} : (L \otimes M)^{\otinv} \to L^{\otinv} \otimes M^{\otinv}$
is defined as the canonical isomorphism between the underlying modules, without any sign correction.
We define the map $1^{\otinv} \to 1$ as the canonical isomorphism $A^{\vee} \cong A$.

For $L \in \mathcal{P}_A$ and a basis $\ell$ of $L$, 
we denote by $\ell^{\otinv}$ the dual basis of $\ell$,
which is an element of the underlying module $L^{\vee}$ of $L^{\otinv}$.

We next fix a determinant functor on $(\mathsf{proj}(A), \mathrm{iso})$.
We define $\det : \mathsf{proj}(A) \to \mathcal{P}_A$ by
\begin{align}
  \det V \coloneqq (\wedge^{\rank V} V,\ \rank V).
\end{align}
For an isomorphism $f : U \cong V$ in $\mathsf{proj}(A)$,
we define $\det f : \det U \to \det V$ as the map induced by $f$ on the top exterior powers.
For a family of vectors $v = (v_\lambda)_{\lambda \in \Lambda}$ of $V$ 
with the linearly ordered index set $\Lambda$, we define
\begin{align}\label{eq:wedge vector}
  \wedge v \coloneqq v_{\lambda_1} \wedge \cdots \wedge v_{\lambda_r} \in \wedge^{r} V,
\end{align}
where $\lambda_1 < \cdots < \lambda_r$ and $r = \lvert \Lambda \rvert$.
When $v$ is a basis, we have $\rank V = \lvert \Lambda \rvert$ and $\wedge v$ is a basis of $\det V$.

For a short exact sequence of finitely generated projective $A$-modules
\begin{align}
  \Sigma = \quad 0 \to U \xlongrightarrow{f} V \xlongrightarrow{g} W \to 0,
\end{align}
we define the map 
\begin{align}\label{eq:det exact}
  \det \Sigma : \det V \cong \det U \otimes \det W 
\end{align}
such that 
\begin{align}
  (\det \Sigma) ((\wedge f u) \wedge (\wedge w)) = \wedge u \otimes \wedge (g w)
\end{align}
holds locally on $\Spec A$ for a basis $u$ of $U$ and a lift $w$ of a basis of $W$.

We now give an explicit formula for the isomorphism in Definition \ref{theorem:det C to det H}.
Let $C_\bullet = (C_n,\partial_n)_{n \in \mathbb{Z}}$ be a bounded complex of 
finitely generated projective $A$-modules, and 
suppose that $c_i = (c_{i, \lambda})_{\lambda \in \Lambda_{\mathsf{c}, i}}$ is a basis of 
$C_i$ such that the index set $\Lambda_{\mathsf{c}, i}$ is linearly ordered. 
Suppose that we are given
\begin{enumerate}
  \item a family of vectors $b_i = (b_{i, \lambda})_{\lambda \in \Lambda_{\mathsf{b}, i}}$ of $C_{i+1}$ with 
    linearly ordered $\Lambda_{\mathsf{b}, i}$ for each $i$,
  \item a family of vectors $h_i = (h_{i, \lambda})_{\lambda \in \Lambda_{\mathsf{h}, i}}$ of $\Ker \partial_{i}$ with
    linearly ordered $\Lambda_{\mathsf{h}, i}$ for each $i$,
\end{enumerate}
such that 
\begin{enumerate}[resume]
  \item $\partial_{i+1} b_i \coloneqq (\partial_{i+1} b_{i, \lambda})_{\lambda \in \Lambda_{\mathsf{b}, i}}$ is a basis of $\Img \partial_{i+1}$ for each $i$,
  \item $[h_i] \coloneqq ([h_{i, \lambda}])_{\lambda \in \Lambda_{\mathsf{h}, i}}$ is a basis of $H_i$ ($\coloneq \Ker \partial_i / \Img \partial_{i+1}$) for each $i$.
\end{enumerate}
Then we see that $\partial_{i+1} b_{i} \sqcup h_i \sqcup b_{i-1} : 
\Lambda_{\mathsf{b}, i} \sqcup \Lambda_{\mathsf{h}, i} \sqcup \Lambda_{\mathsf{b}, i-1} \to C_i$ 
is a basis of $C_i$.
Let 
\begin{equation}
  \sigma_i : 
  \Lambda_{\mathsf{c}, i}
    \cong \Lambda_{\mathsf{b}, i} \sqcup \Lambda_{\mathsf{h}, i} \sqcup \Lambda_{\mathsf{b}, i-1} 
\end{equation}
be the unique order-preserving bijection.

For an $A$-module $V$ and a family of vectors $v : I \to V$ with a finite index set $I$,
we denote by $\lin v$ the linear map $A^I \to V$ induced by $v$.
When $v$ is a basis, the map $\lin v$ is an isomorphism.

\begin{definition}\label{def:ratio notation}
  For $L \in \mathcal{P}_A$,and bases $u$ and $v$ of $L$,
  we define 
  \begin{align}
    [u / v] \coloneqq ((\lin v)^{-1} \circ \lin u) (1) \in A^{\times}.
  \end{align}
  Moreover, for a finitely generated projective $A$-module $V$
  and bases $u, v : I \to V$,
  we define 
  \begin{align}
    [u / v] \coloneqq [\wedge u / \wedge v] = \det \bigl((((\lin v)^{-1} \circ \lin u) (e_i))_j \bigr)_{i,j} \in A^{\times},
  \end{align}
  where 
  $(e_i)_{i \in I}$ is the standard basis of $A^I$, and
  $\det$ is the determinant of $I \times I$ matrices over $A$.
\end{definition}

\begin{lemma}\label{lemma:ratio basic}
  Let $L, L' \in \mathcal{P}_A$. 
  Let $u$ and $v$ be bases of $L$,
  and $u'$ and $v'$ be bases of $L'$.
  Then we have the following:
  \begin{enumerate}
    \item $u = [u / v] \cdot v$
    \item $[u^{\otinv} / v^{\otinv}] = [v / u]^{-1}$
    \item $[u \otimes u' / v \otimes v'] = [u / v] \cdot [u' / v']$
  \end{enumerate}
\end{lemma}
\begin{proof}
  These easily follow from Definition \ref{def:ratio notation}.
\end{proof}

\begin{lemma}\label{lemma:det homology explicit}
We have
  \begin{align}
    \Phi (\otimes_i (\wedge c_i)^{(\otinv)^i}) =
    (-1)^{\alpha}
    \prod_{i} [c_i / (\partial_{i+1} b_{i} \sqcup h_i \sqcup b_{i-1}) \circ \sigma_i]^{(-1)^i} \cdot
    \otimes_i (\wedge [h_i])^{(\otinv)^i},
  \end{align}
  where $\Phi$ is the isomorphism $\det C_\bullet \cong \det H_\bullet (C_\bullet)$
  given in Definition \ref{theorem:det C to det H},
  and $(-1)^\alpha$ is an appropriate sign factor.
\end{lemma}
\begin{proof}
We first recall the definition of the isomorphism $\Phi$ in \cite[Definition 3.1]{Knudsen}.
Let $B_\bullet$ be the complex defined by $B_i = \Img \partial_{i+1}$, with all differentials zero.
Similarly, let $Z_\bullet$ be the complex defined by $Z_i = \Ker \partial_i$, with all differentials zero.
We have an exact sequence
\begin{align}\label{eq:exact Z C B}
  0 \to Z_\bullet \to C_\bullet \to B[1]_\bullet \to 0,
\end{align}
where $(B[1])_i = B_{i-1}$.
We also write $H_\bullet\coloneq H_\bullet (C_\bullet)$, for simplicity.
Let $\iota : B_\bullet \to Z_\bullet$ be the inclusion map, and $\Cone (\iota)_\bullet$ be the mapping cone of $\iota$.
We have $\Cone (\iota)_i = Z_i \oplus B_{i-1}$ and $\partial_i (z, b) = (\iota(b), 0)$.
By definition, we have an exact sequence
\begin{align}\label{eq:exact Z Cone B}
  0 \to Z_\bullet \to \Cone (\iota)_\bullet \to B[1]_\bullet \to 0.
\end{align}
The map $\Cone (\iota)_\bullet \to H_\bullet$ defined by $(z, b) \mapsto [z]$ is an epi quasi-isomorphism.
Thus, we have an exact sequence
\begin{align}\label{eq:exact Q Cone H}
  0 \to Q_\bullet  \to \Cone (\iota)_\bullet \to H_\bullet \to 0,
\end{align}
with an acyclic $Q_\bullet$. Explicitly, $Q_i=B_i \oplus B_{i-1}$.
Since $Q_\bullet$ is acyclic, we have an isomorphism
\begin{align}\label{eq:det Q}
  \det Q_\bullet \cong 1
\end{align}
by \cite[Definition 2.24 (c)]{Knudsen}.
The isomorphism $\Phi$ is defined as the composition
\begin{alignat}{2}
  \label{eq:det C to det H step 1}
  \det C_\bullet &\cong 
  \det Z_\bullet \otimes \det B[1]_\bullet \quad &&\text{by determinant of \eqref{eq:exact Z C B}}\\
  \label{eq:det C to det H step 2}
  &\cong 
  \det \Cone (\iota)_\bullet \quad &&\text{by determinant of \eqref{eq:exact Z Cone B}}\\
  \label{eq:det C to det H step 3}
  &\cong
  \det Q_\bullet \otimes \det H_\bullet \quad &&\text{by determinant of \eqref{eq:exact Q Cone H}}\\
  \label{eq:det C to det H step 4}
  &\cong 
  \det H_\bullet &&\text{by \eqref{eq:det Q}},
\end{alignat}
where isomorphisms associated with short exact sequences are given in each degree by \eqref{eq:det exact} (see \cite[Definition 2.24 (b)]{Knudsen}).

We now compute $\Phi (\otimes_i (\wedge c_i)^{(\otinv)^i})$.
By Lemma \ref{lemma:ratio basic}, it suffices to show
\begin{align}
  \Phi(\otimes_i (\wedge (\partial_{i+1} b_i \sqcup h_i \sqcup b_{i-1}) \circ \sigma_i)^{(\otinv)^i}) = 
  (-1)^\alpha \cdot \otimes_i (\wedge [h_i])^{(\otinv)^i}.
\end{align}
Up to signs, we compute the left-hand side as follows:
\begin{alignat}{2}
  \nonumber
  \otimes_i &(\wedge ((\partial_{i+1} b_i \sqcup h_i \sqcup b_{i-1}) \circ \sigma_i))^{(\otinv)^i}
  \\
  &=
  \label{eq:det C to det H app step 1}
  \otimes_i (\wedge (\partial_{i+1} b_i \sqcup h_i \sqcup b_{i-1}))^{(\otinv)^i} 
  \\
  &\xmapsto{\eqref{eq:det C to det H step 1}}
  \label{eq:det C to det H app step 2}
  \bigl( \otimes_i (\wedge (\partial_{i+1} b_i \sqcup h_i))^{(\otinv)^i} \bigr) \otimes
    \bigl( \otimes_i (\wedge \partial_{i} b_{i-1})^{(\otinv)^{i}} \bigr)
  \\
  &\xmapsto{\eqref{eq:det C to det H step 2}}
  \label{eq:det C to det H app step 3}
  \otimes_i (\wedge (\partial_{i+1} b_i \sqcup h_i \sqcup \partial_{i} b_{i-1}))^{(\otinv)^i}
  \\
  &= \otimes_i (\wedge (\partial_{i+1} b_i \sqcup \partial_{i} b_{i-1} \sqcup h_i))^{(\otinv)^i} 
  \label{eq:det C to det H app step 4}
  \\
  &\xmapsto{\eqref{eq:det C to det H step 3}}
  \label{eq:det C to det H app step 5}
  \bigl(\otimes_i (\wedge (\partial_{i+1} b_i \sqcup \partial_{i} b_{i-1}))^{(\otinv)^i} \bigr)
    \otimes 
    \bigl(\otimes_i (\wedge [h_i])^{(\otinv)^i} \bigr) 
  \\  
  &\xmapsto{\eqref{eq:det C to det H step 4}}
  \label{eq:det C to det H app step 6}
  \otimes_i (\wedge [h_i])^{(\otinv)^i}
\end{alignat}
\end{proof}

We now compute the sign factor.
To do this, we interpret the tensor order as follows:
\begin{align}
  \label{eq:explicit tensor order c}
  \otimes_i c_i^{(\otinv)^i}
  &= \cdots \otimes c_2 \otimes c_1^{\otinv} \otimes c_0 \otimes \cdots,
  \\
  \label{eq:explicit tensor order h}
  \otimes_i [h_i]^{(\otinv)^i}
  &= \cdots \otimes [h_2] \otimes [h_1]^{\otinv} \otimes [h_0] \otimes \cdots.
\end{align}

\begin{lemma}[{cf.~\cite[Proposition 1.17]{Nicolaescu}}]
  \label{lemma:det homology explicit sign}
  The sign factor $(-1)^\alpha$ in Lemma \ref{lemma:det homology explicit} is given by
  \begin{equation}
    \begin{split}
      \alpha \equiv
      \frac{1}{2} \sum_i \rank B_{i} \cdot (\rank B_{i} - (-1)^{i}) 
    \end{split}
  \end{equation}
  modulo $2$.
\end{lemma}
\begin{proof}
We compute the sign factor arising in each step of \eqref{eq:det C to det H app step 1}--\eqref{eq:det C to det H app step 6}.
The sign at \eqref{eq:det C to det H app step 1} is trivial since $\sigma_i$ is order-preserving.
The signs at \eqref{eq:det C to det H app step 2} and \eqref{eq:det C to det H app step 3} cancel out together.
The sign at \eqref{eq:det C to det H app step 4} is $(-1)^{\alpha_1}$ 
where
\begin{align*}
  \alpha_1 \equiv \sum_i \rank B_{i-1} \cdot \rank H_i.
\end{align*}
The sign at \eqref{eq:det C to det H app step 5} is $(-1)^{\alpha_2}$, 
where
  \begin{align*}
    \alpha_2 \equiv 
    \sum_{i < j} (\rank B_{i} + \rank B_{i-1}) \cdot \rank H_j.
  \end{align*}
A short calculation shows that $\alpha_1 + \alpha_2 = 0$ modulo $2$.

The sign at \eqref{eq:det C to det H app step 6} is $(-1)^{\alpha_3}$, 
where
\begin{align*}
  \alpha_3 &\equiv 
  \sum_{i < j} \rank B_{i} \cdot \rank B_{j-1}
  + \sum_{i < j} \rank B_{i-1} \cdot \rank B_{j-1}
  + \frac{1}{2} \sum_{i} \rank B_{i} \cdot (\rank B_{i} + (-1)^{i}) \\
  &\equiv
  \sum_{i} \rank B_{i} \cdot \rank B_{i}
  + \frac{1}{2} \sum_{i} \rank B_{i} \cdot (\rank B_{i} + (-1)^{i}) \\
  &\equiv
  \frac{1}{2} \sum_{i} \rank B_{i} \cdot (\rank B_{i} - (-1)^{i}).
\end{align*}
where in the first line, the three terms come from 
the following three isomorphisms in the definition of $\det Q_\bullet \to 1$ (see \cite[Definition 2.24 (c)]{Knudsen}) respectively:
\begin{alignat*}{2}
  \det Q_\bullet &\cong 
  \bigotimes_i (\det B_{i-1})^{(\otinv)^i} \otimes
  \bigotimes_j (\det B_{i})^{(\otinv)^{j}} \\
  &\cong 
  \bigotimes_i ( (\det B_{i-1})^{(\otinv)^i} \otimes (\det B_{i-1})^{(\otinv)^{i-1}}) \\
  &\cong 1.
\end{alignat*}
\end{proof}

\subsection{Explicit formula of the torsion}
Suppose that $M$ is a link exterior complex and 
$\gamma$ is a peripheral loop in $M$. 
Suppose that we have $\partial M = \bigsqcup_{i=1}^m T_i$,
where each $T_i$ is a connected component of $\partial M$,
and also have $\gamma = \bigsqcup_{i=1}^m \gamma_i$ such that each $\gamma_i$ is on $T_i$.
Let $k \subset K$ be a field extension,
$n$ be a natural number,
and $\mathcal{L} \in \Loc_{\SL, M}(K)$ be a local system of rank $n$.
Let $\ell_{\sigma} = (\ell_{\sigma, i})_{i=1, \dots, n}$ be a $K$-basis of $\mathcal{L}_{\sigma}$
for each $\sigma \in \mathsf{cell}(M)$
such that $\mathsf{vol}_{\sigma} (\ell_{\sigma, 1} \wedge \dots \wedge \ell_{\sigma, n}) = 1$.
In other words, $(\mathcal{L}, \ell) \in \Loc_{\SL, M, S, n}(K)$ is a marked local system
in the sense of \cref{def:moduli of marked lin local systems},
where $S$ is the set of the centers of all cells of $M$.
Then we have that
\begin{equation}\label{eq:geom_basis}
  c_k \coloneqq \ell_{\sigma_1^{(k)}, 1} \otimes \sigma_1^{(k)}, \dots, \ell_{\sigma_1^{(k)}, n} \otimes \sigma_1^{(k)},
  \dots, \ell_{\sigma_{c_k}^{(k)}, 1} \otimes \sigma_{c_k}^{(k)}, \dots, \ell_{\sigma_{c_k}^{(k)}, n} \otimes \sigma_{c_k}^{(k)}
\end{equation}
is a $K$-basis of $C_k(M; \mathcal{L})$,
where $\mathsf{cell}_k(M) = \{ \sigma_1^{(k)}, \dots, \sigma_{c_k}^{(k)} \}$.
Suppose that $\dim_K H^0(T_i; \mathcal{L}|_{T_i}) = r$ for $i = 1, \dots, m$,
and $(v_{i, j})_{j=1,\dots, r}$ is a $K$-basis of $H^0(T_i; \mathcal{L}|_{T_i})$.
Then $\mathcal{L}$ is $\gamma$-regular if and only if
the following conditions hold:
\begin{enumerate}
  \item $H_0 (M; \mathcal{L}) = H_3 (M; \mathcal{L}) = 0$,
  \item $\dim_K H_2(M; \mathcal{L}) = \dim_K H_1(M; \mathcal{L}) = m r$,
  \item $v_{i, 1} \cap [T_i], \dots, v_{i, r} \cap [T_i]$ is linearly independent over $K$ in
    $H_2(M; \mathcal{L})$ for each $i = 1, \dots, m$,
  \item $v_{i, 1} \cap [\gamma_i], \dots, v_{i, r} \cap [\gamma_i]$ is linearly independent over $K$ in
    $H_1(M; \mathcal{L})$ for each $i = 1, \dots, m$.
\end{enumerate}
We assume that these conditions are satisfied,
and define bases of $H_1(M; \mathcal{L})$ and $H_2(M; \mathcal{L})$ by
\begin{align}\label{eq:h1 h2 bases}
  h_1 \coloneqq (v_{i, j} \cap [\gamma_i])_{i,j}, \quad
  h_2 \coloneqq (v_{i, j} \cap [T_i])_{i,j}.
\end{align}
We also set $h_0 = h_n \coloneqq \emptyset$ for $n \ge 3$.

\begin{theorem}\label{theorem:torsion explicit}
  We follow the above setting.
  Let $b_k$ be an ordered family of vectors in $C_{k+1}(M; \mathcal{L})$ for each $k$ such that 
  $\partial_{k+1} b_k$ is an ordered basis of $B_k \coloneqq \Img \partial_{k+1}$.
  Then the torsion is given by 
  \begin{equation}\label{eq:torsion explicit}
    \torsion_{M, \gamma, \mathfrak{o}} (\mathcal{L}) =
    \epsilon_\mathfrak{o,\mathcal{L}} \cdot
    \epsilon \cdot
    \prod_{k \ge 0} [c_k / (\partial_{k+1} b_{k} \sqcup h_k \sqcup b_{k-1})]^{(-1)^k}
  \end{equation}
  where $\epsilon_\mathfrak{o, \mathcal{L}}$ is the sign determined from a 
  homology orientation $\mathfrak{o}$ by \eqref{eq:orientation sign},
  and $\epsilon = (-1)^{\alpha}$ is the sign given by 
  \begin{equation}\label{eq:torsion explicit sign}
    \alpha = \frac{1}{2} \sum_i \rank B_{i} \cdot (\rank B_{i} - (-1)^{i}) 
    + \sum_{\substack{i : \mathrm{even},\ j : \mathrm{odd} \\ i < j}} \rank C_i \cdot \rank C_j.
  \end{equation}
  Moreover, if $n$ is even, then $\epsilon_{\mathfrak{o}, \mathcal{L}} = 1$.
\end{theorem}
\begin{proof}
  Noting that the Porti form satisfies $h_{\gamma} (h_2 \otimes h_1^{\otinv}) = 1$,
the equation \eqref{eq:torsion explicit} except for the sign factor follows from 
\cref{lemma:det homology explicit}.
For the sign factor,
the first term in \eqref{eq:torsion explicit sign} comes from
\cref{lemma:det homology explicit sign},
and the second term comes from comparing the tensor orders
\begin{equation*}
  \dots \otimes c_5^{\otinv} \otimes c_4 \otimes c_3^{\otinv} \otimes c_2 \otimes c_1^{\otinv} \otimes c_0
  \quad \text{and} \quad
  (\dots \otimes c_4 \otimes c_2 \otimes c_0) \otimes ( \dots c_5 \otimes c_3 \otimes c_1)^{\otinv}
\end{equation*}
in \eqref{eq:explicit tensor order c} and 
\eqref{eq:vol complex degree contract}, respectively.
The last assertion follows from \cref{lemma:torsion of even local system}.
\end{proof}

\subsection{Projective general symplectic group}

Let $\GSp_{2n}$ be the general symplectic group defined by
\begin{equation*}
    \GSp_{2n} (A) \coloneqq \{ (g, \lambda) \in \GL_{2n} (A) \times A^{\times} 
      \mid g^\top J g = \lambda J \},
\end{equation*}
where the symplectic form $J$ is chosen to be
\begin{equation*}
  J \coloneqq
  \begin{pmatrix}
    0 & \bar{J}_n \\
    -\bar{J}_n^\top & 0
  \end{pmatrix},
  \qquad
  \bar{J}_n \coloneqq
  \begin{pmatrix}
    &&& 1\\
    && -1 & \\
    & \reflectbox{$\ddots$} && \\
    (-1)^{n-1} &&&
  \end{pmatrix}.
\end{equation*}
The second projection defines the similitude character $\lambda: \GSp_{2n} \to \mathbb{G}_m$. We have the short exact sequence
\begin{align*}
    1 \to \Sp_{2n} \to \GSp_{2n} \to \mathbb{G}_m \to 1,
\end{align*}
where $\Sp_{2n}$ is the symplectic group defined by
\begin{align*}
    \Sp_{2n} (A) \coloneqq \{ g \in \GL_{2n} (A) \mid g^\top J g =  J \}.
\end{align*}
Let $\PGSp_{2n} \coloneqq \GSp_{2n} / Z(\GSp_{2n})$,
which is a split adjoint group of type $C_n$. 
For any field extension $k \subset K$, the set of $K$-points are given by
\begin{equation*}
    \PGSp_{2n} (K) \cong 
    \GSp_{2n} (K) / Z(\GSp_{2n})(K) \cong
    \GSp_{2n} (K) / 
    \langle \text{$(g, \lambda) \sim (c g, c^2 \lambda)$ for $c \in K^\times$} \rangle.
\end{equation*}
The Lie algebras of $\GSp_{2n}$, $\PGSp_{2n}$, and $\Sp_{2n}$ are given by 
\begin{align*}
    &\fgsp_{2n} = \{ (X, \lambda) \in \fgl_{2n} \times k \mid 
        X^\top J + J X = \lambda J \}, \\
    &\fpgsp_{2n} = \fgsp_{2n} / \langle \text{$(X, \lambda) \sim (X + c I_{2n}, \lambda + 2c)$ for $c \in k$} \rangle, \\
    &\fsp_{2n} = \{ X \in \fgl_{2n} \mid 
        X^\top J + J X = 0 \}
\end{align*}
respectively.

The morphism $\Sp_{2n} \to \GSp_{2n} \to \PGSp_{2n}$ induces an isomorphism 
\begin{align*}
    \mathfrak{sp}_{2n} \cong \mathfrak{pgsp}_{2n}, \quad X \mapsto [(X,0)]
\end{align*}
of Lie algebras. We will identify these Lie algebras.  
We will often denote $(g, 1) \in \GSp_{2n}(K)$ simply by $g$, which lies in the image of $\Sp_{2n}(K)$.

\subsection{Figure-eight knot complement}

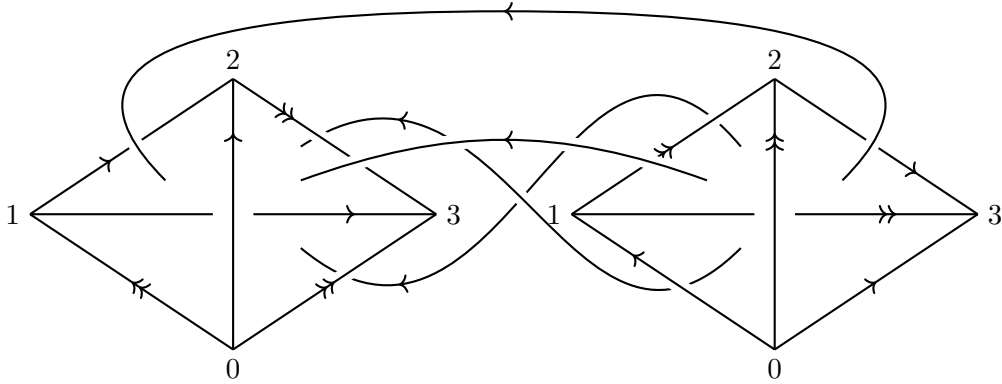
\begin{figure}[t]
  \centering
\begin{tikzpicture}[scale=0.9]
\coordinate (0) at (0,-2);
\coordinate (1) at (-3,0);
\coordinate (2) at (0,2);
\coordinate (3) at (3,0);
\draw[thick,->>-] (0) -- (1);
\draw[thick,->>-] (0) -- (3);
\draw[thick,->-={0.4}{}] (1) -- (2);
\draw[thick,->-={0.8}{}] (1) -- (3);
\crs{0,0}{0.3cm}
\draw[thick,->-={0.8}{}] (0) -- (2);
\draw[thick,->>-={0.3}{}] (2) -- (3);
\node[below] at (0) {$0$};
\node[left] at (1) {$1$};
\node[above] at (2) {$2$};
\node[right] at (3) {$3$};

\begin{scope}[xshift=8cm]
\coordinate (0') at (0,-2);
\coordinate (1') at (-3,0);
\coordinate (2') at (0,2);
\coordinate (3') at (3,0);
\draw[thick,->-={0.7}{}] (0') -- (1');
\draw[thick,->-] (0') -- (3');
\draw[thick,->>-] (1') -- (2');
\draw[thick,->>-={0.8}{}] (1') -- (3');
\crs{0,0}{0.3cm}
\draw[thick,->>-={0.8}{}] (0') -- (2');
\draw[thick,->-={0.7}{}] (2') -- (3');
\node[below] at (0') {$0$};
\node[left] at (1') {$1$};
\node[above] at (2') {$2$};
\node[right] at (3') {$3$};
\end{scope}

\begin{pgfonlayer}{top}
\draw[white,line width=0.2cm] (8+1,0.5) ..controls (8+3,2.5) and (8,3).. (4,3) ..controls (0,3) and (-3,2.5).. (-1,0.5);
\draw[thick,->-] (8+1,0.5) ..controls (8+3,2.5) and (8,3).. (4,3) node[above]{} ..controls (0,3) and (-3,2.5).. (-1,0.5);
\draw[white,line width=0.2cm] (8-1,0.5) to[bend right=20] (1,0.5);
\draw[thick,->-] (8-1,0.5) to[bend right=20] node[midway,above]{} (1,0.5);
\end{pgfonlayer}

\begin{pgfonlayer}{bg}
\draw[thick,->-={0.8}{}] (8-0.5,1) ..controls (8-3,4) and (4,-3).. node[pos=0.8,below]{} (1,-0.5);
\draw[white,line width=0.2cm] (0) -- (3);
\draw[white,line width=0.2cm] (1') -- (2');
\draw[white,line width=0.2cm] (8-0.5,-0.5) ..controls (8-3,-3) and (4,3).. (1,1);
\draw[thick,->-={0.8}{}] (8-0.5,-0.5) ..controls (8-3,-3) and (4,3).. node[pos=0.8,above]{} (1,1);
\draw[white,line width=0.2cm] (2) -- (3);
\draw[white,line width=0.2cm] (0') -- (1');
\end{pgfonlayer}

\end{tikzpicture}
  \caption{Ideal triangulation of the figure-eight knot complement.}
  \label{fig:figure_eight_triangulation}
\end{figure}

Now let $k = \mathbb{Q}$.
Let $4_1 \subset S^3$ be the figure-eight knot, and $M \coloneqq S^3 \setminus \nu (4_1)$ 
its exterior, where $\nu (4_1)$ denotes the tubular neighborhood of $4_1$.
We consider an ideal triangulation of the interior of $M$ shown in \cref{fig:figure_eight_triangulation},
which is considered in \cite[Section 9.5]{Zickert}
to give the Ptolemy relations associated with $M$. 
The associated truncated triangulation of $M$ shown in \cref{fig:figure eight cells} gives a CW structure on $M$.
The cells in $M$ are given by
\begin{itemize}
  \item $0$-cells: $v_0, v_1, v_2, v_3$,
  \item $1$-cells: $E_0, E_1$ (long edges), $e_0, e_1, \dots, e_{11}$ (short edges),
  \item $2$-cells: $H_0$, $H_1$, $H_2$, $H_3$ (hexagons), $t_0, t_1, \dots, t_{7}$ (triangles),
  \item $3$-cells: $B_0, B_1$ (tetrahedra).
\end{itemize}
We endow each $\mathsf{cell}_i$ a linear order as listed above. 

The boundary $\partial M$ is a single torus whose triangulation is shown in \cref{fig:figure-eight boundary torus}.
We also consider the simple closed curve $\mu \coloneqq - e_1$ on $\partial M$, which 
represents the meridian of $4_1$.
With this CW structure, $M$ is a link exterior complex in the sense of \cref{dfn:link exterior complex},
and
$\mu$ is a peripheral loop on $M$.

We denote by $S$ the set of the centers of cells in $M$,
and by $S_0$ the set of $0$-cells in $M$.
By solving the Ptolemy relations associated with the ideal triangulation,
we obtain two concrete examples of marked $\PGSp_4$-local systems on $M$ 
with markings at $S_0$,
whose coefficients are in number fields \cite[Section 9.5]{Zickert}.
These two marked local systems are boundary-unipotent, that is, 
their monodromy along $\partial M$ are upper triangular with all diagonal entries equal to $1$.
We see that one of the local systems is obtained 
from the $\PGL_2$-local system associated with the complete hyperbolic structure via a principal embedding, 
whereas the other is not.

We will compute the torsion of these two local systems. The method is as follows.
First, let $P \in \Loc_{\PGSp_4, M, S_0}(K)$ be a given marked local system. We lift it to a marked local system $\widetilde{P} \in \Loc_{\PGSp_4, M, S}(K)$. This is done by choosing, for each center of cell, a path to one of the $0$-cells. See the proof of \cref{lemma:monodromy of Loc G M S} that contains this construction.
With this choice, 

\begin{itemize}
    \item The differentials in the chain complex $C_\bullet(M;\mathcal{L})$ with $\mathcal{L} \coloneq P \times^{\PGSp_4} \mathfrak{pgsp}_4$ can be computed using formula \eqref{eq:twisted_differential}.
    \item The linear ordering of $\mathsf{cell}_i$ and the ordered basis of $\mathfrak{pgsp}_4$ determines an ordered basis $c_i$ of $C_i(M;\mathcal{L})$ as in \eqref{eq:geom_basis}. We may find bases of $B_i(M;\mathcal{L})$ from the data of differentials. 
    \item We may also find a basis $(v_1,v_2)$ of $H^0(\partial M; \mathcal{L}|_{\partial M}) \cong (\mathfrak{pgsp}_4 \otimes_\bQ K)^{\pi_1 \partial M}$, which determine bases $h_1$ and $h_2$ of
    $H_1 (M; \mathcal{L})$
    and $H_2 (M; \mathcal{L})$,
    respectively,
    in  \eqref{eq:h1 h2 bases}.
\end{itemize}
Finally, we compute the torsion via the explicit formula in \cref{theorem:torsion explicit}. We note that since $\PGSp_4$ has even rank, the torsion is independent of the homological orientation (see \cref{lemma:adjoint torsion independence on choices}).
We will not record the full details of these computations here, but we carried them out using the computer algebra system SageMath \cite{sagemath}. The code is available at \url{https://github.com/yuma-mizuno/pgsp-torsion}.

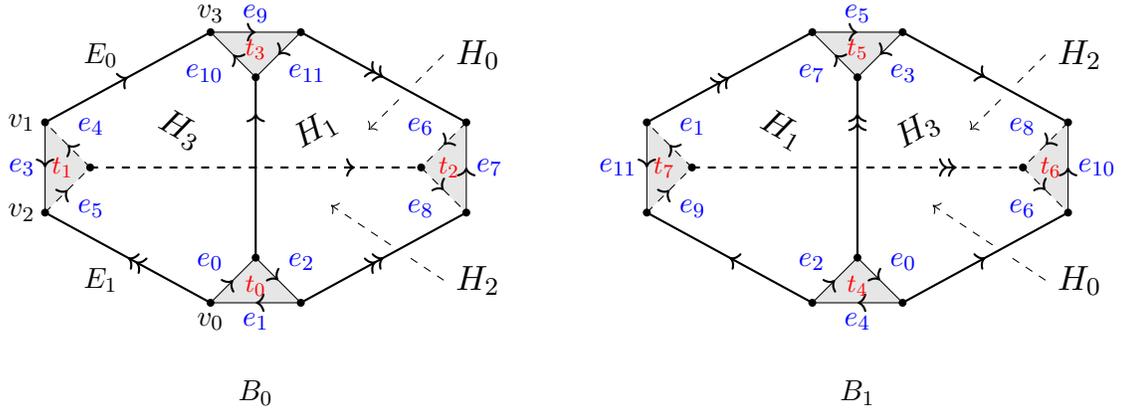
\begin{figure}[t]
  \centering
\begin{tikzpicture}
\coordinate (02) at (0,-1.2);
\coordinate (01) at (-0.6,-1.8);
\coordinate (03) at (0.6,-1.8);
\fill[gray!20] (01) -- (02) -- (03) --cycle;
\draw[->-] (01) --node[blue,midway,above left]{$e_0$} (02);
\draw[->-] (02) --node[blue,midway,above right]{$e_2$} (03);
\draw[->-] (03) --node[blue,midway,below]{$e_1$} (01);
\coordinate (20) at (0,1.2);
\coordinate (21) at (-0.6,1.8);
\coordinate (23) at (0.6,1.8);
\fill[gray!20] (21) -- (20) -- (23) --cycle;
\draw[->-] (21) --node[blue,midway,above]{$e_9$} (23);
\draw[->-] (23) --node[blue,midway,below right]{$e_{11}$} (20);
\draw[->-] (20) --node[blue,midway,below left]{$e_{10}$} (21);
\coordinate (13) at (-2.2,0);
\coordinate (12) at (-2.8,0.6);
\coordinate (10) at (-2.8,-0.6);
\fill[gray!20] (12) -- (10) -- (13) --cycle;
\draw[->-] (12) --node[blue,midway,left]{$e_3$} (10);
\draw[->-,dashed] (10) --node[blue,midway,below right]{$e_5$} (13);
\draw[->-,dashed] (13) --node[blue,midway,above right]{$e_4$} (12);
\coordinate (31) at (2.2,0);
\coordinate (32) at (2.8,0.6);
\coordinate (30) at (2.8,-0.6);
\fill[gray!20] (32) -- (30) -- (31) --cycle;
\draw[->-,dashed] (32) --node[blue,midway,above left]{$e_6$} (31);
\draw[->-,dashed] (31) --node[blue,midway,below left]{$e_8$} (30);
\draw[->-] (30) --node[blue,midway,right]{$e_7$} (32);
\draw[thick,dashed,->-={0.8}{}] (13) -- (31);
\draw[thick,->-={0.8}{}] (02) -- (20); 
\draw[thick,->>-] (01) -- (10); 
\draw[thick,->>-] (03) -- (30); 
\draw[thick,->-] (12) -- (21); 
\draw[thick,->>-] (23) -- (32); 
\node[above left] at ($(12)!0.5!(21)$) {$E_0$};
\node[below left] at ($(01)!0.5!(10)$) {$E_1$};
\foreach \i/\j in {
  0/1,0/2,0/3,
  1/0,1/2,1/3,
  2/0,2/1,2/3,
  3/0,3/1,3/2}
\fill(\i\j) circle(1.5pt);
\node[below] at (01) {$v_0$};
\node[left] at (10) {$v_2$};
\node[left] at (12) {$v_1$};
\node[above] at (21) {$v_3$};
\node[red,scale=0.9] at (0,-1.57){$t_0$};
\node[red,scale=0.9] at (0,1.57){$t_3$};
\node[red,scale=0.9] at (-2.57,0){$t_1$};
\node[red,scale=0.9] at (2.57,0){$t_2$};
\node[scale=1.2,rotate=-30] at (-1,0.5) {$H_3$};
\node[scale=1.2,rotate=30] at (0.8,0.5) {$H_1$};
\begin{pgfonlayer}{bg}
\draw[->,dashed] (2.5,-1.5) node[right,scale=1.2] {$H_2$} -- (1,-0.5);
\draw[white,line width=0.2cm] (03) -- (30);
\draw[->,dashed] (2.5,1.5) node[right,scale=1.2] {$H_0$} -- (1.5,0.5);
\draw[white,line width=0.2cm] (32) -- (23);
\end{pgfonlayer}
\node at (0,-3) {$B_0$};

\begin{scope}[xshift=8cm]
\coordinate (02) at (0,-1.2);
\coordinate (01) at (-0.6,-1.8);
\coordinate (03) at (0.6,-1.8);
\fill[gray!20] (01) -- (02) -- (03) --cycle;
\draw[->-] (01) --node[blue,midway,above left]{$e_2$} (02);
\draw[->-] (02) --node[blue,midway,above right]{$e_0$} (03);
\draw[->-] (03) --node[blue,midway,below]{$e_4$} (01);
\coordinate (20) at (0,1.2);
\coordinate (21) at (-0.6,1.8);
\coordinate (23) at (0.6,1.8);
\fill[gray!20] (21) -- (20) -- (23) --cycle;
\draw[->-] (21) --node[blue,midway,above]{$e_5$} (23);
\draw[->-] (23) --node[blue,midway,below right]{$e_3$} (20);
\draw[->-] (20) --node[blue,midway,below left]{$e_7$} (21);
\coordinate (13) at (-2.2,0);
\coordinate (12) at (-2.8,0.6);
\coordinate (10) at (-2.8,-0.6);
\fill[gray!20] (12) -- (10) -- (13) --cycle;
\draw[->-] (12) --node[blue,midway,left]{$e_{11}$} (10);
\draw[->-,dashed] (10) --node[blue,midway,below right]{$e_9$} (13);
\draw[->-,dashed] (13) --node[blue,midway,above right]{$e_1$} (12);
\coordinate (31) at (2.2,0);
\coordinate (32) at (2.8,0.6);
\coordinate (30) at (2.8,-0.6);
\fill[gray!20] (32) -- (30) -- (31) --cycle;
\draw[->-,dashed] (32) --node[blue,midway,above left]{$e_8$} (31);
\draw[->-,dashed] (31) --node[blue,midway,below left]{$e_6$} (30);
\draw[->-] (30) --node[blue,midway,right]{$e_{10}$} (32);
\draw[thick,dashed,->>-={0.8}{}] (13) -- (31);
\draw[thick,->>-={0.8}{}] (02) -- (20); 
\draw[thick,->-] (01) -- (10); 
\draw[thick,->-] (03) -- (30); 
\draw[thick,->>-] (12) -- (21); 
\draw[thick,->-] (23) -- (32); 
\foreach \i/\j in {
  0/1,0/2,0/3,
  1/0,1/2,1/3,
  2/0,2/1,2/3,
  3/0,3/1,3/2}
\fill(\i\j) circle(1.5pt);
\node[red,scale=0.9] at (0,-1.57){$t_4$};
\node[red,scale=0.9] at (0,1.57){$t_5$};
\node[red,scale=0.9] at (-2.57,0){$t_7$};
\node[red,scale=0.9] at (2.57,0){$t_6$};
\node[scale=1.2,rotate=-30] at (-1,0.5) {$H_1$};
\node[scale=1.2,rotate=30] at (0.8,0.5) {$H_3$};
\begin{pgfonlayer}{bg}
\draw[->,dashed] (2.5,-1.5) node[right,scale=1.2] {$H_0$} -- (1,-0.5);
\draw[white,line width=0.2cm] (03) -- (30);
\draw[->,dashed] (2.5,1.5) node[right,scale=1.2] {$H_2$} -- (1.5,0.5);
\draw[white,line width=0.2cm] (32) -- (23);
\end{pgfonlayer}
\node at (0,-3) {$B_1$};
\end{scope}
\end{tikzpicture}
  \caption{The CW structure of the figure-eight knot exterior obtained by truncating 
  vertices in the ideal triangulation in \cref{fig:figure_eight_triangulation}.}
  \label{fig:figure eight cells}
\end{figure}

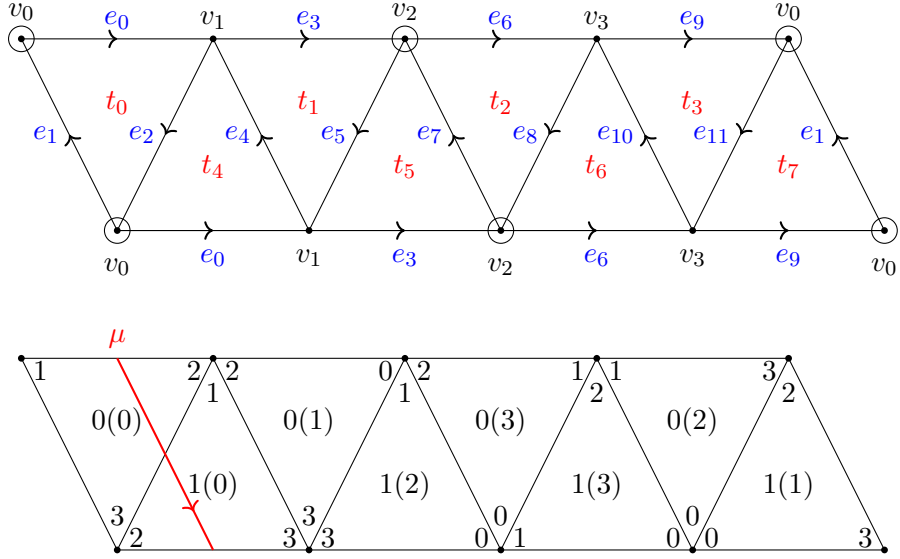
\begin{figure}[t]
  \centering
\begin{tikzpicture}[scale=0.85]
\foreach \i in {0,1,2,3,4} 
\fill (3*\i,0) circle(1.5pt) coordinate(0\i); 
\foreach \i in {0,1,2,3,4} 
\fill (3*\i-1.5,3) circle(1.5pt) coordinate(1\i); 
\foreach \i in {0,2,4} {
    \draw (0\i) circle(0.2cm);
    \draw (1\i) circle(0.2cm);
}
\node[above=0.3em] at (10) {$v_0$};
\node[above] at (11) {$v_1$};
\node[above=0.3em] at (12) {$v_2$};
\node[above] at (13) {$v_3$};
\node[above=0.3em] at (14) {$v_0$};
\node[below=0.6em] at (00) {$v_0$};
\node[below=0.3em] at (01) {$v_1$};
\node[below=0.6em] at (02) {$v_2$};
\node[below=0.3em] at (03) {$v_3$};
\node[below=0.6em] at (04) {$v_0$};

\draw(00) -- (04);
\draw[->-] (00) --node[blue,midway,left]{$e_1$} (10);
\draw[->-] (01) --node[blue,midway,left]{$e_4$} (11);
\draw[->-] (02) --node[blue,midway,left]{$e_7$} (12);
\draw[->-] (03) --node[blue,midway,left]{$e_{10}$} (13);
\draw[->-] (04) --node[blue,midway,left]{$e_1$} (14);
\draw[->-] (11) --node[blue,midway,left]{$e_2$} (00);
\draw[->-] (12) --node[blue,midway,left]{$e_5$} (01);
\draw[->-] (13) --node[blue,midway,left]{$e_8$} (02);
\draw[->-] (14) --node[blue,midway,left]{$e_{11}$} (03);
\draw[->-] (10) --node[blue,midway,above]{$e_0$} (11);
\draw[->-] (11) --node[blue,midway,above]{$e_3$} (12);
\draw[->-] (12) --node[blue,midway,above]{$e_6$} (13);
\draw[->-] (13) --node[blue,midway,above]{$e_9$} (14);
\draw[->-] (00) --node[blue,midway,below, yshift=-0.3em]{$e_0$} (01);
\draw[->-] (01) --node[blue,midway,below, yshift=-0.3em]{$e_3$} (02);
\draw[->-] (02) --node[blue,midway,below, yshift=-0.3em]{$e_6$} (03);
\draw[->-] (03) --node[blue,midway,below, yshift=-0.3em]{$e_9$} (04);

\node[red] at (0,2){$t_0$};
\node[red] at (3,2){$t_1$};
\node[red] at (6,2){$t_2$};
\node[red] at (9,2){$t_3$};
\node[red] at (1.5,1){$t_4$};
\node[red] at (4.5,1){$t_5$};
\node[red] at (7.5,1){$t_6$};
\node[red] at (10.5,1){$t_7$};


\begin{scope}[yshift=-5cm]
\foreach \i in {0,1,2,3,4} 
\fill (3*\i,0) circle(1.5pt) coordinate(0\i); 
\foreach \i in {0,1,2,3,4} 
\fill (3*\i-1.5,3) circle(1.5pt) coordinate(1\i); 
%

\draw(00) -- (04);
\draw(10) -- (14);
\draw(04) -- (14) -- (03) -- (13) -- (02) -- (12) -- (01) -- (11) -- (00) -- (10);

\node at (0,2){$0(0)$};
\node at (3,2){$0(1)$};
\node at (6,2){$0(3)$};
\node at (9,2){$0(2)$};
\node at (1.5,1){$1(0)$};
\node at (4.5,1){$1(2)$};
\node at (7.5,1){$1(3)$};
\node at (10.5,1){$1(1)$};

\node[above=0.5em] at (00) {$3$};
\node[above=0.5em] at (01) {$3$};
\node[above=0.5em] at (02) {$0$};
\node[above=0.5em] at (03) {$0$};
\node[below=0.5em] at (11) {$1$};
\node[below=0.5em] at (12) {$1$};
\node[below=0.5em] at (13) {$2$};
\node[below=0.5em] at (14) {$2$};
\node at ($(00)+(0.3,0.2)$) {$2$};
\node at ($(01)+(0.3,0.2)$) {$3$};
\node at ($(02)+(0.3,0.2)$) {$1$};
\node at ($(03)+(0.3,0.2)$) {$0$};
\node at ($(01)+(-0.3,0.2)$) {$3$};
\node at ($(02)+(-0.3,0.2)$) {$0$};
\node at ($(03)+(-0.3,0.2)$) {$0$};
\node at ($(04)+(-0.3,0.2)$) {$3$};
\node at ($(10)+(0.3,-0.2)$) {$1$};
\node at ($(11)+(0.3,-0.2)$) {$2$};
\node at ($(12)+(0.3,-0.2)$) {$2$};
\node at ($(13)+(0.3,-0.2)$) {$1$};
\node at ($(11)+(-0.3,-0.2)$) {$2$};
\node at ($(12)+(-0.3,-0.2)$) {$0$};
\node at ($(13)+(-0.3,-0.2)$) {$1$};
\node at ($(14)+(-0.3,-0.2)$) {$3$};
\draw[red,thick,->-={0.8}{}] ($(10)!0.5!(11)$) node[above]{$\mu$} -- ($(00)!0.5!(01)$);
\end{scope}
\end{tikzpicture}
  \caption{The CW structure of the boundary torus. In the upper figure, 
  the endpoints of $E_0$ are shown as solid points, and the endpoints of $E_1$ are shown as circled points. 
  The lower figure indicates where each cell lies in the tetrahedra. 
  The label $i(j)$ on a triangle indicates that the triangle is obtained by truncating the $j$-th vertex of the $i$-th tetrahedron. If a corner of the triangle $i(j)$ is labeled by $k$, this means that the vertex corresponding to that corner is an endpoint of the edge $ik$ in the $i$-th tetrahedron.}
  \label{fig:figure-eight boundary torus}
\end{figure}

\subsubsection{Representation from the complete hyperbolic structure}
Let $K$ be a number field defined by 
\begin{align*}
  K \coloneqq \mathbb{Q} [\omega] / (\omega^2 - \omega + 1).
\end{align*}
Note that $K \cong \mathbb{Q}(\sqrt{-3})$ by $\omega \mapsto (1 + \sqrt{-3})/2$.
The following monodromies along the $1$-cells give a representative of 
$\mathsf{geom}$, which is realized as an element of $\Loc_{\PGL_2, M, S_0} (K)$:
\begin{align*}
  \mon_{\mathsf{geom}} (E_0) =
  \begin{pmatrix}
    0 & -\omega^{-1}  \\
    \omega & 0 
  \end{pmatrix}, \quad
  \mon_{\mathsf{geom}} (E_1) =
  \begin{pmatrix}
    0 & 1 \\
    -1 & 0 
  \end{pmatrix}
\end{align*}

\begin{align*}
  \mon_{\mathsf{geom}} (e_{3i}) =
  \begin{pmatrix}
    1 & 1 \\
    0 & 1 
  \end{pmatrix}, \quad
  \mon_{\mathsf{geom}} (e_{3i+1}) =
  \begin{pmatrix}
    1 & -\omega \\
    0 & 1
  \end{pmatrix}, \quad
  \mon_{\mathsf{geom}} (e_{3i+2}) =
  \begin{pmatrix}
    1 & -\omega^{-1} \\
    0 & 1
  \end{pmatrix},
\end{align*}
for $i = 0, 1, 2, 3$.

\begin{lemma}
We have a unique principal embedding $\iota: (\PGL_2)_K \to (\PGSp_4)_K$ whose derivative $(\mathfrak{sl}_2)_K \to (\mathfrak{sp}_4)_K$ is given by
\begin{align*}
    &\mtx{0 & 1 \\ 0 & 0} \mapsto \mtx{
    0 & \frac{9}{4}\omega & 0 & 0 \\ 
    0 & 0 & -\frac{16}{9}\omega^{-1} & 0 \\
    0 & 0 & 0 & \frac{9}{4}\omega \\
    0 & 0 & 0 & 0}, \quad 
    \mtx{0 & 0 \\ 1 & 0} \mapsto \mtx{
    0 & 0 & 0 & 0 \\ 
    \frac{4}{3}\omega^{-1} & 0 & 0 & 0 \\
    0 & -\frac{9}{4}\omega & 0 & 0 \\
    0 & 0 & \frac{4}{3}\omega^{-1} & 0}, \\
    &\mtx{1 & 0 \\ 0 & -1} \mapsto \mtx{3 & 0 & 0 & 0 \\
    0 & 1 & 0 & 0 \\
    0 & 0 & -1 & 0\\
    0 & 0 & 0 & -3}.
\end{align*}
\end{lemma}

\begin{proof}
It is easy to check the assignment defines a Lie algebra homomorphism. Therefore is is integrated into a unique morphism between the simply connected groups, which descends to adjoint groups since it preserves the centers.
\end{proof}

Then we obtain the following marked $\PGSp_4$-local system $\iota(\mathsf{geom})$
that is also obtained from the solution of the 
Ptolemy relations in \cite[(9.4)]{Zickert}:
\begin{align*}
  \setlength{\arraycolsep}{3pt}
  \mon_{\iota(\mathsf{geom})} (E_0) =
  \begin{pmatrix}
    0 & 0 & 0 & -\frac{3}{2} \omega \\
    0 & 0 & \frac{8}{9} \omega & 0 \\
    0 & -\frac{9}{8} \omega^{-1} & 0 & 0 \\
    \frac{2}{3} \omega^{-1} & 0 & 0 & 0
  \end{pmatrix}, \
  \mon_{\iota(\mathsf{geom})} (E_1) =
  \begin{pmatrix}
    0 & 0 & 0 & -\frac{3}{2} \omega \\
    0 & 0 & \frac{8}{9} \omega^{-1} & 0 \\
    0 & -\frac{9}{8} \omega & 0 & 0 \\
    \frac{2}{3} \omega^{-1} & 0 & 0 & 0
  \end{pmatrix},
\end{align*}
\begin{align*}
  \setlength{\arraycolsep}{3pt}
  \mon_{\iota(\mathsf{geom})} (e_{3i}) =
  \begin{pmatrix}
    1 & \frac{9}{4} \omega & -2 & -\frac{3}{2} \omega \\
    0 & 1 & -\frac{16}{9} \omega^{-1} & -2 \\
    0 & 0 & 1 & \frac{9}{4} \omega \\
    0 & 0 & 0 & 1
  \end{pmatrix}, \
  \mon_{\iota(\mathsf{geom})} (e_{3i+1}) =
  \begin{pmatrix}
    1 & \frac{9}{4} \omega^{-1} & 2 \omega^{-1} & -\frac{3}{2} \omega \\
    0 & 1 & \frac{16}{9} & 2 \omega^{-1} \\
    0 & 0 & 1 & \frac{9}{4} \omega^{-1} \\
    0 & 0 & 0 & 1
  \end{pmatrix},
\end{align*}
\begin{align*}
  \setlength{\arraycolsep}{3pt}
  \mon_{\iota(\mathsf{geom})} (e_{3i+2}) =
  \begin{pmatrix}
    1 & -\frac{9}{4} & 2 \omega & -\frac{3}{2} \omega \\
    0 & 1 & \frac{16}{9} \omega & 2 \omega \\
    0 & 0 & 1 & -\frac{9}{4} \\
    0 & 0 & 0 & 1
  \end{pmatrix}, \quad\text{for $i = 0, 1, 2, 3$}
\end{align*}

We see that $\dim_K (\fpgsp_4 \otimes_{\mathbb{Q}} K)^{\pi_1 \partial M} = 2$, and this $K$-vector space is generated by
\begin{align*}
  v_1 =
  \begin{pmatrix}
    0 & 1 & 0 & 0 \\
    0 & 0 & \frac{64}{81} \omega & 0 \\
    0 & 0 & 0 & 1 \\
    0 & 0 & 0 & 0
  \end{pmatrix}, \quad
  v_2 =
  \begin{pmatrix}
    0 & 0 & 0 & 1 \\
    0 & 0 & 0 & 0 \\
    0 & 0 & 0 & 0 \\
    0 & 0 & 0 & 0
  \end{pmatrix}.
\end{align*}
The adjoint $\PGSp_4$-Reidemeister torsion at $\iota(\mathsf{geom})$ is now 
computed as
\begin{align*}
  \torsion_{\PGSp_4, M, \mu}^{\Ad} (\iota(\mathsf{geom})) = 360.
\end{align*}

\subsubsection{Representation not from the complete hyperbolic structure}
Let $K$ be a number field defined by 
\begin{equation}
  K \coloneqq \mathbb{Q} [\omega] / (\omega^6 - \omega^5 + 3\omega^4 - 5\omega^3 + 8\omega^2 - 6\omega + 8).
\end{equation}
From the solution of the Ptolemy relations in \cite[(9.5)]{Zickert},
we get the following marked local system $P \in \Loc_{\PGSp_4, M, S_0} (K)$
whose monodromies along the $1$-cells are given as follows:
\begin{align*}
  \mon_P(E_0) =
  \begin{pmatrix}
  0 & 0 & 0 & -\alpha_0^{-1} \\
  0 & 0 & \beta_0^{-1} & 0 \\
  0 & -\beta_0 & 0 & 0 \\
  \alpha_0 & 0 & 0 & 0
  \end{pmatrix}, \quad
  \mon_P(E_1) =
  \begin{pmatrix}
  0 & 0 & 0 & -\alpha_1^{-1} \\
  0 & 0 & \beta_1^{-1} & 0 \\
  0 & -\beta_1 & 0 & 0 \\
  \alpha_1 & 0 & 0 & 0
  \end{pmatrix}
\end{align*}
where
\begin{align*}
  \alpha_0 &= \frac{3}{32} \omega^{5} - \frac{3}{16} \omega^{4} + \frac{7}{32} \omega^{3} - \frac{11}{16} \omega^{2} + \frac{11}{16} \omega - \frac{1}{4} \\
  \beta_0 &= \frac{1}{32} \omega^{5} + \frac{1}{4} \omega^{4} + \frac{7}{32} \omega^{3} + \frac{7}{16} \omega^{2} + \frac{3}{16} \omega - \frac{5}{4} \\
  \alpha_1 &= -\frac{1}{64} \omega^{5} - \frac{3}{32} \omega^{4} + \frac{3}{64} \omega^{3} - \frac{3}{32} \omega^{2} + \frac{15}{32} \omega - \frac{9}{8} \\
  \beta_1 &= \frac{11}{32} \omega^{5} - \frac{1}{2} \omega^{4} + \frac{29}{32} \omega^{3} - \frac{23}{16} \omega^{2} + \frac{25}{16} \omega - \frac{9}{4}
\end{align*}
and
\begin{align*}
&\mon_P(e_i) = \mon_P(e_{i+6}) =
\begin{pmatrix}
1 & a_i & b_i & c_i \\
0 & 1 & d_i & e_i \\
0 & 0 & 1 & f_i \\
0 & 0 & 0 & 1
\end{pmatrix}
\end{align*}
for $i = 0,1,2,3,4,5$, where
\begin{align*}
  a_0 &= f_0 = a_3 = f_3 = \frac{3}{8} \omega^{5} - \frac{3}{8} \omega^{4} + \frac{3}{4} \omega^{3} - 2 \omega^{2} + \omega - \frac{5}{2}\\
  b_0 &= e_3 = \Bigl(-\frac{1}{4} \omega^{5} + \frac{1}{4} \omega^{4} - \frac{1}{2} \omega^{3} + \omega^{2} - \omega - 1\Bigr)^{-1}\\
  c_0 &= c_3 = -\frac{1}{8} \omega^{5} - \frac{1}{8} \omega^{3} + \frac{1}{2} \omega^{2} + \frac{1}{4} \omega + \frac{3}{2}\\
  d_0 &= d_3 = \Bigl(-\frac{5}{64} \omega^{5} - \frac{9}{32} \omega^{4} - \frac{1}{64} \omega^{3} + \frac{21}{32} \omega^{2} + \frac{3}{32} \omega - \frac{5}{8}\Bigr)^{-1}\\
  e_0 &= b_3 = \Bigl(\frac{1}{4} \omega^{5} - \frac{1}{4} \omega^{4} + \frac{1}{2} \omega^{3} - \omega^{2} + \omega\Bigr)^{-1}\\
  a_1 &= f_1 = a_4 = f_4 = -\frac{1}{16} \omega^{5} - \frac{7}{16} \omega^{3} + \frac{5}{8} \omega^{2} + \frac{5}{8} \omega + \frac{3}{2}\\
  b_1 &= e_1 = b_4 = e_4 = \Bigl(\frac{1}{8} \omega^{5} - \frac{1}{4} \omega^{4} + \frac{1}{8} \omega^{3} - \frac{1}{4} \omega^{2} + \frac{1}{4} \omega - 1\Bigr)^{-1}\\
  c_1 &= \frac{1}{4} \omega^{5} + \frac{3}{8} \omega^{4} + \frac{3}{4} \omega^{3} + \frac{5}{8} \omega^{2} + \frac{1}{4} \omega + \frac{1}{4}\\
  d_1 &= d_4 = \Bigl(-\frac{5}{64} \omega^{5} - \frac{1}{16} \omega^{4} + \frac{9}{64} \omega^{3} - \frac{19}{32} \omega^{2} - \frac{3}{32} \omega - \frac{11}{8}\Bigr)^{-1}\\
  a_2 &= f_2 = a_5 = f_5 = -\frac{5}{16} \omega^{5} + \frac{3}{8} \omega^{4} - \frac{5}{16} \omega^{3} + \frac{11}{8} \omega^{2} - \frac{13}{8} \omega + 1\\
  b_2 &= e_5 = \Bigl(\frac{1}{8} \omega^{5} - \frac{1}{2} \omega^{4} + \frac{3}{8} \omega^{3} - \frac{5}{4} \omega^{2} + \frac{3}{4} \omega - 1\Bigr)^{-1}\\
  c_2 &= c_5 = -\frac{1}{4} \omega^{4} - \frac{1}{4} \omega^{3} - \frac{1}{2} \omega^{2} + \frac{1}{2} \omega - 1\\
  d_2 &= d_5 = \Bigl(\frac{3}{8} \omega^{5} + \frac{7}{32} \omega^{4} + \frac{9}{32} \omega^{3} - \frac{1}{2} \omega^{2} + \frac{1}{16} \omega - \frac{3}{4}\Bigr)^{-1}\\
  e_2 &= b_5 = \Bigl(\frac{1}{8} \omega^{4} + \frac{3}{8} \omega^{2} + \frac{1}{4} \omega - \frac{1}{4}\Bigr)^{-1}\\
  c_4 &= -\frac{1}{8} \omega^{5} + \frac{1}{8} \omega^{3} + \frac{1}{4} \omega^{2} - \frac{3}{4} \omega - 1\\
\end{align*}

We see that $\dim_K (\fpgsp_4 \otimes_{\mathbb{Q}} K)^{\pi_1 \partial M} = 2$, and this $K$-vector space is generated by
\begin{align*}
  v_1 =
  \begin{pmatrix}
    0 & 1 & 0 & 0 \\
    0 & 0 & t & 0 \\
    0 & 0 & 0 & 1 \\
    0 & 0 & 0 & 0
  \end{pmatrix}, \quad
  v_2 =
  \begin{pmatrix}
    0 & 0 & 0 & 1 \\
    0 & 0 & 0 & 0 \\
    0 & 0 & 0 & 0 \\
    0 & 0 & 0 & 0
  \end{pmatrix}
\end{align*}
where $t = (\frac{9}{128} \omega^{5} + \frac{23}{128} \omega^{4} - \frac{21}{32} \omega^{3} - \frac{29}{64} \omega^{2} - \frac{9}{4} \omega - \frac{3}{4})^{-1}$.
The adjoint $\PGSp_4$-Reidemeister torsion at $P$ is now computed as
\begin{align*}
  \torsion_{\PGSp_4, M, \mu}^{\Ad} (P) = \frac{85}{16} \omega^{5} - \frac{33}{8} \omega^{4} + \frac{217}{16} \omega^{3} - \frac{99}{8} \omega^{2} + \frac{321}{8} \omega - 11.
\end{align*}
The complex embeddings of the torsion are numerically given by
\begin{align*}
  \torsion_{\PGSp_4, M, \mu}^{\Ad} (P) =
  \begin{cases}
  - 3.459966243820608\\
  1.104983121910304 + 38.11233948347826 \sqrt{-1}\\
  1.104983121910304 - 38.11233948347826 \sqrt{-1}.
  \end{cases}
\end{align*}
This $\PGSp_4$-local system does not arise from any $\PGL_2$-local system via a principal embedding. Indeed, one can see that the $K$-vector space
\begin{align}\label{eq:commutant in pgsp}
    \{ T \in \End_K (\fpgsp_4 \otimes_{\mathbb{Q}} K) \mid 
    \forall \gamma \in \Gamma,\ T \Ad(\mon_P (\gamma)) = \Ad(\mon_P (\gamma)) T\}
\end{align}
consists only of scalar multiplications,
where $\Gamma = \{ E_0, E_1, e_0, \dots, e_{11} \}$.
This is not possible for a local system arising from a $\PGL_2$-local system, since the space \eqref{eq:commutant in pgsp}
for such a local system has a dimension at least $2$ by the irreducible decomposition $\fpgsp_4 \cong V_3 \oplus V_7$ as $\PGL_2$-modules (\cref{prop:adjoint_decomp}).

\printbibliography

\end{document}